\documentclass[12pt,reqno,usenames]{amsart}
\usepackage{amsopn}
\usepackage[sort, numbers]{natbib}

\usepackage{amsmath}
\usepackage{amssymb}
\usepackage{mathabx}
\usepackage{bbm}
\usepackage{url}
\usepackage{amsfonts}
\usepackage[colorlinks,citecolor=blue,linkcolor=blue,urlcolor=blue]{hyperref}
\usepackage{enumitem}
\setlist[enumerate]{leftmargin=.5in}
\setlist[itemize]{leftmargin=.5in}
\usepackage{mathtools}
\usepackage{parskip}
\usepackage{cleveref}
\usepackage{mathrsfs}
\usepackage{framed}
\usepackage{color}
\usepackage{tikz-cd}
\usepackage{comment}
\usepackage{subcaption}
\usepackage{dsfont}
\usepackage{algorithm}
\usepackage{algorithmic}
\usepackage{crossreftools}
\usepackage[margin=1in]{geometry}

\newtheorem{theorem}{Theorem}[section]

\newtheorem{lemma}[theorem]{Lemma}

\newtheorem{proposition}[theorem]{Proposition}
\newtheorem*{proposition*}{Proposition}
\newtheorem{corollary}[theorem]{Corollary}

\theoremstyle{definition}
\newtheorem{definition}[theorem]{Definition}

\newcounter{claimcount}
\newenvironment{claim}{\refstepcounter{claimcount}\emph{Claim \arabic{claimcount}:}}{\vspace{1pt}}
\AtBeginEnvironment{proof}{\setcounter{claimcount}{0}}
\newtheorem{example}[theorem]{Example}

\newtheorem{remark}[theorem]{Remark}

\allowdisplaybreaks[1]

\newcommand{\eps}{\varepsilon}

\newlist{steps}{enumerate}{1}
\setlist[steps, 1]{label = Step \arabic*:}

\newcommand{\R}{\mathbb{R}}
\newcommand{\Rp}{\R_{\geq 0}}
\newcommand{\norm}[1]{\left\lVert#1\right\rVert}
\newcommand{\N}{\mathbb{N}}

\newcommand{\X}{\ensuremath{\mathcal{X}}}
\newcommand{\muX}{\ensuremath{\mu_{X}}}
\newcommand{\dX}{\ensuremath{d_{X}}}
\newcommand{\mmspaceX}{\ensuremath{\left(X,\dX,\muX \right) }}

\newcommand{\Y}{\ensuremath{\mathcal{Y}}}
\newcommand{\muY}{\ensuremath{\mu_{Y}}}
\newcommand{\dY}{\ensuremath{d_{Y}}}
\newcommand{\mmspaceY}{\ensuremath{\left(Y,\dY,\muY \right) }}

\newcommand{\isocalss}{\mathcal{M}^w}

\newcommand{\uX}{\ensuremath{u_{X}}}
\newcommand{\ummspaceX}{\ensuremath{\left(X,\uX,\muX \right) }}

\newcommand{\uY}{\ensuremath{u_{Y}}}
\newcommand{\ummspaceY}{\ensuremath{\left(Y,\uY,\muY \right) }}

\newcommand{\Z}{\ensuremath{\mathcal{Z}}}

\newcommand{\uZ}{\ensuremath{u_{Z}}}

\newcommand{\ultradiscol}{\ensuremath{\mathcal{U}_\mathrm{dis}^w}}

\newcommand{\spec}[1]{\mathrm{spec}\left(#1\right)}

\newcommand{\supp}[1]{\ensuremath{\mathrm{supp}\left(#1\right) }}

\newcommand{\dis}{\mathrm{dis}}
\newcommand{\disu}{\dis^\mathrm{ult}}

\newcommand{\pseudoWasser}[1]{d_{\mathrm{W},#1}}

\newcommand{\WasserRpS}[1]{d^{(S,\Lambda_\infty)}_{\mathrm{W},#1}}

\newcommand{\dgh}{d_\mathrm{GH}}

\newcommand{\ugh}{u_\mathrm{GH}}

\newcommand{\dgw}[1]{d_{\mathrm{GW},#1}}
\newcommand{\ugw}[1]{u_{\mathrm{GW},#1}}
\newcommand{\usturm}[1]{u_{\mathrm{GW},#1}^\mathrm{sturm}}
\newcommand{\dsturm}[1]{d_{\mathrm{GW},#1}^\mathrm{sturm}}
\newcommand{\colijn}{d_{\mathrm{CP},2}}

\newcommand{\groupUSone}{\mathcal{G}_1}
\newcommand{\groupUStwo}{\mathcal{G}_2}
\newcommand{\grouptrop}{\mathcal{G}_3}

\newcommand{\uFLB}[1]{\mathbf{FLB}_{#1}^\mathrm{ult}}
\newcommand{\dFLB}[1]{\mathbf{FLB}_{#1}}
\newcommand{\uSLB}[1]{\mathbf{SLB}_{#1}^\mathrm{ult}}
\newcommand{\dSLB}[1]{\mathbf{SLB}_{#1}}
\newcommand{\uTLB}[1]{\mathbf{TLB}_{#1}^\mathrm{ult}}
\newcommand{\dTLB}[1]{\mathbf{TLB}_{#1}}

\def\Inf{\operatornamewithlimits{inf\vphantom{p}}}
\newcommand{\diam}[1]{\mathrm{diam}\left(#1\right) }

\newcommand{\lc}{\left(}
\newcommand{\rc}{\right)}

\begin{document}
	\title{The ultrametric Gromov-Wasserstein distance}

	\title{The ultrametric Gromov-Wasserstein distance}
	
 	\author{Facundo M\'emoli}
 	\address{Department of Mathematics and Department of Computer Science and Engineering,
 		The Ohio State University}
 	\email{memoli@math.osu.edu}
	
 	\author{Axel Munk}
 	\address{Institute for Mathematical Stochastics, University of G\"ottingen}
 	\email{munk@math.uni-goettingen.de}
	
 	\author{Zhengchao Wan}
 	\address{Department of Mathematics,
 		The Ohio State University}
 	\email{wan.252@osu.edu}
	
 	\author{Christoph Weitkamp}
 	\address{Institute for Mathematical Stochastics, University of G\"ottingen}
 	\email{cweitka@mathematik.uni-goettingen.de}

	\maketitle
	\begin{abstract}
	    In this paper, we investigate compact ultrametric measure spaces which form a subset $\mathcal{U}^w$ of the collection of all metric measure spaces $\mathcal{M}^w$. In analogy with the notion of the ultrametric Gromov-Hausdorff distance on the collection of ultrametric spaces $\mathcal{U}$, we define ultrametric versions of two metrics on $\mathcal{U}^w$, namely of Sturm's Gromov-Wasserstein distance of order $p$ and of the Gromov-Wasserstein distance of order $p$. We study the basic topological and geometric properties of these distances as well as their relation and derive for $p=\infty$ a polynomial time algorithm for their calculation. Further, several lower bounds for both distances are derived and some of our results are generalized to the case of finite ultra-dissimilarity spaces. Finally, we study the relation between the Gromov-Wasserstein distance and its ultrametric version (as well as the relation between the corresponding lower bounds) in simulations and apply our findings for phylogenetic tree shape comparisons.
	\end{abstract}

	\section{Introduction}

	Over the last decade the acquisition of ever more complex data, structures and shapes has increased dramatically. Consequently, the need to develop meaningful methods for comparing general objects has become more and more apparent. In numerous applications, e.g. in molecular biology \citep{holm1993protein,kufareva2011method,brown2016fast}, computer vision \citep{lowe2001local,jain20003d} and electrical engineering \citep{papazov2012rigid,kuo20143d}, it is important to distinguish between different objects in a pose invariant manner: two instances of the a given object in \emph{different} spatial orientations are deemed to be equal. Furthermore, also the comparisons of graphs, trees, ultrametric spaces and networks, where mainly the underlying connectivity structure matters, have grown in importance \citep{chen2011algebraic, dong2020copt}. One possibility to compare two general objects in a pose invariant manner is to model them as metric spaces $(X,\dX)$ and $(Y,\dY)$ and regard them as elements of the collection of isometry classes of compact metric spaces denoted by $\mathcal{M}$ (i.e. two compact metric spaces $(X,d_X)$ and $(Y,d_Y)$ are in the same class if and only if they are isometric to each other which we denote by $X\cong Y$). It is possible to compare $(X,\dX)$ and $(Y,\dY)$ via the \emph{Gromov-Hausdorff distance} \cite{edwards1975structure,gromov1981groups}, which is a metric on $\mathcal{M}$. It is defined as
	\begin{equation}\label{eq:Gromov Hausdorff}
	   d_{\mathrm{GH}}(X,Y):=\inf_{Z,\phi,\psi}d^{(Z,d_Z)}_{\mathrm{H}}(\phi(X),\psi(Y)),\end{equation}
	where $\phi:X\to Z$ and $\psi:Y\to Z$ are isometric embeddings into a metric space $(Z,d_Z)$ and $d^{(Z,d_Z)}_\mathrm{H}$ denotes the \emph{Hausdorff distance in $Z$}. The Hausdorff distance is a metric on the collection of compact subsets of a metric space $(Z,d_Z)$, which is denoted by $\mathcal{S}(Z)$, and for $A,B\in\mathcal{S}(Z)$ defined as follows
\begin{equation}
d_\mathrm{H}^{(Z,d_Z)}\left(A,B\right):=\max\left(  {\sup\limits_{a\in A} \Inf\limits_{b\in B}}d_Z(a,b),~\sup\limits_{b\in B} \Inf\limits_{a\in A}d_Z(a,b)\right).
\end{equation}	
While the Gromov-Hausdorff distance has been applied successfully for various shape and data analysis tasks (see e.g. \cite{memoli2004comparing,bronstein2006efficient,bronstein2006generalized,bronstein2009partial,bronstein2009topology,chazal2009gromov,bronstein2010gromov,carlsson2010characterization}), it turns out that it is generally convenient to equip the modelled objects with more structure and to model them as \emph{metric measure spaces}  \citep{memoli2007use,memoli2011gromov}. A metric measure space $\X=\mmspaceX$ is a triple, where $(X,\dX)$ denotes a metric space and $\muX$ stands for a Borel probability measure on $X$ with full support. This additional probability measure can be thought of as signalling the importance of different regions in the modelled object. Moreover, two metric measure spaces $\X=\mmspaceX$ and $\Y=\mmspaceY$ are considered as isomorphic (denoted by $\X\cong_w\Y$) if and only if there exists an isometry $\varphi:(X,\dX)\to (Y,\dY)$ such that $\varphi_\#\muX=\muY$. Here, $\varphi_\#$ denotes the pushforward map induced by $\varphi$. From now on, $\isocalss$ denotes the collection of all (isomorphism classes of) compact metric measure spaces.

The additional structure of the metric measure spaces allows to regard the modelled objects as probability measures instead of compact sets. Hence, it is possible to substitute the  Hausdorff component in \Cref{eq:Gromov Hausdorff} by a relaxed notion of proximity, namely the \emph{Wasserstein distance}. This distance is fundamental to a variety of mathematical developments
and is also known as Kantorovich distance \citep{kantorovith1942translocation}, Kantorovich-Rubinstein distance \citep{kantorowitsch1958space}, Mallows distance \citep{mallows1972note} or as the Earth Mover's distance \citep{rubner2000earth}. Given a compact metric space $(Z,d_Z)$, let $\mathcal{P}(Z)$ denote the space of probability measures on $Z$ and let $\alpha,\beta\in \mathcal{P}(Z)$. Then, the Wasserstein distance of order $p$, for $1\leq p< \infty$, between $\alpha$ and $\beta$ is defined as

\begin{equation}\label{eq:Wasserstein}
d^{(Z,d_Z)}_{\mathrm{W},p}(\alpha,\beta):=\left(\inf_{\mu\in\mathcal{C}(\alpha,\beta)}\int_{Z\times Z} d^p_Z(x,y)\,\mu(dx\times dy)\right)^\frac{1}{p},
\end{equation}
 and for $p=\infty$ as
 \begin{equation}\label{eq:def Wasserstein infinity}
     d^{(Z,d_Z)}_{\mathrm{W},\infty}(\alpha,\beta)\coloneqq\inf_{\mu\in\mathcal{C}(\alpha,\beta)}\sup_{(x,y)\in\mathrm{supp}(\mu)}d_Z(x,y),
 \end{equation}
where $\supp{\mu}$ stands for the support of $\mu$ and $\mathcal{C}(\alpha,\beta)$ denotes the set of all couplings of $\alpha$ and $\beta$, i.e., the set of all probability measures $\mu$ on the product space $Z\times Z$ such that 
\[\mu(A\times Z)=\alpha(A)~\text{ and  }~\mu(Z\times B)=\beta(B)\]
for all Borel measurable sets $A$ and $B$ of $Z$. 
It is worth noting that the Wasserstein distance between probability measures on the real line admits a closed form solution (see \cite{villani2003topics} and Remark \ref{rem:closed-form}). 

\citet{sturm2006geometry} has shown that replacing the Hausdorff distance in \Cref{eq:Gromov Hausdorff} with the Wasserstein distance indeed yields a meaningful metric on $\isocalss$. Let $\X=\mmspaceX$ and $\Y=\mmspaceY$ be two metric measure spaces. Then, \emph{Sturm's Gromov-Wasserstein distance} of order $p$, $1\leq p\leq \infty$, is defined as
\begin{equation}\label{eq:stdSturm}d_{\mathrm{GW},p}^{\mathrm{sturm}}(\X,\Y)\coloneqq \inf_{Z,\phi,\psi} d_{\mathrm{W},p}^{{(Z,d_Z)}}(\phi_\#\muX,\psi_\#\muY),\end{equation}
	where $\phi:X\to Z$ and $\psi:Y\to Z$ are isometric embeddings into the metric space $(Z,d_Z)$.

Based on similar ideas but starting from a different representation of the Gromov-Hausdorff distance, M\'emoli \cite{memoli2007use,memoli2011gromov} derived a computationally more tractable and topologically equivalent metric on $\isocalss$, namely the \emph{Gromov-Wasserstein} distance: For $1\leq p<\infty$, the \emph{$p$-distortion} of a coupling $\mu\in\mathcal{C}(\muX,\muY)$ is defined as
\begin{equation}\label{eq:distortion}
\mathrm{dis}_p(\mu)\coloneqq \left(\iint_{X\times Y \times X\times Y}\big|\dX(x,x')-\dY(y,y')\big|^p\,\mu(dx\times dy)\,\mu(dx'\times dy')\right)^{1/p}\end{equation}
and for $p=\infty$ it is given as
\[\mathrm{dis}_\infty(\mu)\coloneqq\sup_{\substack{x,x'\in \X,\,y,y'\in \Y\\ s.t.\,(x,y),(x',y')\in \supp{\mu}}}\big|\dX(x,x')-\dY(y,y')\big|.\]
The \emph{Gromov-Wasserstein distance} of order $p$, $1\leq p\leq \infty$, is defined as
\begin{equation}\label{eq:Gromov Wasserstein}
d_{\mathrm{GW},p}(\X,\Y)\coloneqq\frac{1}{2}\inf_{\mu\in\mathcal{C}(\muX,\muY)} \mathrm{dis}_p(\mu).
\end{equation}

It is known that in general $d_{\mathrm{GW},p}\leq \dsturm{p}$ and that the inequality can be strict \cite{memoli2011gromov}. Although both $\dsturm{p}$ and $d_{\mathrm{GW},p}$, $1\leq p\leq \infty$, are in general NP-hard to compute \citep{memoli2011gromov}, it is possible to efficiently approximate $d_{\mathrm{GW},p}$ via conditional gradient descent \citep{memoli2011gromov,peyre2016gromov}. This has led to numerous applications and extensions of this distance \citep{alvarez2018gromov,titouan2019optimal,bunne2019learning,chowdhury2020generalize,scetbon2021linear}.

In many cases, since the direct computation of either of these distances can be onerous, the determination of the degree of similarity between two datasets is performed via firstly  computing  \emph{invariant features} out of each dataset (e.g. global distance distributions \citep{osada2002shape}) and secondly by suitably  comparing these features. This point of view has motivated the exploration of inverse problems arising from the study of such features \citep{memoli2011gromov,sturm2012space,brinkman2012invariant,memoli2021distance}.

Clearly, $\isocalss$ contains various, extremely general spaces. However, in many applications it is possible to have prior knowledge about the metric measure spaces under consideration and it is often reasonable to restrict oneself to work on a specific sub-collections $\mathcal{O}^w\subseteq\isocalss$. For instance, it could be known that the metrics of the spaces considered are induced by the shortest path metric on some underlying trees and hence it is unnecessary to consider the calculation of $\dsturm{p}$ and $\dgw{p}$, $1\leq p\leq \infty$, for all of $\isocalss$.
The potential advantages of focusing on a specific sub-collection $\mathcal{O}^w$ are twofold. On the one hand, it might be possible to use the features of $\mathcal{O}^w$ to gain computational benefits. On the other hand, it might be possible to refine the definition $\dsturm{p}$ and $\dgw{p}$, $1\leq p\leq \infty$, to obtain more informative comparisons on $\mathcal{O}^w$. Naturally, it is of interest to identify and study these subclasses and the corresponding refinements.
This approach has been pursued to study (variants of) the Gromov-Hausdorff distance on compact \emph{ultrametric spaces} by \citet{zarichnyi2005gromov} and \citet{qiu2009geometry}, and on compact \emph{p-metric spaces} by \citet{memoli2019gromov}. Here, the metric space $\left(X,d_X\right)$ is called a $p$-metric space $(1\leq p<\infty)$, if for all $x,x',x''\in X$ it holds
	\[d_X(x,x'')\leq \left(d_X(x,x')^p+d_X(x',x'')^p\right)^{1/p}.\]   
Further, the metric space $(X,u_X)$ is called an ultrametric space, if 
$\uX$ fulfills for all $x,x',x''\in X$ that
	\begin{equation}\label{eq:ultra triangle ineq}
	\uX(x',x'')\leq \max(\uX(x,x'),\uX(x',x'')).
	\end{equation}
In particular, note that ultrametrics can be considered as the limiting case of $p$-metrics as $p\rightarrow\infty$. In particular, \citet{memoli2019gromov} derived a polynomial time algorithm for the calculation of the \emph{ultrametric Gromov-Hausdorff} distance $\ugh$ between two compact ultrametric spaces $(X,\uX)$ and $(Y,\uY)$ (see \Cref{sec:ultrametric Gromov-Hausdorff}), which is defined as 	\begin{equation}\label{eq:Gromov Hausdorff ultrametric}
	   u_{\mathrm{GH}}(X,Y):=\inf_{Z,\phi,\psi}d^{(Z,u_Z)}_{\mathrm{H}}(\phi(X),\psi(Y)),\end{equation}
where $\phi:X\to Z$ and $\psi:Y\to Z$ are isometric embeddings into a common \emph{ultrametric} space $(Z,u_Z)$ and $d^{(Z,u_Z)}_\mathrm{H}$ denotes the Hausdorff distance on $Z$.

A further motivation to study (surrogates of) the distances $\dsturm{p}$ and $\dgw{p}$ restricted on a subset $\mathcal{O}^w$ comes from the idea of \emph{slicing} which originated as a method to efficiently estimate  the Wasserstein distance $d^{\R^d}_{\mathrm{W},p}(\alpha,\beta)$ between probability measures $\alpha$ and $\beta$ supported in a high dimensional euclidean space $\R^d$ \cite{rubner2000earth}.  The original idea is that given any line $\ell$ in $\R^d$ one first  obtains $\alpha_\ell$ and $\beta_\ell$, the respective pushforwards of $\alpha$ and $\beta$ under the orthogonal projection map $\pi_\ell : \R^d \rightarrow \ell$, and then one invokes the explicit formula for the Wasserstein distance for probability measures on $\R$ (see \Cref{rem:closed-form}) to obtain a lower bound to $d^{\R^d}_{\mathrm{W},p}(\alpha,\beta)$ without incurring the possibly high computational cost associated to solving an optimal transportation problem. This lower bound is improved via repeated (often random) selections of the line $\ell$ \citep{rubner2000earth, bonneel2015sliced,kolouri2019generalized}.

Recently, \citet{le2019tree} pointed out that, thanks to the fact that the $1$-Wasserstein distance also admits an explicit formula when the underlying metric space is a tree \citep{do2011sublinear,evans2012phylogenetic,mcgregor2013sketching}, one can also devise \emph{tree slicing} estimates of the distance between two given probability measures by suitably projecting them onto tree-like structures. Most likely, the same strategy is successful for suitable projections on random ultrametric spaces, as on these there is also an explicit formula for the Wasserstein distance \citep{kloeckner2015geometric}. The same line of of work has also recently been explored in the Gromov-Wasserstein scenario \citep{vayer2019sliced,le2019fast} and could be extended based on efficiently computable restrictions (or surrogates of) $\dsturm{p}$ and $\dgw{p}$. Inspired by the results of \citet{memoli2019gromov} on the ultrametric Gromov-Hausdorff distance and the results of \citet{kloeckner2015geometric}, who derived an explicit representation of the Wasserstein distance on ultrametric spaces, we study  the collection of compact \emph{ultrametric measure spaces} $\mathcal{U}^w\subseteq\isocalss$, where  $\X=\ummspaceX\in \mathcal{U}^w$, whenever the underlying metric space $(X,\uX)$ is a compact ultrametric space.

In terms of applications, ultrametric spaces (and thus also ultrametric \emph{measure} spaces) arise naturally in statistics as metric encodings of dendrograms \citep{jardine1971mathematical,carlsson2010characterization} which is a graph theoretical representations of ultrametric spaces, in the context of phylogenetic trees \citep{semple2003phylogenetics}, in theoretical computer science in the probabilistic approximation of finite metric spaces \citep{bartal1996probabilistic,fakcharoenphol2004tight}, and in physics in the context of a mean-field theory of spin glasses
\citep{mezard1987spin,Rammal1986UltrametricityFP}. 

Especially for phylogenetic trees (and dendrograms), where one tries to characterize the structure of an underlying  evolutionary process or the difference between two such processes, it is important to have a meaningful method of comparison, i.e., to have a meaningful metric on $\mathcal{U}^w$. However, it is evident from the definition of $d_{\mathrm{GW},p}^{\mathrm{sturm}}$ and the relationship between $d_{\mathrm{GW},p}^{\mathrm{sturm}}$ and $d_{\mathrm{GW},p}$ (see \cite{memoli2011gromov}), that the ultrametric structure of $\X,\Y\in\mathcal{U}^w$ is not taken into account in the computation of either $d_{\mathrm{GW},p}^{\mathrm{sturm}}(\X,\Y)$ or $d_{\mathrm{GW},p}(\X,\Y)$, $1\leq p\leq \infty$. Hence, we suggest, just as for the ultrametric Gromov-Hausdorff distance, to adapt the definition of $d_{\mathrm{GW},p}^{\mathrm{sturm}}$ (see \Cref{eq:stdSturm}) as well as the one of $\dgw{p}$ (see \Cref{eq:Gromov Wasserstein}) and verify in the following that this makes the comparisons of ultrametric measure spaces more sensitive and leads for $p=\infty$ to a \emph{polynomial time} algorithm for the derivation of the proposed metrics.

	\subsection{The proposed approach}

	Let $\X=\ummspaceX$ and $\Y=\ummspaceY$ be ultrametric measure spaces. Reconsidering the definition of  Sturm's Gromov-Wasserstein distance in \Cref{eq:stdSturm},  we propose to only infimize over ultrametric spaces $(Z,u_Z)$ in \Cref{eq:stdSturm}. Thus, we define for $p\in[1,\infty]$ \emph{Sturm's ultrametric Gromov-Wasserstein distance} of order $p$ as \begin{equation}\label{eq:ultra Sturm}\usturm{p}(\X,\Y)\coloneqq\inf_{Z,\phi,\psi} d_{\mathrm{W},p}^{{(Z,u_Z)}}(\phi_\#\muX,\psi_\#\muY),\end{equation}
	where $\phi:X\to Z$ and $\psi:Y\to Z$ are isometric embeddings into an ultrametric space $(Z,u_Z)$.

	In the subsequent sections of this paper, we will establish many theoretically appealing properties of $\usturm{p}$. Unfortunately, we will verify that, although an explicit formula for the Wasserstein distance of order $p$ on ultrametric spaces exists \citep{kloeckner2015geometric}, for $p\in [1,\infty)$ the calculation of $\usturm{p}$ yields a highly non-trivial combinatorial optimization problem (see \Cref{subsubsec:alt rep for Sturms GW dist}). Therefore, we demonstrate that an adaption of the Gromov-Wasserstein distance defined in \Cref{eq:Gromov Wasserstein} yields a topologically equivalent and easily approximable distance on $\mathcal{U}^w$. In order to define this adaption, we need to introduce some notation. For $a,b\geq 0$ and $1\leq q <\infty$ let \[\Lambda_q(a,b)\coloneqq |a^q-b^q|^{1/q}.\]
	Further define $\Lambda_\infty(a,b)\coloneqq\max(a,b)$ whenever $a\neq b$ and $\Lambda_\infty(a,b)=0$ if $a=b$. 

Now, we can rewrite $d_{\mathrm{GW},p}$, $1\leq p\leq\infty,$ as follows
\begin{equation}\label{eq:p-dist delta1}
	d_{\mathrm{GW},p}(\X,\Y)=\frac{1}{2}\inf_{\mu\in\mathcal{C}(\muX,\muY)}\left(\iint_{X\times Y \times X\times Y}\!\!\!\!\!\big(\Lambda_1(\dX(x,x'),\dY(y,y'))\big)^p\,\mu(dx\times dy)\,\mu(dx'\times dy')\right)^{1/p}\!\!\!.
\end{equation}
	Considering the derivation of $d_{\mathrm{GW},p}$ in \cite{memoli2011gromov} and the results on the closely related ultrametric Gromov-Hausdorff distance studied in \cite{memoli2019gromov}, this suggests to replace $\Lambda_1$ in \Cref{eq:p-dist delta1} with $\Lambda_\infty$ in order to incorporate the ultrametric structures of $\ummspaceX$ and $\ummspaceY$ into the comparison. Hence, we define the \emph{$p$-ultra-distortion} of a coupling $\mu\in\mathcal{C}(\muX,\muY)$ for $1\leq p<\infty$ as
	\begin{equation}\label{eq:distortion ult}
	    \mathrm{dis}_p^\mathrm{ult}(\mu)\coloneqq \left(\iint_{X\times Y \times X\times Y}\big(\Lambda_\infty(\uX(x,x'),\uY(y,y'))\big)^p\,\mu(dx\times dy)\,\mu(dx'\times dy')\right)^{1/p}.
	\end{equation}
	and for $p=\infty$ as
	\[\mathrm{dis}_\infty^\mathrm{ult}(\mu)\coloneqq\sup_{\substack{x,x'\in \X,\,y,y'\in \Y\\ s.t.\,(x,y),(x',y')\in \supp{\mu}}}\Lambda_\infty(\uX(x,x'),\uY(y,y')).\]
    The \emph{ultrametric Gromov-Wasserstein distance} of order $p\in[1,\infty]$, is given as
	\begin{equation}
	\ugw{p}(\X,\Y)\coloneqq\inf_{\mu\in\mathcal{C}(\muX,\muY)} \mathrm{dis}_p^\mathrm{ult}(\mu).\label{eq:def uGW}
	\end{equation}
	Due to the structural similarity between $\dgw{p}$ and $\ugw{p}$, we can expect (and later verify) that many properties of $\dgw{p}$ extend to $\ugw{p}$. In particular, we will establish that also $\ugw{p}$ can be approximated\footnote{Here ``approximation" is meant in the sense that one can write code which will locally minimize the functional. There are in general no theoretical guarantees that these algorithms will converge to a global minimum.}  via conditional gradient descent and admits several polynomial time computable lower bounds which are useful in applications.	

	It is worth mentioning that \citet{sturm2012space} studied the family of so-called $L^{p,q}$-distortion distances similar to our construction of $\ugw{p}$. In our language, for any $p,q\in[1,\infty)$, the $L^{p,q}$-distortion distance is constructed by infimizing over the $(p,q)$-distortion defined by replacing $\Lambda_\infty$ with $(\Lambda_q)^q$ in \Cref{eq:distortion ult}. This distance shares many properties with $\dgw{p}$.
	
\subsection{Overview of our results}\phantom{a}\vspace{3mm}\\
We give a brief overview of our results.

\textbf{\Cref{sec:preliminaries}.} We  generalize the results of \citet{carlsson2010characterization} on the relation between ultrametric spaces and dendrograms and establish a bijection between compact ultrametric spaces and \emph{proper dendrograms} (see \Cref{def:proper dendrogram}). After recalling some results on the ultrametric Gromov-Hausdorff distance (see \Cref{eq:Gromov Hausdorff ultrametric}), we use the connection between compact ultrametric spaces and dendrograms to reformulate the explicit formula for the $p$-Wasserstein distance ($1\leq p< \infty$) on ultrametric spaces derived by \citet{kloeckner2015geometric} in terms of proper dendrograms. This allows us to derive a formulation of the $\infty$-Wasserstein distance on ultrametric spaces and to study the Wasserstein distance on compact subspaces of the ultrametric space $(\Rp,\Lambda_\infty)$, which will be relevant when studying lower bounds of $\ugw{p}$, $1\leq p\leq \infty$. 

\textbf{\Cref{sec:ultrametric GW distance}.} We demonstrate that $\ugw{p}$ and $\usturm{p}$, $1\leq p\leq \infty$, are $p$-metrics on the collection of ultrametric measure spaces $\mathcal{U}^w$. We derive several alternative representations for $\usturm{p}$ and study the relation between the metrics $\usturm{p}$ and $\ugw{p}$. In particular, we show that, while for $1\leq p<\infty$ it holds in general that $\ugw{p}\leq2^\frac{1}{p}\,\usturm{p}$, both metrics coincide for $p=\infty$, i.e., $\ugw{\infty}=\usturm{\infty}$. Furthermore, we show how this equality in combination with an alternative representation of $\ugw{\infty}$ leads to a \emph{polynomial time algorithm} for the calculation of $\usturm{\infty}=\ugw{\infty}$. Moreover, we study the topological properties of  $(\mathcal{U}^w,\usturm{p})$ and $(\mathcal{U}^w,\ugw{p})$, $1\leq p\leq \infty$. Most importantly, we show that $\usturm{p}$ and $\ugw{p}$ induce the same topology on $\mathcal{U}^w$ which is also different from the one induced by $\dsturm{p}/\dgw{p}$, $1\leq p\leq \infty$. While we further prove that the metric spaces $(\mathcal{U}^w,\usturm{p})$ and $(\mathcal{U}^w,\ugw{p})$,  $1\leq p<\infty$, are neither complete nor separable metric space, we demonstrate that the ultrametric space  $(\mathcal{U}^w,\usturm{\infty})$, which coincides with $(\mathcal{U}^w,\ugw{\infty})$, is complete. Finally, we establish that  $(\mathcal{U}^w,\usturm{1})$ is a geodesic space.

\textbf{\Cref{sec:lower bounds}.} Unfortunately, it does not seem to be possible to derive a polynomial time algorithm for the calculation of $\usturm{p}$ and $\ugw{p}$, $1\leq p<\infty$. Consequently, based on easily computable invariant features, in \Cref{sec:lower bounds} we derive several polynomial time computable lower bounds for $\ugw{p}$, $1\leq p\leq \infty$. Due to the structural similarity between $\dgw{p}$ and $\ugw{p}$, these are in a certain sense analogue to those derived in \cite{memoli2007use,memoli2011gromov} for $\dgw{p}$. Among other things, we show that 
\begin{equation}
\ugw{p}(\X,\Y)\geq\uSLB{p}(\X,\Y)\coloneqq\inf_{\gamma\in\mathcal{C}(\muX\otimes \muX,\muY\otimes\muY)}\norm{\Lambda_\infty(\uX,\uY)}_{L^p(\gamma)}.\end{equation}
We verify that the lower bound $\uSLB{p}$ can be reformulated in terms of the Wasserstein distance on the ultrametric space $(\Rp,\Lambda_\infty)$ (we derive an explicit formula for $d_{\mathrm{W},p}^{(\Rp,\Lambda_\infty)}$ in \Cref{sec:explicit formulat}). This allows us to efficiently calculate $\uSLB{p}(\X,\Y)$ in $O((m\vee n)^2)$, where $m$ stands for the cardinality of $X$ and $n$ for the one of $Y$.

\textbf{\Cref{sec:ultra-dissimilarity spaces}.} As the ultrametric space assumption is somewhat restrictive (especially in the context of phylogenetic trees, see \cite{semple2003phylogenetics}), we prove in \Cref{sec:ultra-dissimilarity spaces} that the results on $\ugw{p}$ can be extended to the more general \emph{ultra-dissimilarity spaces} (see \Cref{def:ultra dissimilarity}). In particular, we prove that $\ugw{p}$, $1\leq p\leq \infty$, is a metric on the \emph{isomorphism classes} of ultra-dissimilarity spaces (see \Cref{def:isomorphism ultra-dissimilarity}).

\textbf{\Cref{sec:computational aspects}.} We illustrate the behaviour and relation between $\ugw{1}$ (which can be approximated via conditional gradient descent) and $\uSLB{1}$ in a set of illustrative examples. Additionally, we carefully illustrate the differences between $\ugw{1}$ and $\uSLB{1}$, and $\dgw{1}$ and $\dSLB{1}$ (see \Cref{sec:lower bounds} for a definition), respectively.

\textbf{\Cref{sec:phylogenetic tree shapes}.} Finally, we apply our ideas to \emph{phylogenetic tree shape comparison}. To this end, we compare two sets of phylogenetic tree shapes based on the HA protein sequences from human influenza collected in different regions with the lower bound $\uSLB{1}$. In particular, we contrast our results in both settings to the ones obtained with the tree shape metric introduced in Equation (4) of \citet{colijn2018metric}.
	%%%%%%%%%%%%%%%%%%%%%%%
	\subsection{Related work}

In order to better contextualize our contribution, we now describe  related work, both in  applied and computational geometry, and in phylogenetics (where notions of distance between trees have arisen naturally).

\subsubsection*{Metrics between trees: the phylogenetics perspective}
In phylogenetics, where one chief objective is to infer the evolutionary relationship between species via methods that evaluate observable traits, such as DNA sequences,
the need to be able to measure dissimilarity between different trees arises from the fact that the process of reconstruction of a phylogenetic tree may depend on the set of genes being considered. At the same time, even for the same set of genes, different reconstruction methods could be applied which would result in different trees. As such, this has led to the development of many different metrics for measuring distance between phylogenetic trees. Examples include  the
Robinson-Foulds metric \citep{robinson1981comparison}, the subtree-prune and regraft distance \citep{hein1990reconstructing}, and the nearest-neighbor interchange distance \citep{robinson1971comparison}.  

As pointed out in \cite{owen2010fast}, many of these distances tend to quantify differences between tree topologies and often do not take into account edge lengths. A certain phylogenetic tree metric space which encodes for edge lengths was proposed in \cite{billera2001geometry} and 
studied algorithmically in \cite{owen2010fast}. This tree space assumes that the all trees have the same set of taxa. An extension to the case of trees over different underlying sets is given in \cite{grindstaff2018geometric}. \citet{lafond2019complexity} considered one type of metrics on possibly \emph{muiltilabeled} phylogenetic trees with a fixed number of leafs. As the authors pointed out, a multilabeled phylogenetic tree in which no leafs are repeated is just a standard phylogenetic tree, whereas a  multilabeled phylogenetic tree in which all labels are equal defines a \emph{tree shape}. The authors then proceeded to study the computational complexity associated to  generalizations of some of the usual metrics for phylogenetic trees (such as the Robinson-Foulds distance) to the multilabeled case. \citet{colijn2018metric} studied a metric between (binary) phylogenetic tree shapes based on a bottom to top enumeration of specific connectivity structures. The authors applied their metric to compare evolutionary trees based on the HA protein sequences from human influenza collected in different regions.

\subsubsection*{Metrics between trees: the applied geometry perspective}
From a different perspective, ideas from applied geometry and applied and computational topology have been applied to the comparison of tree shapes in applications in probability, clustering and applied and computational topology.

Metric trees are also considered in probability theory in the study of models for random trees together with the need to quantify their distance; \citet{evans2007probability} described some variants of the Gromov-Hausdorff  distance between metric trees. See also \cite{greven2009convergence} for the case of metric measure space representations of trees and a certain Gromov-Prokhorov type of metric on the collection thereof.

Trees, in the form of dendrograms, are abundant in the realm of hierarhical clustering methods. In their study of the \emph{stability} of hierarchical clustering methods, \citet{carlsson2010characterization} utilized the Gromov-Hausdorff distance between the ultrametric representation of dendrograms. \citet{DBLP:journals/dcg/Schmiedl17} proved that computing the Gromov-Hausdorff distance between tree metric spaces is NP-hard.  \citet{liebscher2018new} suggested some variants of the Gromov-Hausdorff distance which are applicable in the context of phylogenetic trees. As mentioned before, \citet{zarichnyi2005gromov} introduced the ultrametric Gromov-Hausdorff distance $\ugh$ between compact ultrametric spaces (a special type of tree metric spaces). Certain theoretical properties such as precompactness of $\ugh$ has been studied in \cite{qiu2009geometry}. In contrast with the NP-hardness of computing $\dgh$, \citet{memoli2019gromov} devised an polynomial time algorithm for computing $\ugh$.

In computational topology \emph{merge trees} arise through the study of the sublevel sets of a given function \citep{adelson1945level,reeb1946points} with the goal of shape simplification. \citet{morozov2013interleaving} developed the notion of \emph{interleaving distance} between merge trees which is related to the Gromov-Hausdorff distance between trees through bi-Lipschitz bounds. In \cite{agarwal2018computing}, exploiting the connection between the interleaving distance and the Gromov-Hausdorff between metric trees, the authors approached the computation of the Gromov-Hausdorff distance between metric trees in general and provide certain approximation algorithms. \citet{touli2018fpt} devised fixed-parameter tractable (FPT) algorithms for computing the interleaving distance between metric trees. One can imply from their methods an FPT algorithm to compute a 2-approximation of the Gromov-Hausdorff distance between ultrametric spaces. \citet{memoli2019gromov} devised an FPT algorithm for computing the exact value of the Gromov-Hausdorff distances between ultrametric spaces.

	%%%%%%%%%%%%%%%%%%%%%%%%
\section{Preliminaries}\label{sec:preliminaries}
In this section we briefly summarize the basic notions and concepts required throughout the paper. 
	
\subsection{Ultrametric spaces and dendrograms}
We begin by describing compact ultrametric spaces in terms of  \emph{proper dendrograms}. To this end, we introduce some definitions and some notation. Given a set $X$, a \emph{partition} of $X$ is a set $P_X=\{X_i\}_{i\in I}$ where $I$ is any index set, $\emptyset\neq X_i\subseteq X$, $X_i\cap X_j=\emptyset$ for all $i\neq j\in I$ and $\bigcup_{i\in I}X_i=X$. We call each element $X_i$ a \emph{block} of the given partition $P_X$ and denote by $\mathbf{Part}(X)$ the collection of all partitions of $X$. For two partitions $P_X$ and $P'_X$ we say that $P_X$ is \emph{finer} than $P'_X$, if for every block $X_i\in P_X$ there exists a block $X'_j\in P'_X$ such that $X_i\subseteq X'_j$.

\begin{definition}[Proper dendrogram]\label{def:proper dendrogram}
Given a set $X$ (not necessarily finite), a \emph{proper dendrogram} $\theta_X:[0,\infty)\rightarrow \mathbf{Part}(X)$ is a map satisfying the following conditions:
\begin{enumerate}
\item $\theta_X(s)$ is finer than $\theta_X(t)$ for any $0\leq s<t<\infty$;
\item $\theta_X(0)$ is the finest partition consisting only singleton sets;
\item There exists $T>0$ such that for any $t\geq T$, $\theta_X(t)=\{X\}$ is the trivial partition;			\item For each $t> 0$, there exists $\eps>0$ such that $\theta_X(t)=\theta_X(t')$ for all $t'\in[t,t+\eps]$.
			\item For any distinct points $x,x'\in X$, there exists $T_{xx'}>0$ such that $x$ and $x'$ belong to different blocks in $\theta_X(T_{xx'})$.
			\item For each $t>0$, $\theta_X(t)$ consists of only finitely many blocks.
			\item Let $\{t_n\}_{n\in\mathbb N}$ be a decreasing sequence such that $\lim_{n\rightarrow\infty}t_n=0$ and let $X_n\in \theta_X(t_n)$. If for any $1\leq n<m$, $X_m\subseteq X_n$, then $\bigcap_{n\in\mathbb N}X_n\neq\emptyset$. 
		\end{enumerate}
	\end{definition}
	When $X$ is finite, a function $\theta_X:[0,\infty)\rightarrow \mathbf{Part}(X)$ satifying conditions (1) to (4) will satisfy conditions (5), (6) and (7) automatically, and thus a proper dendrogram reduces to the usual dendrogram (see \cite[Sec. 3.1]{carlsson2010characterization} for a formal definition).
Let $\theta_X$ be a proper dendrogram over a set $X$. For any $x\in X$ and $t\geq 0$, we denote by $[x]_t^X$ the block in $\theta(t)$ that contains $x\in X$ and abbreviate $[x]_t^X$ to $[x]_t$ when the underlying set $X$ is clear from the context. Similar to \cite{carlsson2010characterization}, who considered the relation between finite ultrametric spaces and {dendrograms}, we will prove that there is a bijection between compact ultrametric spaces and proper dendrograms. In particular, one can show that the subsequent theorem generalizes \cite[Theorem 9]{carlsson2010characterization}. Since its  proof depends on several concepts not yet introduced, we postpone it to \Cref{proof:thm:compact ultra-dendro}.
\begin{theorem}\label{thm:compact ultra-dendro}
Given a set $X$, denote by $\mathcal{U}(X)$ the collection of all compact ultrametrics on $X$ and $\mathcal{D}(X)$ the collection of all proper dendrograms over $X$. For any $\theta\in\mathcal{D}(X)$, consider $u_\theta$ defined as follows:
\[\forall x,x'\in X,\,\,\, u_\theta(x,x')\coloneqq\inf\{t\geq 0\,|\,x,x' \text{ belong to the same block of } \theta(t)\}.\]
Then, $u_\theta\in\mathcal{U}(X)$ and the map $\Delta_X:\mathcal{D}(X)\rightarrow\mathcal{U}(X)$ sending $\theta$ to $u_\theta$ is a bijection.
\end{theorem}
\begin{remark}\label{rem:corresponding dendrogram}
From now on, we denote by $\theta_X$ the proper dendrogram corresponding to a given compact ultrametric $u_X$ on $X$ under the bijection given above. Note that a block $[x]_t$ in $\theta_X(t)$ is actually the closed ball $B_t(x)$ in $X$ centered at $x$ with radius $t$. So for each $t\geq 0$, $\theta_X(t)$ partitions $X$ into a union of several closed balls in $X$ with respect to $u_X$.    
\end{remark}

			%%%%%%%%%%%%%%
	\subsection{The ultrametric Gromov-Hausdorff distance}\label{sec:ultrametric Gromov-Hausdorff}
	Both $\dsturm{p}$ and $\dgw{p}$, $1\leq p\leq \infty$, are by construction closely related to the Gromov-Hausdorff distance. In a recent paper, \citet{memoli2019gromov} studied an ultrametric version of this distance, namely the \emph{ultrametric Gromov-Hausdorff distance} (denoted as $u_\mathrm{GH}$). Since we will demonstrate several connections between $\usturm{p}$, $\ugw{p}$, $1\leq p\leq \infty$, and this distance, we briefly summarize some of the results in \cite{memoli2019gromov}. We start by recalling the formal definition of $u_\mathrm{GH}$.
    \begin{definition}
	Let $(X,\uX)$ and $(Y,\uY)$ be two compact ultrametric spaces. Then, the \emph{ultrametric Gromov-Hausdorff} between $X$ and $Y$ is defined as
	\[u_\mathrm{GH}(X,Y)=\inf_{Z,\phi,\psi}d^Z_\mathrm{H}\left(\phi(X),\psi(Y)\right),\]
	where $\phi:X\to Z$ and $\psi:Y\to Z$ are isometric embeddings (distance preserving transformations) into the ultrametric space $(Z,u_Z)$.
	\end{definition}
	
\citet{zarichnyi2005gromov} has shown that $\ugh$ is an ultrametric on the isometry classes of compact ultrametric spaces, which are denoted by $\mathcal{U}$, and \citet{memoli2019gromov} identified a structural theorem (cf. \Cref{thm:ultrametric GH-distance}) that gives rise to a polynomial time algorithm for the calculation of $\ugh$. More precisely, it was proven in \cite{memoli2019gromov} that $\ugh$ can be calculated via so-called \emph{quotient ultrametric spaces}, which we define next. Let $(X,\uX)$ be an ultrametric space and let $t\geq 0$. We define an equivalence relation $\sim_t$ on $X$ as follows: $x\sim_t x'$ if and only if $\uX(x,x')\leq t$. We denote by $[x]^X_t$ (resp. $[x]_t$) the equivalence class of $x$ under $\sim_t$ and by $X_t$ the set of all such equivalence classes. In fact, $[x]_t^X=\{x'\in X|\,u(x,x')\leq t\}$ is exactly the closed ball centered at $x$ with radius $t$ and corresponds to a block in the corresponding proper dendrogram $\theta_X(t)$ (see \Cref{rem:corresponding dendrogram}). 
{Thus, one can think of $X_t$ as a ``set representation'' of $\theta_X(t)$.}
	We define an ultrametric $u_{X_t}$ on $X_t$ as follows:
	\[u_{X_t}([x]_t,[x']_t)\coloneqq\begin{cases}
	\uX(x,x'),& [x]_t\neq[x']_t\\
	0,& [x]_t=[x']_t.
	\end{cases} \]
	Then, $(X_t,u_{X_t})$ is an ultrametric space and we call $(X_t, u_{X_t})$ the \emph{quotient }of $(X,u_X)$ at level $t$  (see \Cref{fig:accumulated spaces} for an illustration). It is straightforward to prove that the quotient of a compact ultrametric space at level $t>0$ is a finite ultrametric space (cf. \cite[Lemma 2.3]{wan2020novel}). Furthermore, the quotient spaces characterize $\ugh$ as follows.
	\begin{theorem}[Structural theorem for $\ugh$, {\cite[Theorem 5.7]{memoli2019gromov}}]\label{thm:ultrametric GH-distance}
		Let $(X,\uX)$ and $(Y,\uY)$ be two compact ultrametric spaces. Then, 
		\[u_\mathrm{GH}(X,Y)=\inf\left\lbrace t\geq 0 \,|\,X_t \cong Y_t\right\rbrace.\]
	\end{theorem}
	\begin{figure}
		\centering
			\includegraphics[width=0.5\textwidth]{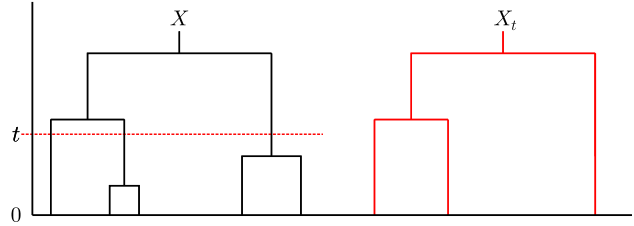}
		\caption{\textbf{Metric quotient:} An ultrametric space (black) and its quotient at level $t$ (red).} \label{fig:accumulated spaces}
	\end{figure}
	\begin{remark}
Let $(X,\uX)$ and $(Y,\uY)$ denote two finite ultrametric spaces and let $t\geq 0$. The quotient spaces $X_t$ and $Y_t$ can be considered as vertex weighted, rooted trees \citep{memoli2019gromov}. Hence, it is possible to check whether $X_t \cong Y_t$ in polynomial time \citep{aho1974design}. Consequently, \Cref{thm:ultrametric GH-distance} induces a simple, polynomial time algorithm to calculate $u_\mathrm{GH}$ between two finite ultrametric spaces.
	\end{remark}

%%%%%%%%%%%%%%%%%%%%%%%%%%%%
	\subsection{Wasserstein distance on ultrametric spaces}\label{sec:explicit formulat}
	\begin{comment}
	Using the synchronized rooted tree representation of ultrametric spaces, Kloeckner identified in \cite{kloeckner2015geometric} the formula computing the Wasserstein distance between two probability measures $\alpha,\beta$ on a finite ultrametric space $(X,\uX)$:
	\begin{equation}\label{eq:w-dist-ultra-tree}
	\left(d_{\mathrm{W},p}^{X}\right)^p(\alpha,\beta)=2^{p-1}\sum_{v\in V\backslash\{o\}}\left(h^p(v*)-h^p(v))\right)\left|\mu(X_v)-\beta(X_v)\right|,
	\end{equation}
	where $X_v$ denotes the set of all points in $X$ (as leafs in $T_X$) having $v$ as an ancestor in $T_X$.
	 Then, we can reinterpret Equation \Cref{eq:w-dist-ultra-tree} as follows: 
	\end{comment}

\citet{kloeckner2015geometric} uses the representation of ultrametric spaces as so called \emph{synchronized rooted trees} to derive an explicit formula for the Wasserstein distance on ultrametric spaces. By the constructions of the dendrograms and of the synchronized rooted trees (see \Cref{sec:synchronized rooted tree}), it is immediately clear how to reformulate the results of \citet{kloeckner2015geometric} on compact ultrametric spaces in terms of proper dendrograms. To this end, we need to introduce some notation. For a compact ultrametric space $X$, let $\theta_X$ be the associated proper dendrogram and let $V(X)\coloneqq \bigcup_{t>0}\theta_X(t)=\{[x]_t|\,x\in X,t> 0\}$. 
It can be shown that $V(X)$ is the collection of all closed balls in $X$ except for singletons $\{x\}$ such that $x$ is a cluster point\footnote{A cluster point $x$ in a topological space $X$ is such that any neighborhood of $x$ contains countably many points in $X$.} (see \Cref{lm:vx characterization}).
For $B\in V(X)$, we denote by $B^*$ the smallest (under inclusion) element in $V(X)$ such that $B\subsetneqq B^*$ (for the existence and uniqueness of $B^*$ see \Cref{lemma:existence of B^*}).

\begin{theorem}[The Wasserstein distance on ultrametric spaces, {\cite[Theorem 3.1]{kloeckner2015geometric}}]\label{lemma:Wasserstein on ultrametric spaces} Let $(X,\uX)$ be a compact ultrametric space. For all $\alpha,\beta\in\mathcal{P}(X)$ and $1\leq p<\infty$, we have 
\begin{equation}
\left(d_{\mathrm{W},p}^{X}\right)^p(\alpha,\beta)=2^{-1}\sum_{B\in V(X)\backslash\{X\}}\left(\diam{B^*}^p-\diam{B}^p\right)\left|\alpha(B)-\beta(B)\right|.
\end{equation}
\end{theorem}
While \Cref{lemma:Wasserstein on ultrametric spaces} is only valid for $p<\infty$, it can be extended to the case $p=\infty$.

	\begin{lemma}\label{lm:winfty-finite}
		Let $X$ be a compact ultrametric space. Then, for any $\alpha,\beta\in P(X)$, we have
		\begin{equation}\label{eq:ultra Wasserstein infinity}
		    d_{\mathrm{W},\infty}^X(\alpha,\beta)=\max_{B\in V(X)\backslash\{X\}\text{ and }\alpha(B)\neq\beta(B)}\diam{B^*}.
		\end{equation}  
	\end{lemma}
	The proof of \Cref{lm:winfty-finite} is technical and we postpone it to \Cref{sec:proof of lm winfty-finite}.

	\subsubsection{Wasserstein distance on \texorpdfstring{$(\Rp,\Lambda_\infty)$}{the ultrametric real line}} 
	The non-negative half real line $\Rp$ endowed with $\Lambda_\infty$ turns out to be an ultrametric space (cf. \cite[Remark 1.14]{memoli2019gromov}). Finite subspaces of $(\mathbb R_{\geq 0},\Lambda_\infty)$ are of particular interest in this paper. These spaces possess a particular structure (see Figure \ref{fig:Rinfty}) and the computation of the Wasserstein distance on them can be further simplified. 
	
	\begin{figure}
		\centering
		\includegraphics[scale = 0.2]{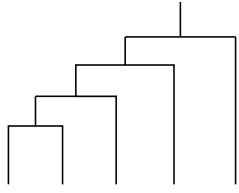}
		\caption{\textbf{Illustration of $(\Rp,\Lambda_\infty)$:} This is the dendrogram for a subspace of $(\Rp,\Lambda_\infty)$ consisting of 5 arbitrary distinct points of $\mathbb{R}_+$. } \label{fig:Rinfty}
	\end{figure}

	\begin{theorem}[$d^{(\Rp,\Lambda_\infty)}_{\mathrm{W},p}$ between finitely supported measures]\label{thm:closed-form-w-infty-real}
	Suppose $\alpha,\beta$ are two probability measures supported on a finite subset $\{x_0,\dots,x_n\}$ of $(\mathbb{R}_{\geq 0},\Lambda_\infty)$ such that $0\leq x_0<x_1<\dots<x_n$. Denote $\alpha_i\coloneqq \alpha(\{x_i\})$ and $\beta_i\coloneqq \beta(\{x_i\})$. Then, we have for $p\in[1,\infty)$ that
		\begin{equation}\label{eq:dp finite} d^{(\Rp,\Lambda_\infty)}_{\mathrm{W},p}(\alpha,\beta)=2^{-\frac{1}{p}}\left(\sum_{i=0}^{n-1}\left|\sum_{j=0}^i(\alpha_j-\beta_j)\right|\cdot|x_{i+1}^p-x_i^p|+\sum_{i=0}^n|\alpha_i-\beta_i|\cdot x_i^p\right)^\frac{1}{p}.
		\end{equation}
	Let $F_\alpha$ and $F_\beta$ denote the cumulative distribution functions of $\alpha$ and $\beta$, respectively. Then, for the case $p=\infty$ we obtain
	\[d_{\mathrm{W},\infty}^{(\Rp,\Lambda_\infty)}(\alpha,\beta)=\max\left(\max_{0\leq i\leq n-1, F_\alpha(x_i)\neq F_\beta(x_i)}x_{i+1},\max_{0\leq i\leq n, \alpha_i\neq\beta_i}x_i\right).\]
	\end{theorem}
	
	\begin{proof}
		Clearly, $V(X)=\{\{x_0,x_1,\ldots,x_i\}|\,i=1,\ldots,n\}\cup\{\{x_i\}|\,i=1,\ldots,n\}$ (recall that each set corresponds to a closed ball). Thus, we conclude the proof by applying \Cref{lemma:Wasserstein on ultrametric spaces} and \Cref{lm:winfty-finite}.
	\end{proof}
	
	\begin{remark}[The case $p=1$]\label{rmk:int-w-inf-real}
		Note that when $p=1$, for any finitely supported probability measures $\alpha,\beta\in\mathcal{P}(\mathbb R_{\geq 0})$,
		\[d_{\mathrm{W},1}^{(\Rp,\Lambda_\infty)}(\alpha,\beta)=\frac{1}{2}\left(d_{\mathrm{W},1}^{(\mathbb{R},\Lambda_1)}(\alpha,\beta)+\int_\mathbb{R}x\,|\alpha-\beta|(dx)\right). \]
		The formula indicates that the $1$-Wasserstein distance on $(\Rp,\Lambda_\infty)$ is the average of the usual $1$-Wasserstein distance on $(\Rp,\Lambda_1)$ and a ``weighted total variation distance''. The weighted total variation like distance term is sensitive to difference of supports. For example, let $\alpha=\delta_{x_1}$ and $\beta=\delta_{x_2}$, then $\int_\mathbb{R}x\,|\alpha-\beta|(dx)=x_1+x_2$ if $x_1\neq x_2$.

\end{remark}
	\begin{remark}[Extension to compactly supported measures]\label{rem:extension to compactly supported measures}
    In fact, $X\subseteq(\Rp,\Lambda_\infty)$ is compact if and only if it is either a finite set or countable with 0 being the unique cluster point (w.r.t. the usual Euclidean distance $\Lambda_1$) (see \Cref{lm:compact of R}). Hence, it is straightforward to extend \Cref{thm:closed-form-w-infty-real} to compactly supported measures and we refer to \Cref{sec:extension to compactly supported measures} for the missing details.
	\end{remark}
	\begin{remark}[Closed-form solution for $d_{\mathrm{W},p}^{(\Rp,\Lambda_q)}$]\label{rem:closed-form}
    We know that there is a closed-form solution for Wasserstein distance on $\mathbb{R}$ with the usual Euclidean distance $\Lambda_1$: 
	\[d_{\mathrm{W},p}^{(\mathbb{R},\Lambda_1)}(\alpha,\beta)=\left(\int_0^1|F_\alpha^{-1}(t)-F_\beta^{-1}(t)|^pdt\right)^\frac{1}{p},\]
	where $F_\alpha$ and $F_\beta$ are cumulative distribution functions of $\alpha$ and $\beta$, respectively. We have also obtained a closed-form solution for $d_{\mathrm{W},p}^{(\Rp,\Lambda_\infty)}$ in \Cref{thm:closed-form-w-infty-real}. We generalize these formulas to the case $d_{\mathrm{W},p}^{(\Rp,\Lambda_q)}$ when $q\in(1,\infty)$ and $q\leq p$ in \Cref{sec:closed form solution}. 
	\end{remark}

%%%%%%%%%%%%%%%%%%%%%%%%%%%%%%%%%%%%%%%%%%%%%%%%%%%%%%%%%%%%%%
	\section{Ultrametric Gromov-Wasserstein distances}\label{sec:ultrametric GW distance}
	
	In this section we investigate the properties of $\ugw{p}^\mathrm{sturm}$ as well as $\ugw{p}$, $1\leq p\leq\infty$, and study the relation between them.
	
	%%%%%%%%%%%%%%%%%%%
\subsection{Sturm's ultrametric Gromov-Wasserstein distance}\label{subsec:Sturms ultrametric GW distance}
	We begin by establishing several basic properties of $\usturm{p}$, $1\leq p\leq \infty,$ including a proof that $\usturm{p}$ is indeed a metric (or more precisely a $p$-metric) on the collection of compact ultrametric measure spaces $\mathcal{U}^w$.
	
	The definition of $\usturm{p}$ given in \Cref{eq:ultra Sturm} is clunky, technical and in general not easy to work with. Hence, the  first observation to make is the fact that $\usturm{p}$, $1\leq p \leq\infty$, shares a further property with $\dsturm{p}$: $\usturm{p}$ can be calculated by minimizing over pseudo-ultrametrics instead of isometric embeddings.
\begin{lemma}\label{lemma:pseudometric representation of usturm}
	Let $\X=\ummspaceX$ and $\Y=\ummspaceY$ be two ultrametric measure spaces. Let $\mathcal{D}^\mathrm{ult}(\uX,\uY)$ denote the collection of all pseudo-ultrametrics $u$ on the disjoint union $X\sqcup Y$ such that $u|_{X\times X} = \uX$ and $u|_{Y\times Y} = \uY$. Let $p\in[1,\infty]$. Then, it holds that \begin{equation}\label{eq:pseudometric def of usturm}\usturm{p}(\X,\Y)=\inf_{u \in \mathcal{D}^\mathrm{ult}(\uX,\uY)} \pseudoWasser{p}^{(X\sqcup Y,u)}(\muX,\muY), \end{equation}
	where $\pseudoWasser{p}^{(X\sqcup Y,u)}$ denotes the\emph{ Wasserstein pseudometric} of order $p$ defined in \Cref{eq:def Wasserstein pseudpmetric p} (resp. in \Cref{eq:def Wasserstein pseudpmetric infinity} for $p=\infty$) in \Cref{sec:Wasserstein pseudometric} of the supplement.
\end{lemma}
\begin{proof}
    The above lemma follows by the same arguments as Lemma 3.3 $(iii)$ in \cite{sturm2006geometry}.
\end{proof}
 \begin{remark}[Wasserstein pseudometric]
 The \emph{Wasserstein pseudometric} is a natural extension of the Wasserstein distance to pseudometric spaces  and has for example been studied in \citet{thorsley2008model}. In \Cref{sec:Wasserstein pseudometric} we carefully show that it is closely related to the Wasserstein distance on a canonically induced metric space. We further establish that the Wasserstein distance and the Wasserstein pseudometric share many relevant properties. Hence, we do not notationally distinguish between these two concepts. \end{remark}	
The representation of $\usturm{p}$, $1\leq p\leq \infty$, given by the above lemma is much more accessible and we first use it to establish the subsequent basic properties of $\usturm p$ (see \Cref{sec:proof of prop usturm_basic} for a full proof).

\begin{proposition}\label{prop:usturm_basic}
		Let $\X,\Y\in\mathcal{U}^w$. Then, the following holds:
		\begin{enumerate}
		\item For any $p\in[1,\infty]$, we always have that $\usturm{p}(\X,\Y)\geq \dsturm{p}(\X,\Y)$.
			\item For any $1\leq p\leq q\leq \infty$, we have that $\usturm{p}(\X,\Y)\leq\usturm{q}(\X,\Y) $.
			\item It holds that $\lim_{p\rightarrow\infty}\usturm{p}(\X,\Y)=\usturm{\infty}(\X,\Y). $
		\end{enumerate}
	\end{proposition}

Moreover, we use \Cref{lemma:pseudometric representation of usturm} to prove that $(\mathcal{U}^w,\usturm{p})$ is indeed a metric space. 
\begin{theorem}\label{thm:sturms um}
	$\ugw{p}^\mathrm{sturm}$ is a $p$-metric on the collection $\mathcal{U}^w$ of compact ultrametric measure spaces. In particular, when $p=\infty$, $\ugw{\infty}^\mathrm{sturm}$ is an ultrametric.
\end{theorem}
In order to increase the readability of this section we postpone the proof of \Cref{thm:sturms um} to \Cref{sec:proof of thm sturms um}. In the course of the proof, we will, among other things, verify the existence of optimal metrics and optimal couplings in \Cref{eq:pseudometric def of usturm} (see \Cref{prop:usturm_optimal}). Furthermore, it is important to note that the topology induced on $\mathcal{U}^w$ by $\usturm{p}$, $1\leq p\leq \infty$, is different from the one induced by $d_{\mathrm{GW},p}^{\mathrm{sturm}}$. This is well illustrated in the following example.

	\begin{example}[$\usturm{p}$ and $d_{\mathrm{GW},p}^{\mathrm{sturm}}$ induce different topologies]\label{ex:notation two point space}
		This example is an adaptation from \citet[Example 3.14]{memoli2019gromov}. For each $a>0$, denote by $\Delta_2(a)$ the two-point metric space with interpoint distance $a$. Endow with $\Delta_2(a)$ the uniform probability measure $\mu_a$ and denote the corresponding ultrametric measure space $\hat{\Delta}_2(a)$. Now, let $\X\coloneqq \hat{\Delta}_2(1)$ and let $\X_n\coloneqq\hat{\Delta}_2\left(1+\frac{1}{n}\right)$ for $n\in\mathbb N$. It is easy to check that for any $1\leq p \leq \infty$, $\dsturm{p}(\X,\X_n)=\frac{1}{2n}$ and $\usturm{p}(\X,\X_n)=
		2^{-\frac{1}{p}}(1+\frac{1}{n})$ where we adopt the convention that $1/\infty=0$. Hence, as $n$ goes to infinity $\X_n$ will converge to $\X$ in the sense of $\dsturm{p}$, but not in the sense of $\usturm{p}$, for any $1\leq p\leq\infty$.
		\end{example}

	%%%%%%%%%%%%%%%%%%%%%%%
	\subsubsection{Alternative representations of \texorpdfstring{$\usturm{p}$}{Sturm's Gromov-Wasserstein distance}}\label{subsubsec:alt rep for Sturms GW dist}
	In this subsection, we derive an alternative representation for $\ugw{p}^\mathrm{sturm}$ defined in \Cref{eq:ultra Sturm}. We mainly focus on the case $p<\infty$, however it turns out that the results also hold for $p=\infty$ (see \Cref{subsec:relation between ugw and usturm}).

Let $\X,\Y\in\mathcal{U}^w$ and recall the original definition of $\usturm{p}$, $p\in[1,\infty]$, given in \Cref{eq:ultra Sturm}, i.e., 
\[\usturm{p}(\X,\Y)=\inf_{Z,\phi,\psi} d_{\mathrm{W},p}^{(Z,u_Z)}(\varphi_\#\muY,\psi_\#\muY),\]
where $\phi:X\to Z$ and $\psi:Y\to Z$ are isometric embeddings into an ultrametric space $(Z,u_Z)$. It turns out that we only need to consider relatively few possibilities of mapping two ultrametric spaces into a common ultrametric space. Exemplarily, this is shown in \Cref{fig:common ultrametric space}, where we see two finite ultrametric spaces and two possibilities for a common ultrametric space $Z$.
	\begin{figure}
		\centering
		\includegraphics[width =0.8\textwidth]{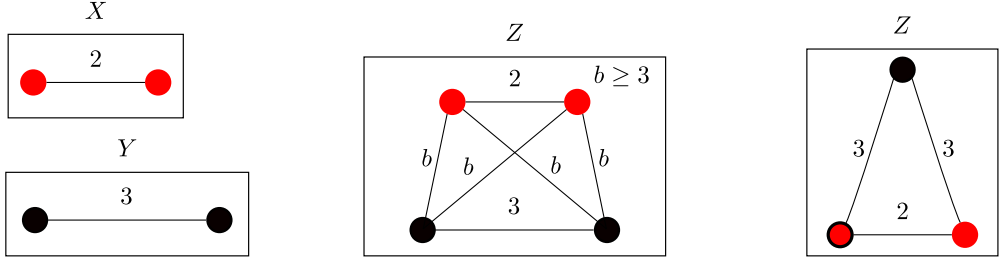}
		\caption{\textbf{Common ultrametric spaces:} Representation of the two kinds of ultrametric spaces $Z$ (middle and right) into which we can isometrically embed the spaces $X$ and $Y$ (left).} \label{fig:common ultrametric space}
	\end{figure}
	Indeed, it is straightforward to write down all reasonable embeddings and target spaces. We define the set
	\begin{equation}\label{eq:definition of A}
	\mathcal{A}\coloneqq\{(A,\varphi)\,|\,\emptyset\neq A\subseteq X \text{ is closed and } \varphi:A\hookrightarrow Y \text{ is an isometric embedding } \}.
	\end{equation}
	Clearly, $\mathcal{A}\neq\emptyset$, as it holds for each $x\in X$ that $\{(\{x\},\varphi_y)\}_{y\in Y}\subseteq\mathcal{A}$, where $\varphi_y$ is the map sending $x$ to $y\in Y$. Another possibility to construct elements in $\mathcal{A}$ is illustrated in the subsequent example.
	
	\begin{example}\label{ex:u=0 A}
	Let $\X,\Y\in\mathcal{U}^w$ be finite spaces and let $u\in\mathcal{D}^\mathrm{ult}(\uX,\uY)$. If $u^{-1}(0)\neq \emptyset$, we define $A\coloneqq\pi_X(u^{-1}(0))\subseteq X$, where $\pi_X:X\times Y\rightarrow X$ is the canonical projection. Then, the map $\varphi:A\rightarrow Y$ defined by sending $x\in A$ to $y\in Y$ such that $u(x,y)=0$ is an isometric embedding and in particular, $(A,\varphi)\in\mathcal{A}$.
	\end{example}
	Now, fix two compact spaces $\X,\Y\in\mathcal{U}^w$. Let $(A,\varphi)\in\mathcal{A}$ and let $Z_A=X\sqcup (Y\setminus \varphi(A))\subseteq X\sqcup Y$. Furthermore, define $u_{Z_A}:Z_A\times Z_A\rightarrow\Rp$ as follows:
	\begin{enumerate}
	    \item $u_{Z_A}|_{X\times X}\coloneqq u_X$ and $u_{Z_A}|_{Y\setminus \varphi(A)\times Y\setminus \varphi(A)}\coloneqq u_Y|_{Y\setminus \varphi(A)\times Y\setminus \varphi(A)}$;
	    \item For any $x\in A$ and $y\in Y\setminus \varphi(A)$ define $u_{Z_A}(x,y)\coloneqq \uY(y,\varphi(x)) $; 
	    \item For $x\in X\setminus A$ and $y\in Y\setminus \varphi(A)$ let $u_{Z_A}(x,y)\coloneqq \inf\{\max(\uX(x,a),\uY(\varphi(a),y))\,|\,a\in A\} $;
	    \item For any $x\in X$ and $y\in Y\setminus\varphi(A)$, $u_{Z_A}(y,x)\coloneqq u_{Z_A}(x,y). $
	\end{enumerate}
Then, $(Z_A,u_{Z_A})$ is an ultrametric space such that $X$ and $Y$ can be mapped isometrically into $Z_A$  (see \cite[Lemma 1.1]{zarichnyi2005gromov}). Let $\phi^X_{(A,\varphi)}$ and $\psi^Y_{(A,\varphi)}$ denote the corresponding isometric embeddings of $X$ and $Y$, respectively.
This allows us to derive the following statement, whose proof is postponed to \Cref{sec:proof of compact usturm a phi}.

\begin{theorem}\label{thm:compact usturm A Phi representation}
	Let $\X,\Y\in\mathcal{U}^w$. Then, we have for each $p\in[1,\infty)$ that
\begin{equation}\label{eq:alternative representation ustum}
\usturm{p}(\X,\Y)=\inf_{(A,\varphi)\in\mathcal{A}}d_{\mathrm{W},p}^{Z_A}\left({\left(\phi^X_{(A,\varphi)}\right)}_\#\muX,{\left(\psi^Y_{(A,\varphi)}\right)}_\#\muY\right).\end{equation}
	\end{theorem}
	
\begin{remark}\label{rem:computation of usturm}
Let $\X$ and $\Y$ be two finite ultrametric measure spaces. The representation of $\ugw{p}(\X,\Y)$, $1\leq p\leq \infty$ given by \Cref{thm:compact usturm A Phi representation} is very explicit and recasts the computation of $\ugw{p}(\X,\Y)$, $1\leq p\leq \infty$, as a combinatorial problem. In fact, as $\X$ and $\Y$ are finite, the set $\mathcal{A}$ in \Cref{eq:alternative representation ustum} can be further reduced. More precisely, we demonstrate in \Cref{sec:proof of compact usturm a phi} (see \Cref{coro:usturm A Phi representation}) that it is sufficient to infimize over the set of all \emph{maximal pairs}, denoted by $\mathcal{A}^*$. Here, a pair $(A,\varphi_1)\in\mathcal{A}$ is denoted as \emph{maximal}, if for all pairs $(B,\varphi_2)\in \mathcal{A}$ with $A\subseteq B$ and $\varphi_2|_A=\varphi_1$ it holds $A=B$.
Using the ultrametric Gromov-Hausdorff distance (see \Cref{eq:Gromov Hausdorff ultrametric}) it is possible to determine if two ultrametric spaces are isometric in polynomial time \cite[Theorem 5.7]{memoli2019gromov}. However, this is clearly not sufficient to identify all $(A,\varphi)\in\mathcal{A}^*$ in polynomial time. Especially, for a given, viable $A\subseteq X$, there are usually multiple ways to define the corresponding map $\varphi$. Furthermore, we have for $1\leq p<\infty$ neither been able to further restrict the set $\mathcal{A}^*$ nor to identify the optimal $(A^*,\varphi^*)$. This just leaves a brute force approach which is computationally not feasible. On the other hand, for $p=\infty$ we are able to explicitly construct the optimal pair $(A^*,\varphi^*)$ (see \Cref{thm:optimal A and varphi (usturm)}).
\end{remark}
	
	%%%%%%%%%%%%%%%%%%
	\subsection{The ultrametric Gromov-Wasserstein distance}\label{subsec:the ultrametric GW distance}
	In the following, we consider basic properties of $\ugw{p}$ and prove the analogue of \Cref{thm:sturms um}, i.e., we verify that also $\ugw{p}$ is a $p$-metric, $1\leq p\leq \infty$, on the collection of ultrametric measure spaces. 

	The subsequent proposition collects three  basic properties of $\ugw{p}$ which are also shared by $\usturm{p}$ (cf. \Cref{prop:usturm_basic}). We refer to \Cref{sec:proof:prop:ugw-properties} for its proof.
	\begin{proposition}\label{prop:ugw-properties}
		Let $\X,\Y\in\mathcal{U}^w$. Then, the following holds:
		\begin{enumerate}
		    \item For any $p\in[1,\infty]$, we always have that $\ugw{p}(\X,\Y)\geq \dgw{p}(\X,\Y)$.
			\item For any $1\leq p\leq q\leq \infty$, it holds $\ugw{p}(\X,\Y)\leq\ugw{q}(\X,\Y) $;
			\item We have that $\lim_{p\rightarrow\infty}\ugw{p}(\X,\Y)=\ugw{\infty}(\X,\Y).$
		\end{enumerate}
	\end{proposition}

Next, we verify that $\ugw{p}$ is indeed a metric on the collection of ultrametric measure spaces.
\begin{theorem}\label{thm:ugw-p-metric}
The ultrametric Gromov-Wasserstein distance $\ugw{p}$ is a $p$-metric on the collection $\mathcal{U}^w$ of compact ultrametric measure spaces. In particular, when $p=\infty$, $\ugw{\infty}$ is an ultrametric.
\end{theorem}
The full proof of \Cref{thm:ugw-p-metric}, which is based on the existence of optimal couplings in \Cref{eq:def uGW} (see \Cref{{prop:ugw-ext-opt}}), is postponed to \Cref{sec:proof of thm ugw-p-metric}.

\begin{remark}[$\ugw{p}$ and $\dgw{p}$ induce different topologies]
Reconsidering \Cref{ex:notation two point space}, it is easy to verify that in this setting $\ugw{p}(\X,\X_n)=2^{-\frac{1}{p}}\left(1+\frac{1}{n}\right)$ while $\dgw{p}(\X,\X_n)=\frac{1}{2^{1/p}n}$, $1\leq p\leq \infty$. Hence, just like $\usturm{p}$ and $\dsturm{p}$, $\ugw{p}$ and $\dgw{p}$ do not induce the same topology on $\mathcal{U}^w$. This result can also be obtained from \Cref{sec:topology and geodesic properties} where we derive that $\ugw{p}$ and $\usturm{p}$ give rise to the same topology. 
\end{remark}

\begin{remark}\label{rem:computational complexity ugw}
By the same arguments as for $\dgw{p}$, $1\leq p<\infty$, \citep[Sec. 7]{memoli2011gromov}, it follows that for two finite ultrametric
measure spaces $\X$ and $\Y$ the computation of $\ugw{p}(\X, \Y)$, $1\leq p<\infty $, boils down to solving a (non-convex) quadratic program. This is in general NP-hard \citep{pardalos1991quadratic}. On the other hand, for $p=\infty$, we will derive a polynomial time algorithm to determine $\ugw{\infty}(\X, \Y)$ (cf. \Cref{subsubsec:alternative repesentation of ugw infty}).    
\end{remark}
\subsubsection{Alternative representations of \texorpdfstring{$\ugw\infty$}{the ultrametric Gromov-Wasserstein distance}}\label{subsubsec:alternative repesentation of ugw infty} In the following, we will derive an alternative representation of $\ugw\infty$ that resembles the one of $\ugh$ derived in \cite[Theorem 5.7]{memoli2019gromov}. It also leads to a polynomial time algorithm for the computation of $\ugw{\infty}$. For this purpose, we define the \emph{weighted quotient} of an ultrametric measure space. Let $\X=\ummspaceX\in\mathcal{U}^w$ and let $t\geq 0$. Then, the \emph{weighted quotient} of $\X$ at level $t$, is given as $\X_t=(X_t,u_{X_t},\mu_{X_t})$, where $(X_t, u_{X_t})$ is the quotient of the ultrametric space $(X,u_X)$ at level $t$ (see \Cref{sec:ultrametric Gromov-Hausdorff}) and $\mu_{X_t}\in\mathcal{P}(X_t)$ is the push forward of $\muX$ under the canonical quotient map $Q_t:(X,u_X)\rightarrow(X_t,u_{X_t})$ sending $x$ to $[x]_t$ for $x\in X$. \Cref{fig:accumulated measure spaces}  illustrates the weighted quotient in a simple example.
	\begin{figure}
		\centering
			\includegraphics[width=0.5\textwidth]{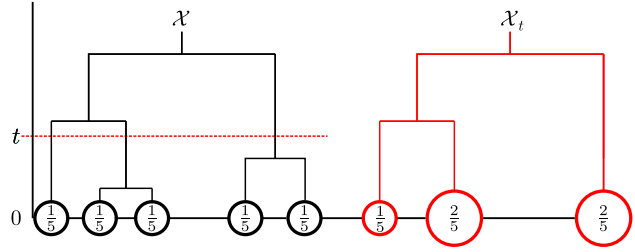}
		\caption{\textbf{Weighted Quotient:} An ultrametric  measure space (black) and its weighted quotient at level $t$ (red).} \label{fig:accumulated measure spaces}
	\end{figure}
Based on this definition, we show the following theorem, whose proof is postponed to \Cref{sec:proof of thm ugw-infty-eq}.
	\begin{theorem}\label{thm:ugw-infty-eq}
		Let $\X=\ummspaceX$ and $\Y=\ummspaceY$ be two compact ultrametric measure spaces. 
		Then, it holds that
		\[u_{\mathrm{GW},\infty}(\X,\Y)=\min\left\lbrace t\geq 0 \,|\,\X_t \cong_w \Y_t\right\rbrace.\]
	\end{theorem}
	\begin{remark} The weighted quotients $\X_t$ and $\Y_t$ can be considered as vertex weighted, rooted trees and thus it is possible to verify whether $\X_t \cong_w \Y_t$ in polynomial time \citep{aho1974design}. In consequence, we obtain an polynomial time algorithm for the calculation of $\ugw{\infty}$. See \Cref{subsubsec:p equals infty} for details.
	\end{remark}
	
	The representations of $\ugh$  in \Cref{thm:ultrametric GH-distance} and $\ugw{\infty}$ in \Cref{thm:ugw-infty-eq} strongly resemble themselves. As a direct consequence of both \Cref{thm:ultrametric GH-distance} and \Cref{thm:ugw-infty-eq}, we obtain the following comparison between the two metrics
  \begin{corollary}\label{coro:ugw>ugh}
	Let $\X,\Y\in\mathcal{U}^w$. Then, it holds that
		\begin{equation}\label{eq:ugw geq ugh}
	u_{\mathrm{GW},\infty}(\X,\Y)\geq \ugh(X,Y).	\end{equation}
	\end{corollary}
	
The inequality in \Cref{eq:ugw geq ugh} is sharp and we illustrate this as follows. By \citet[Corollary 5.8]{memoli2019gromov} we know that if the considered ultrametric spaces $(X,u_X)$ and $(Y,u_Y)$ have different diameters (w.l.o.g. $\diam{X}<\diam{Y}$), then $\ugh(X,Y)=\diam{Y}$. The same statement also holds for $\ugw{\infty}$
\begin{corollary}\label{coro:uGW trivial case}
		Let $\X,\Y\in\mathcal{U}^w$ be such that $\diam{X}<\diam{Y}$. Then,
		\[\ugw{\infty}(\X,\Y)=\diam{Y}=\ugh(X,Y).\]
	\end{corollary}
	\begin{proof}
		The rightmost equality follows directly from Corollary 5.8 of \citet{memoli2019gromov}. As for the leftmost equality, let $t\coloneqq\diam Y$, then it is obvious that $\X_t\cong_w *\cong_w \Y_t$, where $*$ denotes the one point ultrametric measure space. Let $s\in(\diam X,\diam Y)$, then $\X_t\cong_w *$ whereas $\Y\not\cong_w*$. By \Cref{thm:ugw-infty-eq}, $\ugw{\infty}(\X,\Y)=t=\diam{Y}$.
	\end{proof}

%%%%%%%%%%%%%%%%%%%%%%%%%%%%%%%%%%%%%%%%%%%%%%
	
	\subsection{The relation between \texorpdfstring{$\ugw{p}$}{the ultrametric Gromov-Wasserstein distance} and \texorpdfstring{$\usturm{p}$}{Sturm's ultrametric Gromov-Wasserstein distance}}\label{subsec:relation between ugw and usturm}
	In this section, we study the relation of $\usturm{p}$ and $\ugw{p}$, $1\leq p\leq, \infty$ and establish the topological equivalence between the two metrics. 
	
	\subsubsection{Lipschitz relation} We first study the Lipschitz relation between $\usturm{p}$ and $\ugw{p}$. For this purpose, we have to distinguish the cases $p<\infty$ and $p =\infty$.

\emph{The case \texorpdfstring{$p<\infty$}{p<infty}.} 
We start the consideration of this case by proving that it is essentially enough to consider the case $p=1$ (see \Cref{thm:snow-ugw}). To this end, we need to introduce some notation. For each $\alpha>0$, we define a function $S_\alpha:\Rp\rightarrow\Rp$ by $x\mapsto x^\alpha$. Given an ultrametric space $(X,\uX)$ and $\alpha>0$, we abuse the notation and denote by $S_\alpha(X)$ the new space $(X,S_\alpha\circ\uX)$. It is obvious that $S_\alpha(X)$ is still an ultrametric space. This transformation of metric spaces is also known as the \emph{snowflake transform} \citep{david1997fractured}. 
	Let $\X=\ummspaceX$ and $\Y=\ummspaceY$ denote two ultrametric measure spaces. Let $1\leq p<\infty$. We denote by $S_p(\X)$ the ultrametric measure space $(X,S_p\circ\uX,\muX)$. The snowflake transform can be used to relate $\ugw{p}(\X,\Y)$ as well as $\usturm{p}(\X,\Y)$ with $\ugw{1}(S_p(\X),S_p(\Y))$ and $\usturm{1}(S_p(\X),S_p(\Y))$, respectively. 
	\begin{theorem}\label{thm:snow-ugw}
		Let $\X,\Y\in\mathcal{U}^w$ and let $p\in [1,\infty)$. Then, we obtain
		\[\big(\ugw{p}(\X,\Y)\big)^p=\ugw{1}(S_p(\X),S_p(\Y))~ \text{ and }~\big(\ugw{p}^\mathrm{sturm}(\X,\Y)\big)^p=\ugw{1}^\mathrm{sturm}(S_p(\X),S_p(\Y)). \]	\end{theorem}
We give full proof of \Cref{thm:snow-ugw} in \Cref{sec:proof of thm snow-ugw}. Based on this result, we can directly relate the metrics $\ugw{p}$ and $\usturm{p}$ by only considering the case $p=1$ and prove the following \Cref{thm:ugw<ugw-sturm} (see \Cref{sec:proof of thm ugw<ugw-sturm} for a detailed proof).
\begin{theorem}\label{thm:ugw<ugw-sturm}
		Let $\X,\Y\in\mathcal{U}^w$. Then, we have for $p\in[1,\infty)$ that 
		\[\ugw{p}(\X,\Y)\leq 2^\frac{1}{p}\,\ugw{p}^\mathrm{sturm}(\X,\Y).\]
	\end{theorem}
The subsequent example verifies that the coefficient in \Cref{thm:ugw<ugw-sturm} is tight.
	\begin{example} For each $n\in\mathbb N$, let $\X_n$ be the three-point space $\Delta_3(1)$ (i.e. the 3-point metric labeled by $\{x_1,x_2,x_3\}$ where all distances are 1) with a probability measure $\muX^n$ such that $\muX^n(x_1)=\muX^n(x_2)=\frac{1}{2n}$ and $\muX^n(x_3)=1-\frac{1}{n}$. Let $Y=*$ and $\muY$ be the only probability measure on $Y$. Then, it is routine (using \Cref{prop:ugw and one point space} from \Cref{sec:ugw and one point space}) to check that $\ugw{1}(\X_n,\Y)=\frac{2}{n}\left(1-\frac{3}{4n}\right)$ and $\ugw{1}^\mathrm{sturm}(\X_n,\Y)=\frac{1}{n}$. Therefore, we have 
		\[\lim_{n\rightarrow\infty}\frac{\ugw{1}(\X_n,\Y)}{\ugw{1}^\mathrm{sturm}(\X_n,\Y)}=2. \]
	\end{example}
	
\begin{example}[$\usturm{p}$ and $\ugw{p}$ are not bi-Lipschitz equivalent]\label{ex:non bi lipschitz}
Following \cite[Remark 5.17]{memoli2011gromov}, we verify in \Cref{sec: ex non bi lipschitz} that for any positive integer $n$
\[\usturm{p}\lc\hat{\Delta}_n(1),\hat{\Delta}_{2n}(1)\rc\geq \frac{1}{4}\,\,\text{and}\,\,\ugw{p}\lc\hat{\Delta}_n(1),\hat{\Delta}_{2n}(1)\rc\leq \lc\frac{3}{2n}\rc^\frac{1}{p}.\]
Here, $\hat{\Delta}_n(1)$ denotes the $n$-point metric measure space with interpoint distance $1$ and the uniform probability measure.
Thus, there exists no constant $C>0$ such that $\usturm{p}(\X,\Y)\leq C\cdot\ugw{p}(\X,\Y)$ holds for every input spaces $\X$ and $\Y$. Hence, $\usturm{p}$ and $\ugw{p}$ are not bi-Lipschitz equivalent.
\end{example}

\emph{The case \texorpdfstring{$p=\infty$}{infinity}.}
Next, we consider the relation between $\usturm{\infty}$ and $\ugw{\infty}$. By taking the limit $p\rightarrow\infty$ in \Cref{thm:ugw<ugw-sturm}, one might expect that $\usturm{\infty}\geq \ugw{\infty}$. In fact, we prove that the equality holds (for the full proof see \Cref{sec:proof of thm ugw infty and sturm ugw infty}).
	
	\begin{theorem}\label{thm:ugw infty and sturm ugw infty}
		Let $\X,\Y\in\mathcal{U}^w$.
		Then, it holds that \[\usturm{\infty}(\X,\Y)=\ugw{\infty}(\X,\Y).\]
	\end{theorem}

	One application of \Cref{thm:ugw infty and sturm ugw infty} is to explicitly derive the minimizing pair $(A,\phi)\in \mathcal{A}^*$ in \Cref{eq:usturm_maxiam_pair} for $p=\infty$ (see \Cref{sec:proof of thm optimal A and varphi (usturm)} for an explicit construction):
	\begin{theorem}\label{thm:optimal A and varphi (usturm)}
		Let $\X,\Y\in\mathcal{U}^w$. Let $s\coloneqq u^\mathrm{sturm}_{\mathrm{GW},\infty}(\X,\Y)$ and assume that $s>0$. Then, there exists $(A,\phi)
		\in\mathcal{A}$ defined in \Cref{eq:definition of A}
		such that
		\[\usturm\infty(\X,\Y)=d_{\mathrm{W},\infty}^{Z_A}(\muX,\muY),\]
		where $Z_A$ denotes the ultrametric space defined in \Cref{subsubsec:alt rep for Sturms GW dist}.
	\end{theorem}
	
	\subsubsection{Topological equivalence between \texorpdfstring{$\ugw p$}{the ultrametric Gromov-Wasserstein distance} and \texorpdfstring{$\usturm p$}{Sturm's ultrametric Gromov-Wasserstein distance}} \citet{memoli2011gromov} proved the topological equivalence between $\dgw p$ and $\dsturm p$. We establish an analogous result for $\ugw p$ and $\usturm p$. To this end, we recall  the \emph{modulus of mass distribution}.
	
	\begin{definition}[{\citet[Def. 2.9]{{greven2009convergence}}}]\label{def:modulus of mass distribution}
	Given $\delta>0$ we define \emph{the modulus of mass distribution} of $\X\in\mathcal{U}^w$ as
	\begin{equation}\label{eq:modulus}
	    v_\delta(\X)\coloneqq\inf\left\{ \eps>0|\,\muX\left(\left\{x:\,\muX\left(B_\eps^\circ(x)\right)\leq\delta\right\}\right)\leq \eps\right\},
	\end{equation}
	where $B_\eps^\circ(x)$ denotes the \emph{open} ball centered at $x$ with radius $\eps$.
	\end{definition}
	
	We note that $v_\delta(\X)$ is non-decreasing, right-continuous and bounded above by 1. Furthermore, it holds that $\lim_{\delta\searrow 0}v_\delta(\X)=0$ \citep[Lemma 6.5]{greven2009convergence}. With \Cref{def:modulus of mass distribution} at hand, we derive the following theorem.
	\begin{theorem}\label{thm:equivalence}
	Let $\X,\Y\in\mathcal{U}^w$, $p\in[1,\infty)$ and $\delta\in\left(0,\frac{1}{2}\right)$. Then, whenever $\ugw p(\X,\Y)<\delta^5$ we have
	\[\usturm p(\X,\Y)\leq \left(4\cdot\min(v_\delta(\X),v_\delta(\Y))+\delta\right)^\frac{1}{p}\cdot M, \]
	where $M\coloneqq 2\cdot\max(\diam X,\diam Y)+54$.
	\end{theorem}
	
	\begin{remark}
    Since it holds that $\lim_{\delta\searrow0}v_\delta(\X)=0$ and that $2^{-{1}/{p}}\usturm{p}\geq\ugw{p}$ (see \Cref{thm:ugw<ugw-sturm}), the above theorem gives the topological equivalence between $\ugw{p}$ and $\usturm {p}$, $1\leq p<\infty$ (the topological equivalence between $\usturm{\infty}$ and $\ugw{\infty}$ holds trivially thanks to \Cref{thm:ugw infty and sturm ugw infty}).
\end{remark}

	The proof of the \Cref{thm:equivalence} follows the same strategy used for proving Proposition 5.3 in \cite{memoli2011gromov} and we refer to \Cref{app:proof of equivalence} for the details.

	%%%%%%%%%%%%%%%%%%%%%%%%%%%%%%
	\subsection{Topological and geodesic properties}\label{sec:topology and geodesic properties}
	In this section, we consider the topology induced by $\ugw{p}$ and $\usturm{p}$ on $\mathcal{U}^w$ and discuss the geodesic properties of both $\ugw p$ and $\usturm p$ for $1\leq p\leq \infty$.

 	\subsubsection{Completeness and separability} We study completeness and separability of the two metrics $\ugw p$ and $\usturm p$, $1\leq p\leq \infty$, on $\mathcal{U}^w$. To this end, we derive the subsequent theorem whose proof is postponed to \Cref{sec:proof of thm complete and separable}.

	\begin{theorem}\label{thm: complete and separable}
	\begin{enumerate}
	    \item For $p\in[1,\infty)$, the metric space $(\mathcal{U}^w,\ugw p)$ is neither complete nor separable.
	    \item For $p\in[1,\infty)$, the metric space $\left(\mathcal{U}^w,\usturm p\right)$ is neither complete nor separable.
	    \item $(\mathcal{U}^w,\ugw \infty)=(\mathcal{U}^w,\usturm \infty)$ is complete but not separable.
	\end{enumerate}
	 
	\end{theorem}

	\subsubsection{Geodesic property}
	
	A \emph{geodesic} in a metric space $(X,d_X)$ is a continuous function $\gamma:[0,1]\rightarrow X$ such that for each $s,t\in[0,1]$, $d_X(\gamma(s),\gamma(t))=|s-t|\cdot d_X(\gamma(0),\gamma(1))$. We say a metric space is geodesic if for any two distinct points $x,x'\in X$, there exists a geodesic $\gamma:[0,1]\rightarrow X$ such that $\gamma(0)=x$ and $\gamma(1)=x'$. For any $p\in[1,\infty)$, the notion of $p$-geodesic is introduced in \cite{memoli2019gromov}: A $p$-geodesic in a metric space $(X,d_X)$ is a continuous function $\gamma:[0,1]\rightarrow X$ such that for each $s,t\in[0,1]$, $d_X(\gamma(s),\gamma(t))=|s-t|^{1/p}\cdot d_X(\gamma(0),\gamma(1))$. Similarly, we say a metric space is $p$-geodesic if for any two distinct points $x,x'\in X$, there exists a $p$-geodesic $\gamma:[0,1]\rightarrow X$ such that $\gamma(0)=x$ and $\gamma(1)=x'$. Note that a $1$-geodesic is a usual geodesic and a $1$-geodesic space is a usual geodesic space. The subsequent theorem establishes ($p$-)geodesic properties of $\left(\mathcal{U}^w,\usturm p\right)$ for $p\in[1,\infty)$. A full proof is given in \Cref{sec:proof of prop usturm 1 geodesic}.
	\begin{theorem}\label{prop:usturm 1 geodesic}
For any $p\in[1,\infty)$, the space $\left(\mathcal{U}^w,\usturm p\right)$ is $p$-geodesic.
	\end{theorem}
	
\begin{remark}
Due to the fact that a $p$-geodesic space cannot be geodesic when $p>1$ (cf. \Cref{lm:p metric not geodesic}), $\left(\mathcal{U}^w,\usturm p\right)$ is not geodesic for all $p> 1$.
\end{remark}

\begin{remark}
Though the geodesic properties of $\left(\mathcal{U}^w,\usturm p\right)$, $1\leq p< \infty$ are clear, we remark that geodesic properties of $\left(\mathcal{U}^w,\ugw p\right)$, $1\leq p<\infty$, still remain unknown to us.
\end{remark}

\begin{remark}[The case $p=\infty$]
Being an ultrametric space itself (cf. \Cref{thm:ugw-p-metric}), $\left(\mathcal{U}^w,\ugw \infty\right)$ ($=\left(\mathcal{U}^w,\usturm \infty\right)$) is \emph{totally disconnected}, i.e., any subspace with at least two elements is disconnected \cite{semmes2007introduction}. This in turn implies that each continuous curve in $\left(\mathcal{U}^w,\ugw \infty\right)$ is constant. Therefore, $\left(\mathcal{U}^w,\ugw \infty\right)$ is not a $p$-geodesic space for any $p\in[1,\infty)$.
\end{remark}

	%%%%%%%%%%%%%%%%%%%%%%%%%%%
	
	\section{Lower bounds for \texorpdfstring{$\ugw{p}$}{the ultrametric Gromov-Wasserstein distance}}\label{sec:lower bounds}
	Let $\X=\ummspaceX$ and $\Y=\ummspaceY$ be two ultrametric measure spaces. The metrics $\usturm{p}$ and $\ugw{p}$ respect the ultrametric structure of the spaces $\X$ and $\Y$. Thus, one would hope that comparing ultrametric measure spaces with $\usturm{p}$ or $\ugw{p}$ is more meaningful than doing it with the usual Gromov-Wasserstein distance or Sturm's distance. Unfortunately, for $p<\infty$, the computation of both $\usturm{p}$ and $\ugw{p}$ is complicated and for $p=\infty$ both metrics are extremely sensitive to differences in the diameters of the considered spaces (see \Cref{coro:uGW trivial case}). Thus, it is not feasible to use these metrics in many applications. However, we can derive meaningful lower bounds for $\ugw{p}$ (and hence also for $\usturm{p}$) that resemble those of the Gromov-Wasserstein distance. Naturally, the question arises whether these lower bounds are better/sharper than the ones of the usual Gromov-Wasserstein distance in this setting. This question is addressed throughout this section and will be readdressed in \Cref{sec:computational aspects} as well as \Cref{sec:phylogenetic tree shapes}.
	
	In \cite{memoli2011gromov}, the author introduced three lower bounds for $\dgw{p}$ that are computationally less expensive than the calculation of $\dgw{p}$. We will briefly review these three lower bounds and then define candidates for the corresponding lower bounds for $\ugw{p}$.  In the following, we always assume $p\in[1,\infty]$.\\
	
	%%%%%%%%%%%%%%%%%%%	
	\paragraph{\textbf{{First lower bound}}} Let $s_{X,p}:X\rightarrow\Rp$, $x\mapsto\norm{u_X(x,\cdot)}_{L^p(\mu_X)}$. Then, the first lower bound $\dFLB{p}(\X,\Y)$ for $\dgw{p}(\X,\Y)$ is defined as follows
	\[\dFLB{p}(\X,\Y)\coloneqq\frac{1}{2}\inf_{\mu\in\mathcal{C}(\mu_X,\mu_Y)}\norm{\Lambda_1(s_{X,p}(\cdot),s_{Y,p}(\cdot))}_{L^p(\mu)}.\]
	
Following our intuition of replacing $\Lambda_1$ with $\Lambda_\infty$, we define the ultrametric version of $\dFLB{}$ as
	\[\uFLB{p}(\X,\Y)\coloneqq\inf_{\mu\in\mathcal{C}(\mu_X,\mu_Y)}\norm{\Lambda_\infty(s_{X,p}(\cdot),s_{Y,p}(\cdot))}_{L^p(\mu)}.\]
	
	%%%%%%%%%%%%%%%%%%%
	\paragraph{\textbf{Second lower bound}}
	The second lower bound $\dSLB{p}(\X,\Y)$ for $\dgw{p}(\X,\Y)$ is given as
	\[\mathbf{SLB}_p(\X,\Y)\coloneqq\frac{1}{2}\inf_{\gamma\in\mathcal{C}(\muX\otimes \muX,\muY\otimes\muY)}\norm{\Lambda_1(u_X,u_Y)}_{L^p(\gamma)}.\]

	Thus, we define the ultrametric second lower bound between two ultrametric measure spaces $\X$ and $\Y$ as follows:
	
	\[\uSLB{p}(\X,\Y)\coloneqq\inf_{\gamma\in\mathcal{C}(\muX\otimes \muX,\muY\otimes\muY)}\norm{\Lambda_\infty(\uX,\uY)}_{L^p(\gamma)}.\]
	
	%%%%%%%%%%%%%%%%%%%%%%%
	\paragraph{\textbf{Third lower bound}}
	Before we introduce the final lower bound, we have to define several functions. First, let $\Gamma^1_{X,Y}:X\times Y\times X\times Y\rightarrow\Rp$, $(x,y,x',y')\mapsto\Lambda_1(u_X(x,x'),u_Y(y,y'))$ and let $\Omega_p^1:X\times Y\rightarrow\Rp$, $p\in[1,\infty]$, be given by
	\[\Omega_p^1(x,y)\coloneqq\inf_{\mu\in\mathcal{C}(\mu_X,\mu_Y)}\norm{\Gamma_{X,Y}^1(x,y,\cdot,\cdot)}_{L^p(\mu)}. \]
	Then, the third lower bound $\dTLB{p}$ is given as
	\[\dTLB{p}(\X,\Y)\coloneqq\frac{1}{2}\inf_{\mu\in\mathcal{C}(\mu_X,\mu_Y)}\norm{\Omega_p^1(\cdot,\cdot)}_{L^p(\mu)}. \]
	Analogously to the definition of previous ultrametric versions, we define $\Gamma^\infty_{X,Y}:X\times Y\times X\times Y\rightarrow\Rp$, $(x,y,x',y')\mapsto\Lambda_\infty(u_X(x,x'),u_Y(y,y'))$. Further, for $p\in[1,\infty]$, let $\Omega_p^\infty:X\times Y\rightarrow\Rp$ be given by
	\[\Omega_p^\infty(x,y)\coloneqq\inf_{\mu\in\mathcal{C}(\mu_X,\mu_Y)}\norm{\Gamma_{X,Y}^\infty(x,y,\cdot,\cdot)}_{L^p(\mu)}. \]
	Then, the ultrametric third lower bound between two ultrametric measure spaces $\X$ and $\Y$ is defined as
	\[\uTLB{p}(\X,\Y)\coloneqq\inf_{\mu\in\mathcal{C}(\mu_X,\mu_Y)}\norm{\Omega_p^\infty(\cdot,\cdot)}_{L^p(\mu)}. \]
	%%%%%%%%%%%%%%%%%%%%%%%%%%
	\subsection{Properties and computation of the lower bounds}
		Next, we examine the quantities $\uFLB{},\uSLB{}$ and $\uTLB{}$ more closely. Since $\Lambda_\infty(a,b)\geq\Lambda_1(a,b)=|a-b|$ for any $a,b\geq 0$, it is easy to conclude that $\uFLB{p}\geq \dFLB{p}$, $\uSLB{p}\geq \dSLB{p}$ and $\uTLB{p}\geq \dTLB{p}$. Moreover, the three ultrametric lower bounds satisfy the following theorem (for a complete proof see \Cref{sec:proof of thm comparison with original}).
	\begin{theorem}\label{thm:comparison with original}
		Let $\X,\Y\in\mathcal{U}^w$ and let $p\in[1,\infty]$.
		\begin{enumerate}
		\item $\ugw{\infty}(\X,\Y)\geq \uFLB{\infty}(\X,\Y)$.
			\item $\ugw{p}(\X,\Y)\geq \uTLB{p}(\X,\Y)\geq\uSLB{p}(\X,\Y)$.
		\end{enumerate}
	\end{theorem}
	\begin{remark}
		Interestingly, it turns out that $\uFLB{p}$ is not a lower bound of $\ugw{p}$ in general when $p<\infty$. For example, let $X=\{x_1,x_2,\ldots,x_n\}$ and $Y=\{y_1,\ldots,y_n\}$ and define $\uX$ such that $\uX(x_1,x_2)=1$ and $\uX(x_i,x_j)=2\delta_{i\neq j}$ for $(i,j)\neq (1,2)$, $(i,j)\neq (2,1)$ and $i,j=1,\ldots,n$. Let $\uY(y_i,y_j)=2\delta_{i\neq j}$, $i,j=1,\ldots,n$, and let $\muX$ and $\muY$ be uniform measures on $X$ and $Y$, respectively. Then, $\ugw{1}(\X,\Y)\leq \frac{4}{n^2}$ whereas $\uFLB{1}(\X,\Y)=\frac{4n-4}{n^2}$ which is greater than $\ugw{1}(\X,\Y)$ as long as $n>2$. Moreover, we have in this case that $\uFLB{1}(\X,\Y)=O\left(\frac{1}{n}\right)$ whereas $\ugw{1}(\X,\Y)=O\left(\frac{1}{n^2}\right)$. Hence, there exists no constant $C>0$ such that $\uFLB{1}\leq C\cdot\ugw 1$ in general.
	\end{remark}
	\begin{remark}
	    There exist ultrametric measure spaces $\X$ and $\Y$ such that $\uTLB{p}(\X,\Y)=0$ whereas $\ugw{p}(\X,\Y)>0$ (examples described in \cite[Figure 8]{memoli2011gromov} will serve the purpose). Furthermore, there are spaces $\X$ and $\Y$ such that $\uSLB{p}(\X,\Y)=0$ whereas $\uTLB{p}(\X,\Y)>0$ (see \Cref{sec:uSLB zero uTLB greater zero}). The analogous statement holds true for $\dTLB{p}$ and $\dSLB{p}$, which  are nevertheless useful in various applications (see e.g. \cite{gellert2019substrate}).
	\end{remark}
	From the structure of $\uSLB{p}$ and $\uTLB{p}$ it is obvious that their computations leads to different optimal transport problems (see e.g.  \cite{villani2003topics}). However, in analogy to \citet[Theorem 3.1]{chowdhury2019gromov} we can rewrite $\uSLB{p}$ and $\uTLB{p}$ in order to further simplify their computation. The full proof of the subsequent proposition is given in \Cref{sec:proof of prop flb-slb-w-form}.
	\begin{proposition}\label{prop:flb-slb-w-form}
		Let $\X,\Y\in\mathcal{U}^w$ and let $p\in[1,\infty]$. Then, we find that
		\begin{enumerate}
			\item $ \uSLB{p}(\X,\Y)=d_{\mathrm{W},p}^{(\Rp,\Lambda_\infty)}\lc(\uX)_\#(\muX\otimes\muX),(\uY)_\#(\muY\otimes\muY)\rc;$
			\item For each $x,y\in X\times Y$, $\Omega_p^\infty(x,y)= d_{\mathrm{W},p}^{(\Rp,\Lambda_\infty)}\left(u_X(x,\cdot)_\#\muX,u_Y(y,\cdot)_\#\muY\right)$. 
			
            \end{enumerate}
	\end{proposition}
	\begin{remark}
	Since we have by \Cref{thm:closed-form-w-infty-real} an explicit formula for the Wasserstein distance on $(\Rp,\Lambda_\infty)$ between finitely supported probability measures, these alternative representations of the lower bound $\uSLB{p}$ and the cost functional $\Omega_p^\infty$ drastically reduce the computation time of $\uSLB{p}$ and $\uTLB{p}$, respectively. In particular, we note that this allows us to compute $\uSLB{p}$, $1\leq p\leq \infty$, between finite ultrametric measure spaces $\X$ and $\Y$ with $|X|=m$ and $|Y|=n$ in $O((m\vee n)^2)$ steps.
	\end{remark}

\Cref{prop:flb-slb-w-form} allows us to direclty compare the two lower bounds $\uSLB{1}$ and $\dSLB{1}$. 

\begin{corollary}\label{coro:representation of SLB1}
For any finite ultrametric measure spaces $\X$ and $\Y$, we have that 
\begin{equation}\label{eq:slbu-slb}
    \uSLB{1}(\X,\Y)=\mathbf{SLB}_1(\X,\Y)+\frac{1}{2}\int_\mathbb{R}t\,\left|(\uX)_\#(\muX\otimes\muX)-(\uY)_\#(\muY\otimes\muY)\right|(dt).
\end{equation}
	\end{corollary}
	\begin{proof}
		The claim follows directly from \Cref{prop:flb-slb-w-form} and \Cref{rmk:int-w-inf-real}.
	\end{proof}
	
	This corollary implies that $\uSLB{p}$ is more rigid than $\dSLB{p}$, since the second summand on the right hand side of \Cref{eq:slbu-slb} is sensitive to distance perturbations. This is also illustrated very well in the subsequent example.
	\begin{example}
	Recall notations from \Cref{ex:notation two point space}. For any $d,d'>0$, we let $X\coloneqq\Delta_2(d)$ and let $Y\coloneqq\Delta_2(d')$. Assume that $X$ and $Y$ have underlying sets $\{x_1,x_2\}$ and $\{y_1,y_2\}$, respectively. Define $\muX\in\mathcal{P}(X)$ and $\muY\in\mathcal{P}(Y)$ as follows. Let $\alpha_1,\alpha_2\geq 0$ be such that $\alpha_1+\alpha_2=1$. Let $\muX(x_1)=\muY(y_1)\coloneqq\alpha_1$ and let $\muX(x_2)=\muY(y_2)\coloneqq\alpha_2$. Then, it is easy to verify that
	\begin{enumerate}
	    \item $\ugw{1}(\X,\Y)=\uSLB{1}(\X,\Y)=2\alpha_1\alpha_2\Lambda_\infty(d,d').$
	    \item $\dgw{1}(\X,\Y)=\dSLB{1}(\X,\Y)=\alpha_1\alpha_2\Lambda_1(d,d')=\alpha_1\alpha_2|d-d'|.$
	    \item $\frac{1}{2}\int_\mathbb{R}t\,\left|(\uX)_\#(\muX\otimes\muX)-(\uY)_\#(\muY\otimes\muY)\right|(dt)=\alpha_1\alpha_2(d+d')\delta_{d\neq d'}$.
	\end{enumerate}
	From 1 and 2 we observe that both second lower bounds are tight.
	Moreover, since we obviously have that $(d+d')\delta_{d\neq d'}+ |d-d'|=2\Lambda_\infty(d,d')$, we have also verified \Cref{eq:slbu-slb} through this example.
	Unlike $\dSLB{1}(\X,\Y)$ being proportional to $|d-d'|$, as long as $d\neq d'$, even if $|d-d'|$ is small, $\Lambda_\infty(d,d')=\max(d,d')$ which results in a large value of $\uSLB{1}(\X,\Y)$ when $d$ and $d'$ are large numbers. This example illustrates that $\uSLB{1}$ (and hence $\ugw{1}$) is rigid with respect to distance perturbation. 
	\end{example}

\section{\texorpdfstring{$\ugw p$}{The ultrametric Gromov-Wasserstein distance} on ultra-dissimilarity spaces}\label{sec:ultra-dissimilarity spaces}   A natural generalization of ultrametric spaces is provided by \emph{ultra-dissimilarity spaces}. These spaces naturally occur when working with {symmetric  ultranetworks} (see \cite{smith2016hierarchical}) or phylogenetic tree data (see \cite{semple2003phylogenetics}). In this section, we will introduce these spaces and briefly illustrate to what extend the results for $\ugw{p}$ can be adapted for ultra-dissimilarity measure spaces. We start by formally introducing \emph{ultra-dissimilarity spaces}. 
\begin{definition}[Ultra-dissimilarity spaces] \label{def:ultra dissimilarity}
		An \emph{ultra-dissimilarity} space is a couple $(X,\uX)$ consisting of a set $X$ and a function $\uX:X\times X\rightarrow\Rp$ satisfying the following conditions for any $x,y,z\in X$:
		\begin{enumerate}
			\item  $\uX(x,y)=\uX(y,x)$;
			\item  $\uX(x,y)\leq\max(\uX(x,z),\uX(z,y)); $
			\item $\max(\uX(x,x),\uX(y,y))\leq \uX(x,y)$ and the equality holds if and only if $x=y$.
		\end{enumerate}
	\end{definition}

\begin{remark}
Note that when $(X,u_X)$ is an ultrametric space the third condition is trivially satisfied.
\end{remark} 

In the following, we restrict ourselves to \emph{finite} ultra-dissimilarity spaces to avoid technical issues in topology (see \cite{chowdhury2019metric,chowdhury2019gromov} for a more complete treatment of infinite spaces). One important aspect of ultra-dissimilarity spaces is the connection with the so-called \emph{treegrams} \citep{smith2016hierarchical,memoli2019gromov}, {which can be regarded as generalized} dendrograms. For a finite set $X$, let $\mathbf{SubPart}(X)$ denote the collection of all \emph{subpartitions} of $X$: Any partition $P'$ of a non-empty subset $X'\subseteq X$ is called a subpartition of $X$. Given two subpartitions $P_1,P_2$, we say $P_1$ is coarser than $P_2$ if each block in $P_2$ is contained in some block in $P_1$.

\begin{definition}[Treegrams]\label{def:treegram}
A \emph{treegram} $T_X:[0,\infty)\rightarrow \mathbf{SubPart}(X)$ is a map parametrizing a nested family of subpartitions over the same set $X$ and satisfying the following conditions:
\begin{enumerate}
	\item For any $0\leq s<t<\infty$, $T_X(t)$ is coarser than $T_X(s)$; 
	\item There exists $t_X>0$ such that for any $t\geq t_X$, $T_X(t)=\{X\}$;
	\item For each $t\geq 0$, there exists $\eps>0$ such that $T_X(t)=T_X(t')$ for all $t'\in[t,t+\eps]$;
	\item For each $x\in X$, there exists $t_x\geq 0$ such that $\{x\}$ is a block in $T_X(t_x)$.
\end{enumerate}
\end{definition}
Similar to \Cref{thm:compact ultra-dendro}, which correlates ultrametrics to dendrograms, there exists an equivalence relation between ultra-dissimilarity functions and treegrams on a finite set (see \Cref{fig:treegram} for an illustration).

\begin{proposition}[\citet{smith2016hierarchical}]\label{prop:ultradis-tree}
		Given a finite set $X$, denote by $\mathcal{U}_\mathrm{dis}(X)$ the collection of all ultrametric dissimilarity functions on $X$ and by $\mathcal{T}(X)$ the collection of all treegrams over $X$. Then, there exists a bijection $\Delta_X:\mathcal{T}(X)\rightarrow\mathcal{U}_\mathrm{dis}(X)$.
\end{proposition}
\begin{figure}
    \centering
    \includegraphics[width=0.8\textwidth]{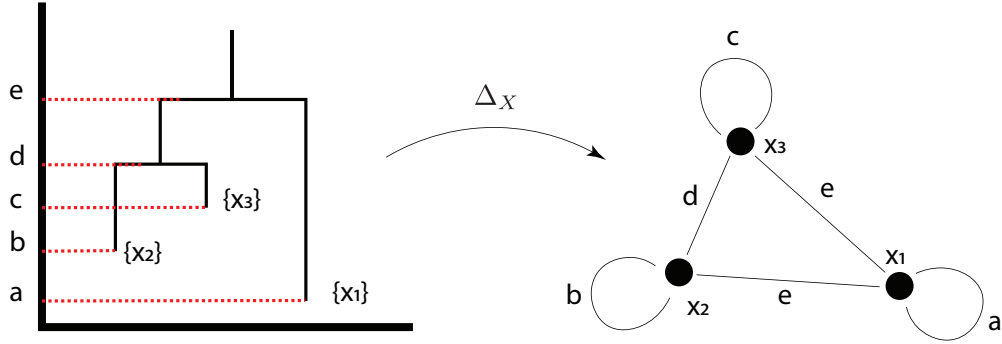}
    \caption{\textbf{Treegrams:} Relation between ultra-dissimilarity functions and treegrams}
    \label{fig:treegram}
\end{figure}

An \emph{ultra-dissimilarity measure space} is a triple $\mathcal{X}=(X,u_X,\mu_X)$ where $(X,u_X)$ is an ultra-dissimilarity space and $\mu_X$ is a probability measure fully supported on $X$.
Just as for metric spaces or metric measure spaces, it is important to have a notion of isomorphism between ultra-dissimilarity spaces.
\begin{definition}[Isomorphism]\label{def:isomorphism ultra-dissimilarity}
Given two ultra-dissimilarity measure spaces  $\X$ and $\Y$, we say they are \emph{isomorphic}, denoted $\X\cong_w\Y$, if there is a bijective function $f:X\rightarrow Y$ such that $f_\#\mu_X=\mu_Y$ and for any $x,x'\in X$ it holds $u_Y(f(x),f(x'))=u_X(x,x')$. The collection of all isomorphism classes of ultra-dissimilarity spaces is denoted by $\ultradiscol$.
\end{definition}
Given the previous results it is straightforward to show that $\ugw{p}$, $1\leq p\leq \infty$, is a metric on the isomorphism classes of $\ultradiscol$. For the complete proof of the subsequent statement, we refer to \Cref{sec: proof of thm ugw-p-metric-dis}.
\begin{theorem}\label{thm:ugw-p-metric-dis}
	The ultrametric Gromov-Wasserstein distance $\ugw{p}$ is a $p$-metric on $\ultradiscol$.
\end{theorem}
\begin{remark}
    Since $\ugw{p}$ translates to a metric on $\ultradiscol$, it is clear that it admits the lower bounds introduced in \Cref{sec:lower bounds}. 
\end{remark}

	%%%%%%%%%%%%%%%%%%%%%%%%
	\section{Computational aspects}\label{sec:computational aspects}
	In this section, we investigate algorithms for approximating/calculating $\ugw{p}$, $1\leq p\leq \infty$. Furthermore, we evaluate for $p<\infty$ the performance of the computationally efficient lower bound $\uSLB{}$ introduced in \Cref{sec:lower bounds} and compare our findings to the results of the classical Gromov-Wasserstein distance $d_{\mathrm{GW},p}$ (see \Cref{eq:Gromov Wasserstein}). Matlab implementations of the presented algorithms and comparisons are available at \url{https://github.com/ndag/uGW}.

	\subsection{Algorithms}
	Let $\X=\ummspaceX$ and $\Y=\ummspaceY$ be two finite ultrametric measure spaces with cardinalities $m$ and $n$, respectively. 
	
	\subsubsection{The case \texorpdfstring{$p<\infty$}{p finite}}\label{subsubsec:p smaller infty} We have already noted in \Cref{rem:computational complexity ugw} that calculating $\ugw{p}(\X,\Y)$ for $p<\infty$ yields a non-convex quadratic program (which is an NP-hard problem in general \citep{pardalos1991quadratic}). Solving this is not feasible in practice. However, in many practical applications it is sufficient to work with good approximations. Therefore, we propose to approximate $\ugw{p}(\X,\Y)$ for $p<\infty$ via conditional gradient descent. To this end, we note that the gradient $G$ that arises from \Cref{eq:distortion ult} can in the present setting be expressed with the following partial derivative with respect to $\mu\in\mathcal{C}(\muX,\muY)$
	\begin{equation}\label{eq:ugw gradient}
	G_{i,j}=2\sum_{k=1}^m\sum_{l=1}^{n}(\Lambda_\infty(\uX(x_i,x_k),\uY(y_j,y_l)))^p\mu_{kl}, \quad\forall 1\leq i\leq m,1\leq j\leq n.
	\end{equation}
	As we deal with a non-convex minimization problem, the performance of the gradient descent strongly depends on the starting coupling $\mu^{(0)}$. Therefore, we follow the suggestion of \citet{chowdhury2020generalize} and employ a Markov Chain Monte Carlo Hit-And-Run sampler to obtain multiple random start couplings. Running the gradient descent from each point in this ensemble greatly improves the approximation in many cases. For a precise description of the proposed procedure, we refer to \Cref{algo:gradient descent}.
	\begin{algorithm}[htb]
				\caption{$\ugw{p}(X,Y,p,N,L)$}\label{algo:gradient descent}
		\begin{algorithmic} 
		
		\STATE{\emph{//Create a list of random couplings}}
		\STATE{couplings =CreateRandomCouplings(N)}
		\STATE{stat\_points = cell(N)}
		\FOR{i=1:N}
			\STATE{$\mu^{(0)}=$couplings\{$i$\}}
			\FOR{j=1:L}
				\STATE{$G=$ Gradient from \Cref{eq:ugw gradient} w.r.t. $\mu^{(j-1)}$}
				\STATE{$\tilde{\mu}^{(j)}=$ Solve OT with ground loss $G$}
				\STATE{$\gamma^{(j)}=\frac{2}{j+2}$}
				\STATE{\emph{//Alt. find $\gamma\in[0,1]$ that minimizes $\mathrm{dis}_p^\mathrm{ult}\Big(\mu^{(j-1)}+\gamma\big( \tilde{\mu}^{(j)}-\mu^{(j-1)}\big) \Big)$} }
				\STATE{$\mu^{(j)}=(1-\gamma^{(j)})\mu^{(j-1)}+\gamma^{(j)}\tilde{\mu}^{(j)}$}	
			\ENDFOR
			\STATE{stat\_points\{$i$\}= $\mu^{(L)}$}
		\ENDFOR
		\STATE{Find $\mu^*$ in stat\_points that minimizes $\mathrm{dis}_p^\mathrm{ult}(\mu)$}
		\STATE{result =$\mathrm{dis}_p^\mathrm{ult}(\mu^*)$}

	\end{algorithmic}
	\end{algorithm}
		\subsubsection{The case \texorpdfstring{$p=\infty$}{p infinite}} \label{subsubsec:p equals infty} 
			For $p=\infty$, it follows by \Cref{thm:ugw-infty-eq} that
	\begin{equation}\ugw{\infty}(\X,\Y)=\inf\left\lbrace t\geq 0 \,|\,\X_t \cong_w \Y_t\right\rbrace.\label{eq:p infty algo idea}\end{equation}
	This identity allows us to construct a polynomial time algorithm for $\ugw{\infty}(\X,\Y)$ based on the ideas of \citet[Sec. 8.2.2]{memoli2019gromov}. More precisely, let  $\spec{X}\coloneqq\{\uX(x,x')|\,x,x'\in X\}$ denote the spectrum of $X$. Then, it is evident that in order to find the infimum in \Cref{eq:p infty algo idea}, we only have to check $\X_t \cong_w \Y_t$ for each $t\in\spec{X}\cup\spec{Y}$, starting from the largest to the smallest and $\ugw{\infty}$ is given as the smallest $t$ such that $\X_t\cong_w \Y_t$. This can be done in polynomial time by considering $\X_t$ and $\Y_t$ as labeled, weighted trees (e.g. by using a slight modification of the algorithm in Example 3.2 of \cite{aho1974design}). This gives rise to a simple algorithm (see \Cref{algo:ugw infty}) to calculate $\ugw{\infty}$.
		\begin{algorithm}[H]
		\caption{$\ugw{\infty}(\X,\Y)$}\label{algo:ugw infty}
		\begin{algorithmic}
		\STATE{spec = sort($\spec{X}\cup\spec{Y}$, 'descent')}
		\FOR{$i=1:\mathrm{length(spec)}$}
			\STATE{$t=\mathrm{spec}(i)$}
			\IF{$\X_t\ncong_w \Y_t$}
				\RETURN $\mathrm{spec}(i-1)$
			\ENDIF
		\ENDFOR
		\RETURN $0$
		\end{algorithmic}
	\end{algorithm}

	%%%%%%%%%%%%%%%%%%%%%%

	\subsection{The relation between \texorpdfstring{$\ugw{1}$, $\ugw{\infty}$}{the ultrametric Gromov-Wasserstein distance} and \texorpdfstring{$\uSLB{1}$}{its second lower bound}}\label{sec:ugw toy examples}
	In order to understand how $\ugw{p}$ (or at least its approximation), $\ugw{\infty}$ and $\uSLB{p}$ are influenced by small changes in the structure of the considered ultrametric measure spaces, we exemplarily consider the ultrametric measure spaces $ \X_i=(X_i,d_{X_i},\mu_{X_i})$, $1\leq i \leq 4$, displayed in \Cref{fig:umm spaces}. These ultrametric measure spaces differ only by one characteristic (e.g. one side length or the equipped measure). Exemplarily, we calculate $\ugw{1}(\X_i,\X_j)$ (approximated with \Cref{algo:gradient descent}, where $L=5000$ and $N=40$), $\uSLB{1}(\X_i,\X_j)$ and $\ugw{\infty}(\X_i,\X_j)$, $1\leq i,j\leq 4$. The results suggest that $\uSLB{1}$ and $\ugw{1}$ are influenced by the change in the diameter of the spaces the most (see \Cref{tab:exemplary comparison} and \Cref{tab:exemplary comparison uslb} in \Cref{app:ugw toy examples} for the complete results). Changes in the metric influence $\uSLB{1}$ in a similar fashion as $\ugw{1}$, while changes in the measure have less impact on $\uSLB{1}$. Further, we observe that $\ugw{\infty}$ attains for almost all comparisons the maximal possible value. Only the comparison of $\X_1$ with $\X_3$, where the only small scale structure of the space was changed, yields a value that is smaller than the maximum of the diameters of the considered spaces. 
	\begin{figure}
		\centering
		\includegraphics[width=0.8\textwidth]{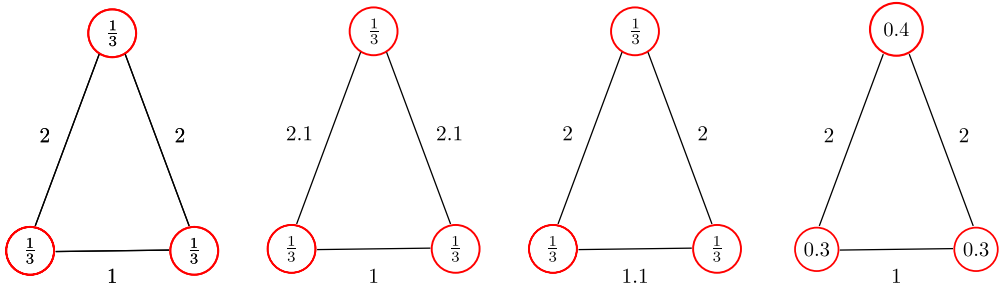}
		\caption{\textbf{Ultrametric measure spaces:} Four non-isomorphic ultrametric measure spaces denoted (from left to right) as  $\X_i=(X_i,d_{X_i},\mu_{X_i})$, $1\leq i \leq 4$.
		}\label{fig:umm spaces}
	\end{figure}
	
\subsection{Comparison of \texorpdfstring{$\ugw{1}$, $\uSLB{1}$, $\dgw{1}$ and $\dSLB{1}$}{the ultrametric and the usual Gromov-Wasserstein distance and Second Lower Bounds}}\label{subsec:relation to the GW-dist}
In the remainder of this section, we will demonstrate the differences between $\ugw{1}$, $\uSLB{1}$, $\dgw{1}$ and $\dSLB{1}$. To this end, we first compare the metric measure spaces in \Cref{fig:umm spaces} based on $\dgw{1}$ and $\dSLB{1}$. We observe that $\dgw{1}$ (approximated in the same manner as $\ugw{1}$) and $\dSLB{1}$ are hardly influenced by the differences between the ultrametric measure spaces $\X_i$, $1\leq i\leq 4$. In particular, it is remarkable that $\dgw{1}$ is affected the most by the changes made to the measure and not the metric structure (see  \Cref{tab:exemplary comparison GW} in \Cref{sec:details from subsec relation to the GW-dist} for the complete results).

Next, we consider the differences between the aforementioned quantities more generally. For this purpose, we generate 4 ultrametric spaces $Z_{k}$, $1\leq k\leq 4$, with totally different dendrogram structures, whose diameters are between 0.5 and 0.6 (for the precise construction of these spaces see \Cref{sec:details from subsec relation to the GW-dist}). For each $t=0,0.2,0.4,0.6$, we perturb each $Z_k$ independently to generate 15 ultrametric spaces $Z^{i}_{k,t}$, $1\leq i\leq 15$, such that $(Z^{i}_{k,t})_t\equiv (Z_{k})_t$ for all $i$. The spaces $Z^{i}_{k,t}$ are called \emph{pertubations of $Z_{k}$ at level $t$} (see \Cref{fig:randomly sampled ultrametric measure spaces} for an illustration and see \Cref{sec:details from subsec relation to the GW-dist} for more details). The spaces $Z^{i}_{k,t}$ are endowed with the uniform probability measure and we obtain a collection of ultrametric measure spaces $\Z^{i}_{k,t}$. Naturally, we refer to $k$ as the class of the ultrametric measure space $\Z^{i}_{k,t}$. We compute for each $t$ the quantities $\ugw{1}$, $\uSLB{1}$, $\dgw{1}$ and $\dSLB{1}$ among the resulting $60$ ultrametric measure spaces. The results, where the spaces have been ordered lexicographically by $(k,i)$, are visualized in \Cref{fig:noise ums}. As previously, we observe that $\ugw{1}$ and $\uSLB{1}$ as well as $\dgw{1}$ and  $\dSLB{1}$ behave in a similar manner. More precisely, we see that both $\dgw{1}$ and $\dSLB{1}$ discriminate well between the different classes and that their behavior does not change too much for an increasing level of perturbation. On the other hand, $\ugw{1}$ and $\uSLB{1}$ are very sensitive to the level of perturbation. For small $t$ they discriminate better than $\dgw{1}$ and $\dSLB{1}$ between the different classes and pick up clearly that the perturbed spaces differ. However, if the level of perturbation becomes too large both quantities start to discriminate between spaces from the same class (see \Cref{fig:noise ums}). 
\begin{figure}[htb]
		\centering

		\includegraphics[width=\textwidth]{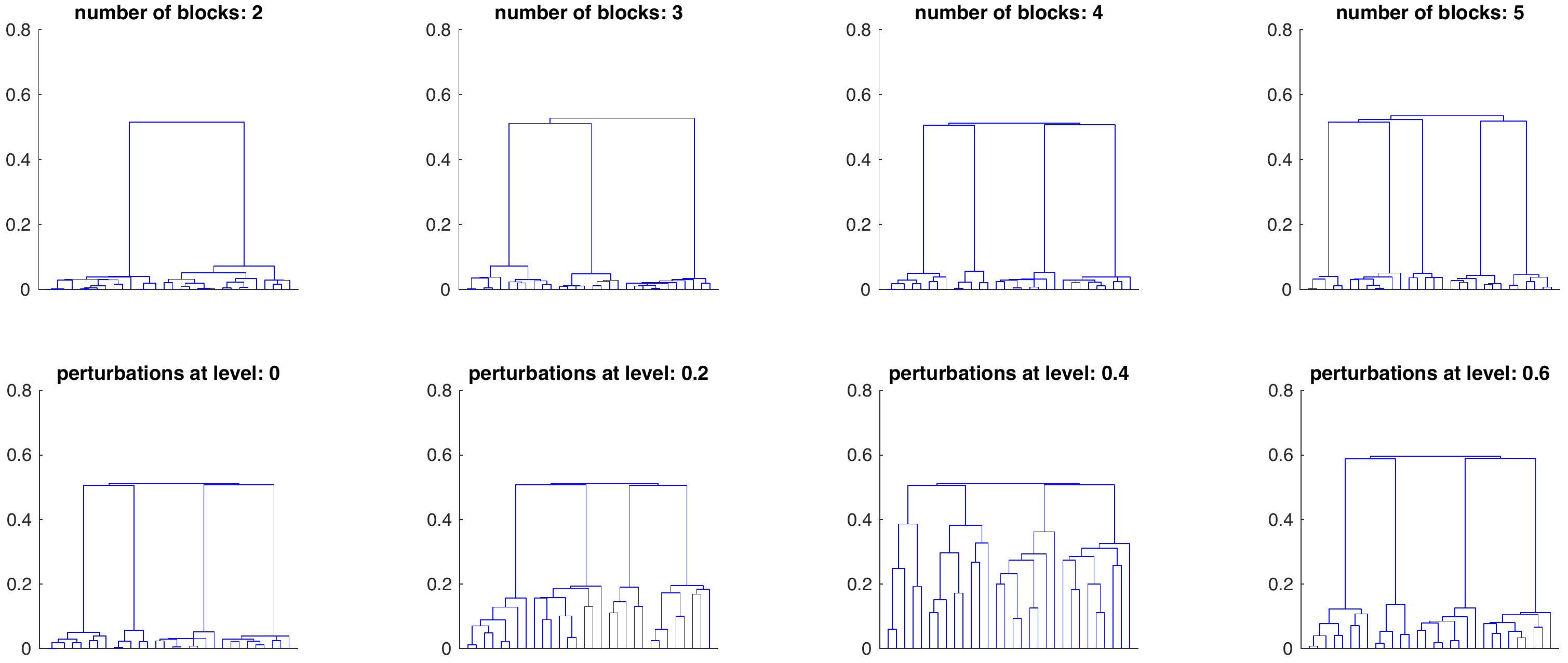}
		
		\caption{\textbf{Randomly sampled ultrametric measure spaces:} Illustration of $Z_k$ for $k=2,3,4,5$ (top row) and instances for perturbations of $Z_4$ with respect to perturbation level $t\in\{0,0.2,0.4,0.6\}$ (bottom row). } \label{fig:randomly sampled ultrametric measure spaces}
\end{figure}

	\begin{figure}[htb]
		\centering
		\includegraphics[width=\textwidth]{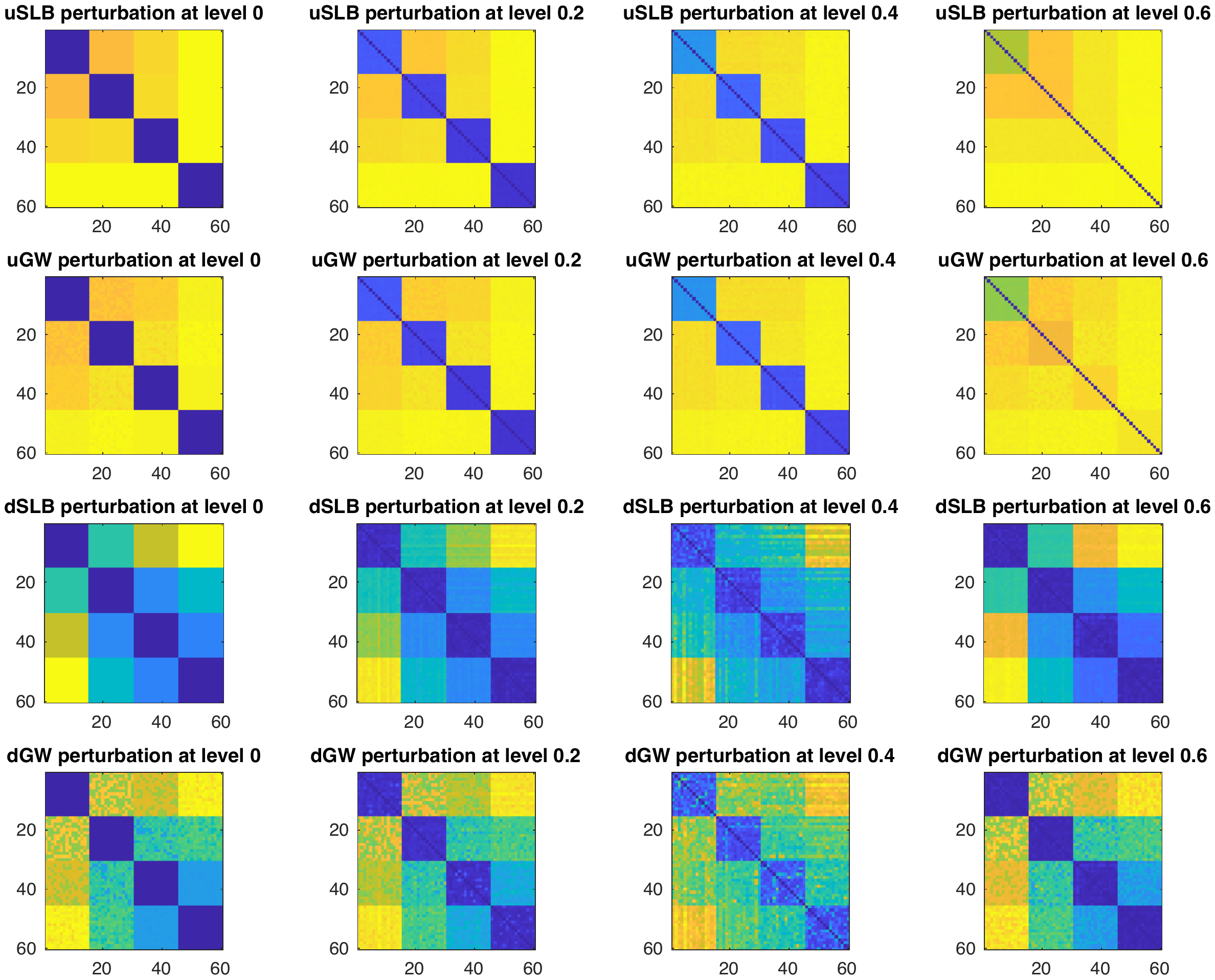}
		\caption{\textbf{$\ugw{1}/\uSLB{1}$ and $\dgw{1}/\dSLB{1}$ among randomly generated ultrametric measure spaces:} Heatmap representations of $\uSLB{1}(\Z^i_{n,t},\Z^{i'}_{n',t})$ (top row), $\ugw{1}(\Z^i_{n,t},\Z^{i'}_{n',t})$ (second row), $\dSLB{1}(\Z^i_{n,t},\Z^{i'}_{n',t})$ (third row) and $\dgw{1}(\Z^i_{n,t},\Z^{i'}_{n',t})$ (bottom row), $k,k'\in\{2,\dots,5\}$ and $i,i'\in\{1,\ldots,15\}$.
		}\label{fig:noise ums}
	\end{figure}
 
In conclusion, $\ugw{1}$ and $\uSLB{1}$ are sensitive to differences in the large scales of the considered ultrametric measure spaces. While this leads (from small $t$) to good discrimination in the above example, it also highlights that they are (different from $\dgw{1}$ and $\dSLB{1}$) susceptible to large scale noise.

\section{Phylogenetic tree shapes}\label{sec:phylogenetic tree shapes}
Rooted phylogenetic trees (for a formal definition see e.g.,  \cite{semple2003phylogenetics}) are a common tool to visualize and analyze the evolutionary relationship between different organisms. In combination with DNA sequencing, they are an important tool to study the rapid evolution of different pathogens. It is well known that the (unweighted) shape of a phylogenetic tree, i.e., the tree's connectivity structure without referring to its labels or the length of its branches, carries important information about macroevolutionary processes (see e.g.,  \cite{mooers1997inferring, blum2006random,dayarian2014infer, wu2016joint}). In order to study the evolution of  and the relation between different pathogens, it is of great interest to compare the shapes of phylogenetic trees created on the basis of different data sets. Currently, the number of tools for performing phylogenetic tree shape comparison is quite limited and the development of new methods for this is an active field of research \citep{colijn2018metric,morozov2018extension,kim2019metric,liu2020polynomial}. It is well known that certain classes of phylogenetic trees (as well as their respective tree shapes) can be identified as ultrametric spaces \citep[Sec. 7]{semple2003phylogenetics}. On the other hand, general phylogenetic trees are closely related to treegrams (see \Cref{def:treegram}). In the following, we will use this connection and demonstrate exemplarily that the computationally efficient lower bound $\uSLB{1}$ has some potential for comparing phylogenetic tree shapes. In particular, we contrast it to the metric defined for this application in Equation (4) of \citet{colijn2018metric}, in the following denoted as $\colijn$, and study the behavior of $\dSLB{1}$ in this framework.

	 		\begin{figure}[htb]
        \centering
		\includegraphics[width=0.7\textwidth]{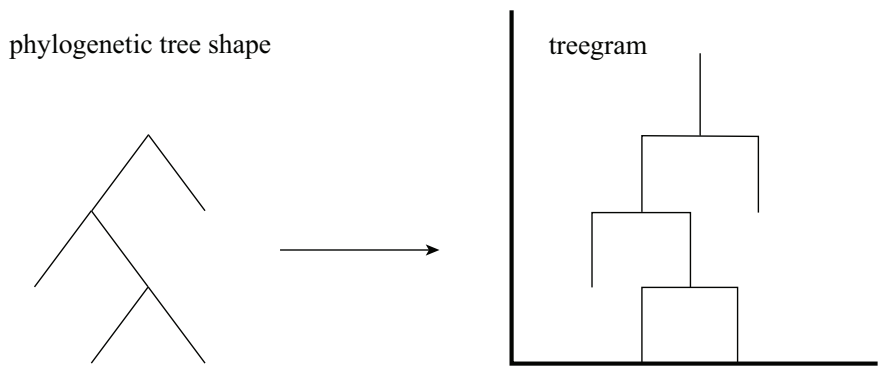}
		
		\caption{\textbf{Transforming a phylogenetic tree shape into an ultra-dissimilarity space:} In this figure, we illustrate the treegram corresponding to the ultra-dissimilarity space generated by \Cref{eq:tree shape distance} with respect to the phylogenetic tree shape on the left. Note that the treegram preserves the tree structure and the smallest birth time of points is exactly 0.}
		\label{fig:phlogenetric tree to treegram}
	\end{figure}

	In this section, we reconsider phylogenetic tree shape comparisons from \citet{colijn2018metric} and thereby study HA protein sequences from human influenza A (H3N2) (data downloaded from NCBI on 22 January 2016). More precisely, we investigate the relation between two samples of size 200 of phylogenetic tree shapes with 500 tips. Phylogenetic trees from the first sample are based on a random subsample of size 500 of 2168 HA-sequences that were collected in the USA between March 2010 and September 2015, while trees from the second sample are based on a random subsample of size 500 of 1388 HA-sequences gathered in the tropics between January 2000 and October 2015 (for the exact construction of the trees see \cite{colijn2018metric}). Although both samples of phylogenetic trees are based on HA protein sequences from human influenza A, we expect them to be quite different. On the one hand, influenza A is highly seasonal outside the tropics (where this seasonal variation is absent) with the majority of cases occurring in the winter \citep{russell2008global}. On the other hand, it is well known that the undergoing evolution of the HA protein causes a `ladder-like' shape of long-term influenza phylogenetic trees \citep{koelle2010two,volz2013viral,westgeest2012genetic,luksza2014predictive} that is typically less developed in short term data sets. Thus, also the different collection period of the two data sets will most likely influence the respective phylogenetic tree shapes.
	
	In order to compare the phylogenetic tree shapes of the resulting 400 trees, we have to transform the phylogenetic tree shapes into ultra-dissimilarity measure spaces $\mathcal{X}_i=(X_i,u_{X_i},\mu_{X_i})$, $1\leq i\leq 400$. To this end, we discard all the lables, denote by $X_i$ the tips of the $i$'th phylogenetic tree and refer to the corresponding tree shape as $\mathcal{T}_i$. Next, we define the ultra-dissimilarities $u_{X_i}$ on $X_i$, $1\leq i \leq 400$. For this purpose, we set all edge length in the considered phylogenetic trees to one and construct $u_{X_i}$ as follows: let $x^i_1,x^i_2\in X_i$  and let $a^i_{1,2}$ be the most recent common ancestor of $x^i_1$ and $x^i_2$. Let $d^i_{a_{1,2}}$ be the length of the shortest path from $a^i_{1,2}$ to the root, let $d^i_1$ be the length of the shortest path from $x^i_1$ to the root and let $d^i$ be the length of the longest shortest path from any tip to the root. Then, we define for any $x^i_1,x^i_2\in X_i$
	 \begin{equation}\label{eq:tree shape distance}
	     u_{X_i}(x^i_1,x^i_2)=\begin{cases}d^i- d^i_{a_{1,2}}& \text{if }x^i_1\neq x^i_2\\
	 d^i-d^i_1&\text{if }x^i_1= x^i_2, \end{cases}
	 \end{equation}
	 and weight all tips in $X_i$ equally (i.e. $\mu_{X_i}$ is the uniform measure on $X_i$). This naturally transforms the collection of phylogenetic tree shapes $\mathcal{T}_i$, $1\leq i\leq 400$, into a collection of ultra-dissimilarity spaces (see \Cref{fig:phlogenetric tree to treegram} for an illustration), which allows us to directly apply $\uSLB{1}$ to compare them (once again we exemplarily choose $p=1$).

	In \Cref{fig:uSLB and colijn phylo} we contrast our findings for the comparisons of the shapes $\mathcal{T}_i$, $1\leq i \leq 400$, to those obtained by computing the metric $\colijn$  described in \cite{colijn2018metric}. The top row of \Cref{fig:uSLB and colijn phylo} visualizes the dissimilarity matrix for the comparisons of all 400 phylogenetic tree shapes (the first 200 entries correspond to the tree shapes from the US-influenza and the second 200 correspond to the ones from the tropic influenza) obtained by applying $\uSLB{1}$ as heat map (left) and as multidimensional scaling plot (right). The heat map shows that the collection of US trees is divided into a large group $\groupUSone\coloneqq(\mathcal{T}_i)_{1\leq i\leq 161}$, that is well separated from the phylogenetic tree shapes based on tropical data $\grouptrop\coloneqq(\mathcal{T}_i)_{201\leq i\leq 400}$, and a smaller subgroup $\groupUStwo\coloneqq(\mathcal{T}_i)_{162\leq i\leq 200}$, that seems to be more similar (in the sense of $\uSLB{1}$) to the tropical phylogenetic tree shapes. In the following $\groupUSone$ and $\groupUStwo$ are referred to as \emph{US main} and \emph{US secondary group}, respectively. This division is even more evident in the MDS-plot on the right (black points represent trees shapes from the US main group, blue points trees shapes from the US secondary group and red points trees shapes based on the tropical data). 
	 \begin{framed}
We remark that in order to highlight the subgroups the US tree shapes have been reordered according to the output permutation of a single linkage dendrogram (w.r.t. $\uSLB{1}$) based on the US tree submatrix created by \citet{Matlabbasicversion} and that the tropical tree shapes have been reordered analogously.
\end{framed}
	 
	 The second row of \Cref{fig:uSLB and colijn phylo} displays the analogous plots for $\colijn$. It is noteworthy, that the coloring in the MDS-plot of the left is the same, i.e., $T_1\in\groupUSone$ is represented by a black point, $T_2\in\groupUStwo$ by a blue one and $T_3\in\grouptrop$ by a red one. Interestingly, the analysis based on these plots differs from the previous one. Using $\colijn$ to compare the phylogenetic tree shapes at hand, we can split the data into two clusters, where one corresponds to the US data and the other one to the tropical data, with only a small overlap (see the MDS-plot in the second row of \Cref{fig:uSLB and colijn phylo} on the right). In particular, we notice that $\colijn$ does not clearly distinguish between the US groups $\groupUSone$ and $\groupUStwo$.

	 \begin{figure}[htb]
		\centering
		\begin{subfigure}[c]{0.4\textwidth}
			\centering
			\includegraphics[width=0.8\textwidth]{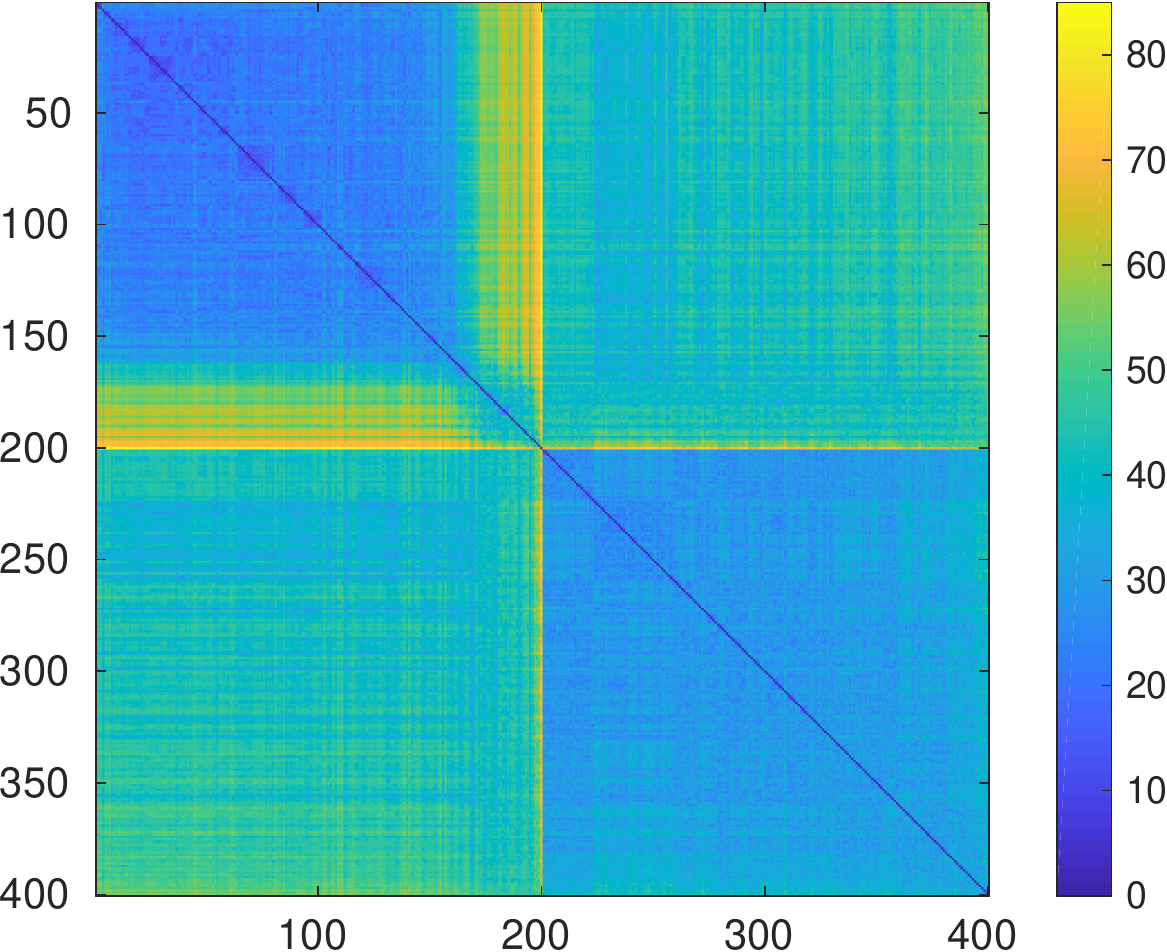}
		\end{subfigure}
			\begin{subfigure}[c]{0.4\textwidth}
			\centering
			\includegraphics[width=0.8\textwidth]{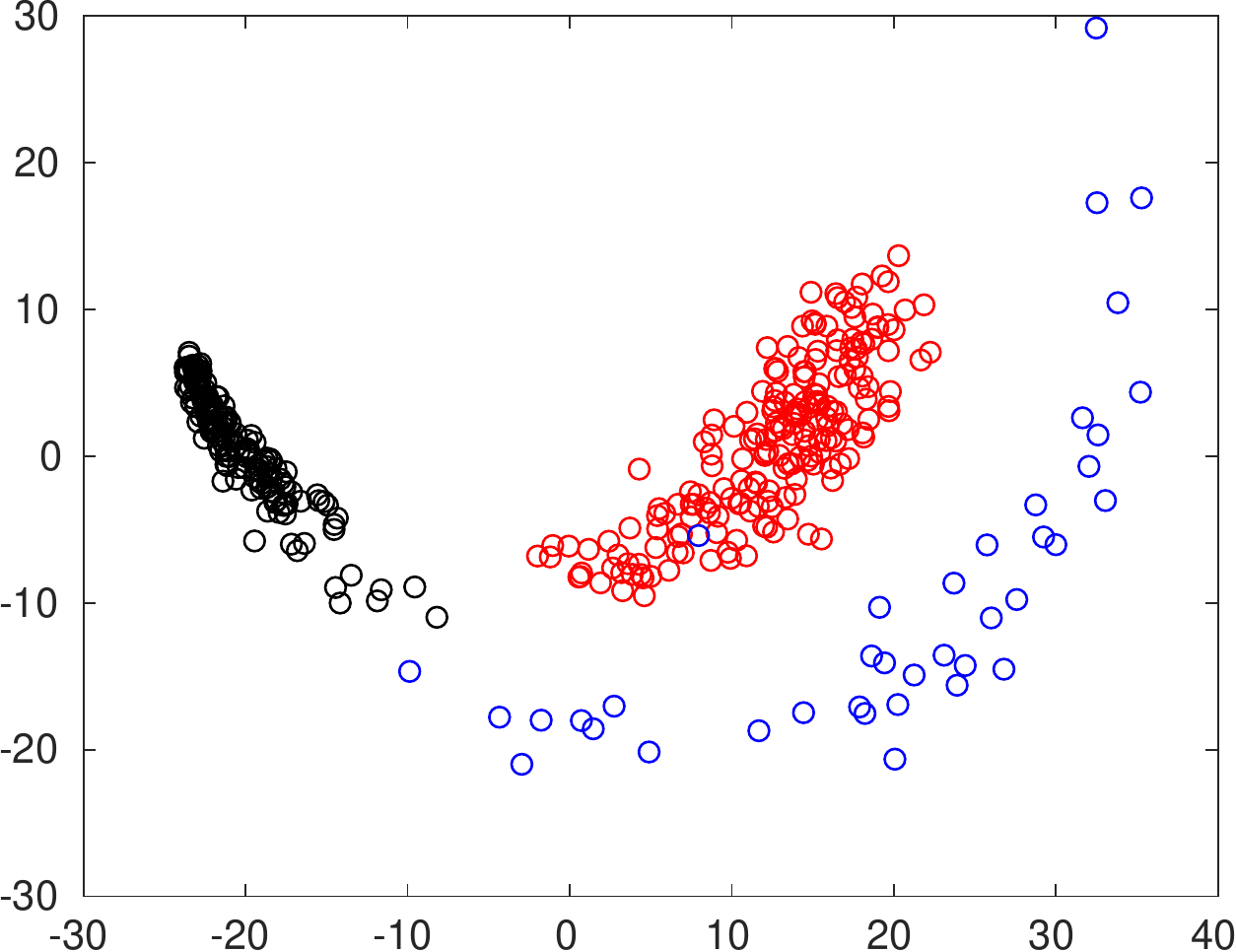}
		\end{subfigure}
		\vspace*{3mm}\\
	\begin{subfigure}[c]{0.4\textwidth}
			\centering
			\includegraphics[width=0.8\textwidth]{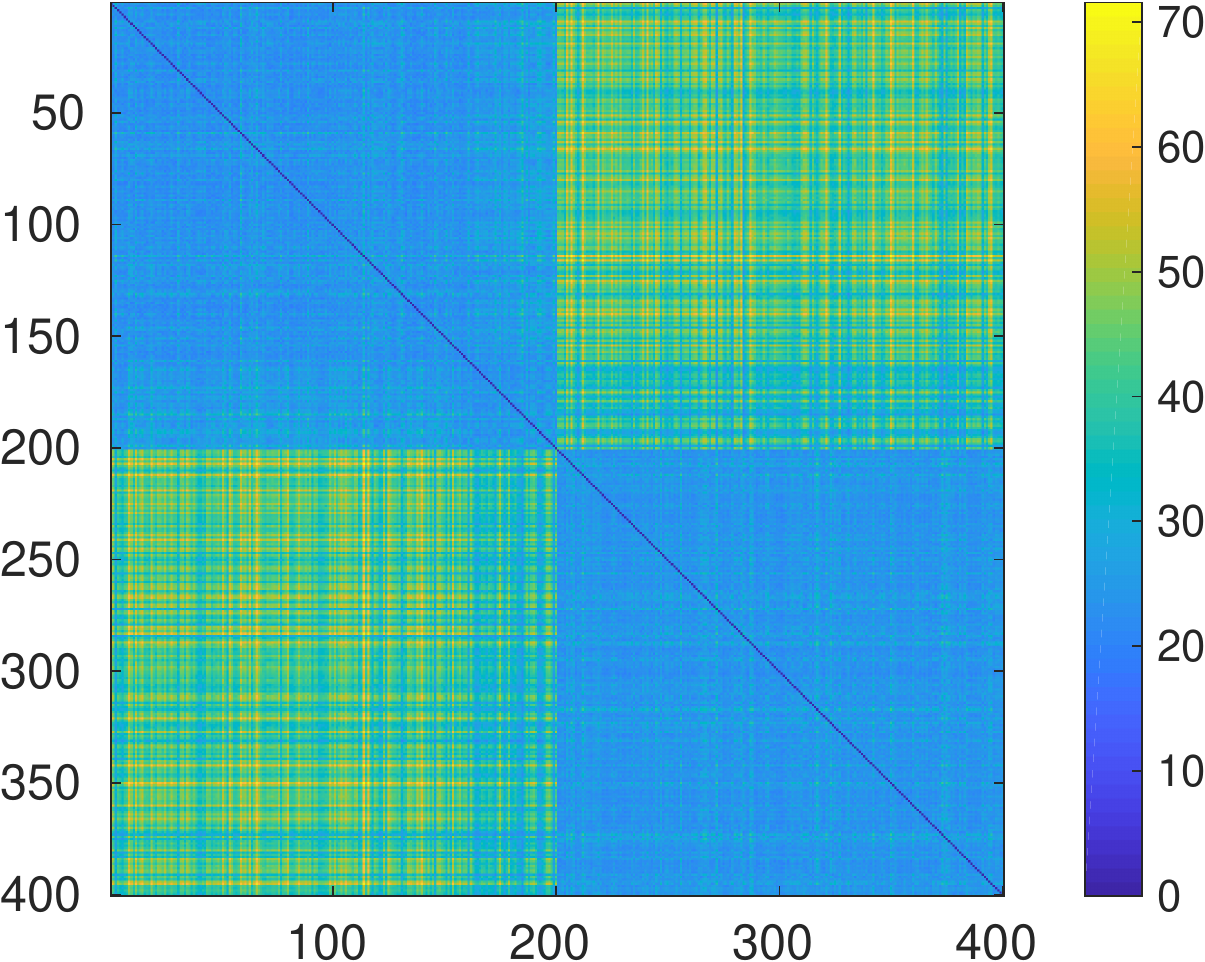}
		\end{subfigure}
		\begin{subfigure}[c]{0.4\textwidth}
			\centering
			\includegraphics[width=0.8\textwidth]{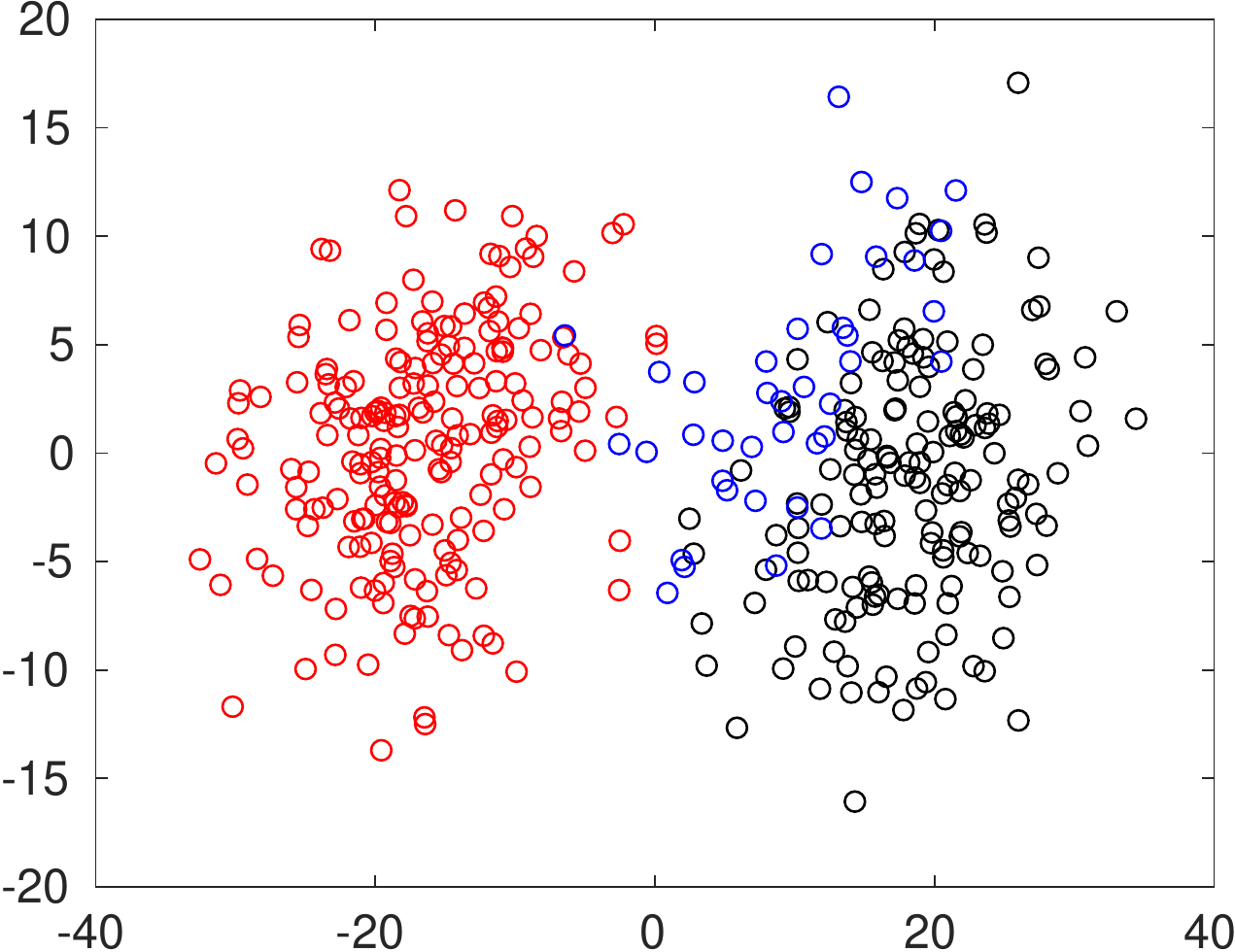}
		\end{subfigure}
 		\caption{\textbf{Phylogenetic tree shape comparison:} Visualization of the dissimilarity matrices for the comparison of the phylogenetic tree shapes $\mathcal{T}_i$, $1\leq i \leq 400$, based on $\uSLB{1}$ (top row) and $\colijn$ (bottom row) as heat maps (left) and MDS-plots (right). 
}
		\label{fig:uSLB and colijn phylo}
	\end{figure}
	
In order to analyze the different findings of $\uSLB{1}$ and $\colijn$, we collect and compare different characteristics of the tree shapes in the groups $\mathcal{G}_i$, $1\leq i\leq 3$. More precisely, we concentrate on various ``metric" properties of the considered ultra-dissimilarity spaces like $\frac{1}{500^2|\mathcal{G}_i|}\sum_{\mathcal{T}_i\in \mathcal{G}_i}\sum_{x,x'\in X_i}u_{X_i}(x,x')$ (``mean average distance") or $\frac{1}{|\mathcal{G}_i|}\sum_{\mathcal{T}_i\in \mathcal{G}_i}\max\{u_{X_i}(x,x')|x,x'\in X_i\}$  (``mean maximal distance"), $1\leq i\leq 3$,  (these influence $\uSLB{1}$ strongly) as well as the mean numbers of certain connectivity structures, like the 4- and 5-structures (these influence $\colijn$, for a formal definition see \cite{colijn2018metric}). Theses values (see \Cref{fig:phylo property illustration}) show that the mean average distance and the mean maximal distance differ drastically between the two groups of the US tree shapes. The tree shapes in these two groups are completely different from a metric perspective and the values for the secondary US group strongly resemble those of the tropic tree shapes. On the other hand, the connectivity characteristics do not change too much between the US main and secondary group. Hence, the metric $\colijn$ does not clearly divide the US trees into two groups, although the differences are certainly present. When carefully checking the phylogenetic trees, the reasons for the differences between trees in the US main group and US secondary group are not immediately apparent. Nevertheless, it is remarkable that trees from the secondary US cluster generally contain more samples from California and Florida (on average 1.92 and 0.88 more) and less from Maryland, Kentucky and Washington (on average 0.73, 0.83 and 0.72 less). 

\begin{table}[htb]
\centering
	\begin{tabular}{|c|c|c|c|}
		\hline 	\rule{0pt}{10pt} & USA (main group) &USA (secondary group)& Tropics   \\\hline
	Mean Avg. Dist. &36.16  &61.88    &  53.45	 \\
	Mean Max. Dist.  & 56.12 &    86.13   & 
   94.26	  \\
	Mean Num. of 4-Struc. & 15.61&     14.08  &7.81	\\
	Mean Num. of 5-Struc. & 28.04&     27.97  &35.82	\\\hline
	\end{tabular}
	\medskip
	\caption{\textbf{Tree shape characteristics:} The means of several metric and connectivity characteristics of the ultra-dissimilarity spaces $\X_i$ and the corresponding phylogenetic tree shapes $\mathcal{T}_i$, $1\leq i \leq 400$, for the three groups $\mathcal{G}_i$, $1\leq i\leq 3$.}\label{fig:phylo property illustration}
\end{table}

To conclude this section, we remark that using $\dSLB{1}$ instead of $\uSLB{1}$ for comparing the ultra-dissimilarity spaces $\X_i$, $1\leq i\leq 400$, gives comparable results (cf. \Cref{fig:SLB phylo}, coloring and ordering as previously). Nevertheless, we observe (as we already have in \Cref{sec:computational aspects}) that $\uSLB{1}$ is more discriminating than $\dSLB{1}$. 
Furthermore, we mention that so far we have only considered unweighted phylogenetic tree shapes. However, the branch lengths of the considered phylogenetic trees are relevant in many examples, because they can for instance reflect the (inferred) genetic distance between evolutionary events \citep{colijn2018metric}. While the branch lengths cannot easily be included in the metric $\colijn$,
the modeling of phylogenetic tree shapes as ultra-dissimilarity spaces is extremely flexible. It is straightforward to include branch lengths into the comparisons or to put emphasis on specific features (via weights on the corresponding tips). However, this is beyond the scope of this illustrative data analysis.

\begin{figure}[htb]
    \centering
	\begin{subfigure}[c]{0.4\textwidth}
	\centering
	\includegraphics[width=0.8\textwidth]{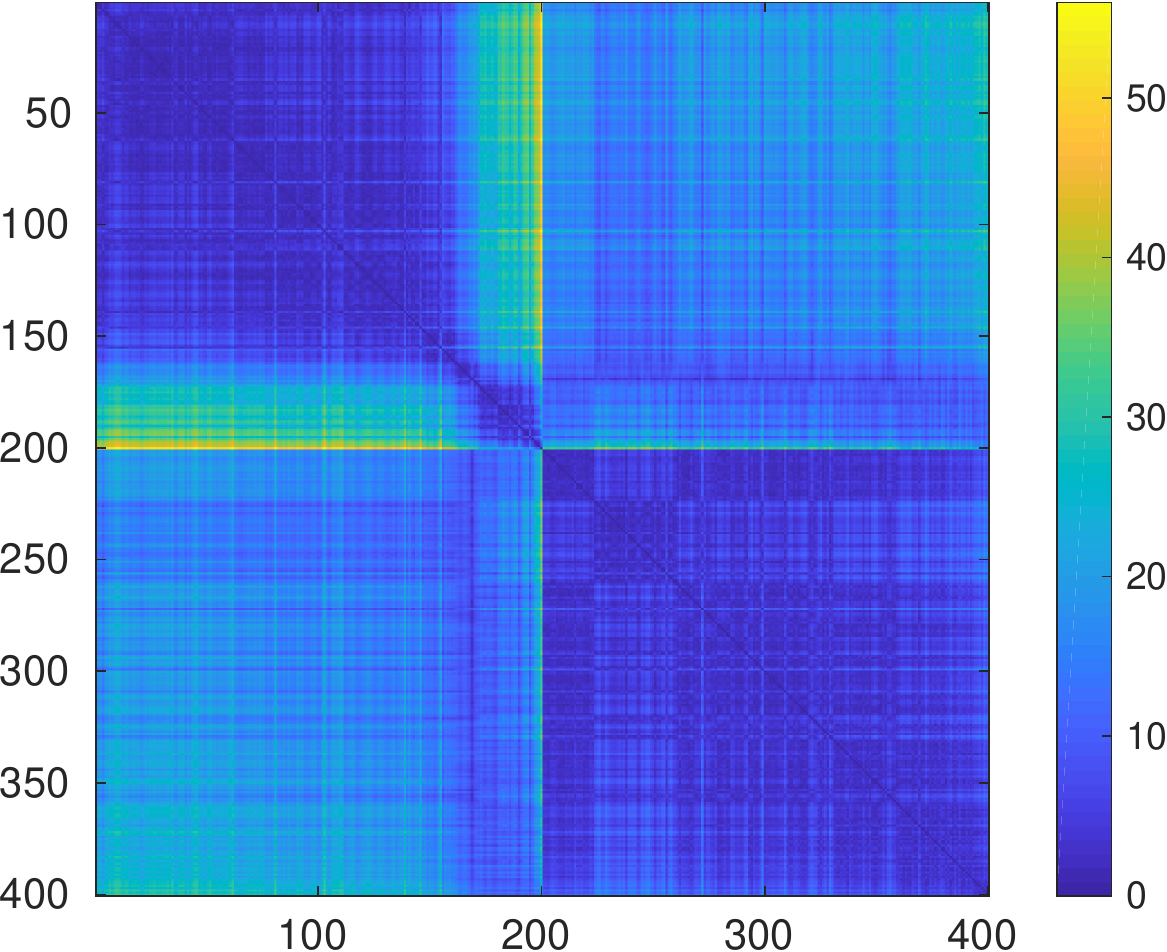}
	\end{subfigure}
	\begin{subfigure}[c]{0.4\textwidth}
	\centering
	\includegraphics[width=0.8\textwidth]{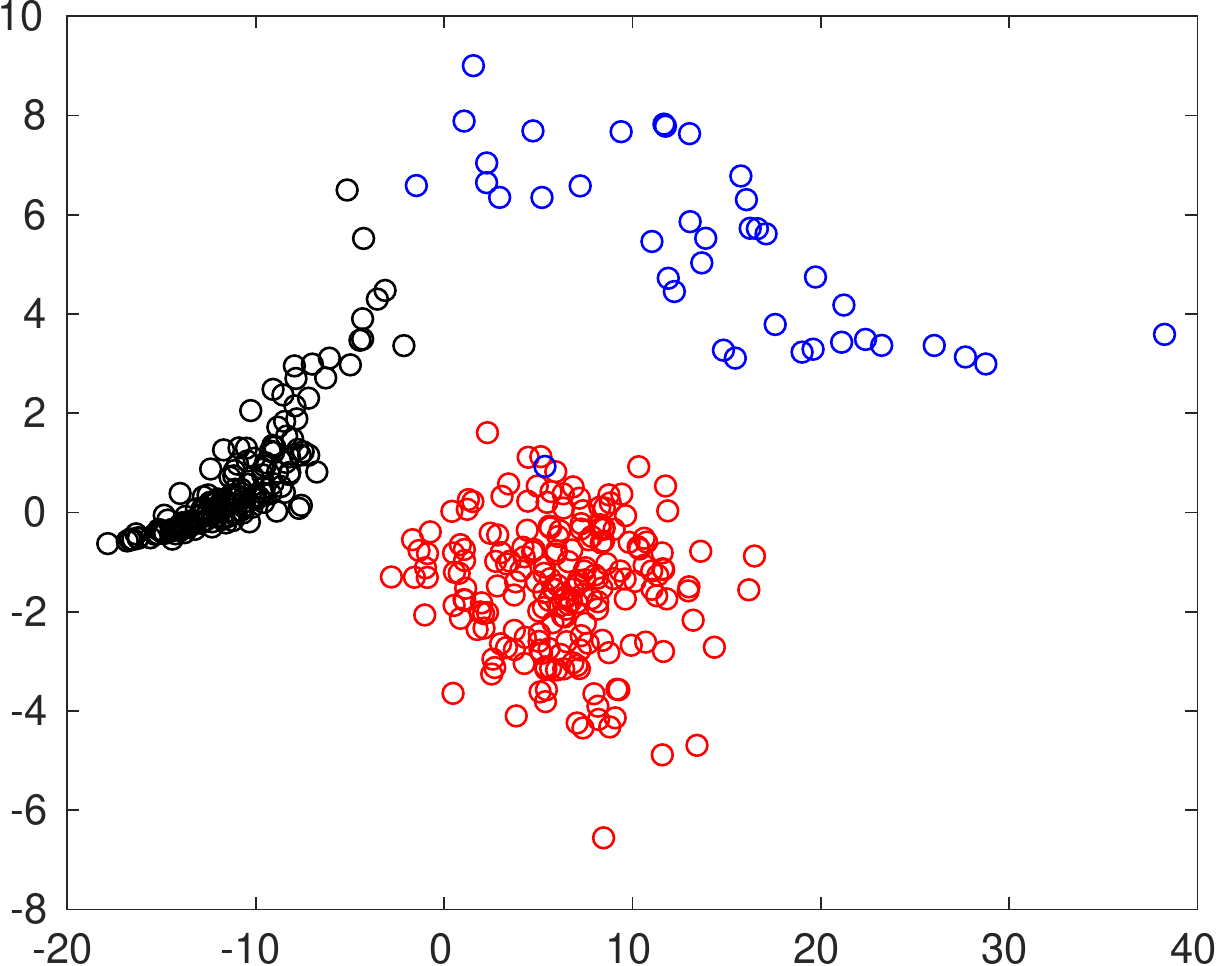}
 	\end{subfigure}
	\caption{\textbf{Phylogenetic tree shape comparison based on $\dSLB{1}$:} Representation of the dissimilarity matrices for the comparisons of the ultra-dissimilarity spaces $\X_i$, $1\leq i \leq 400$, based on $\dSLB{1}$ as heat maps (left) and MDS-plots (right).} \label{fig:SLB phylo}
\end{figure}

\section{Concluding remarks}
Since we suspect that computing  $\ugw{p}$ and $\usturm p$ for finite $p$ leads to NP-hard problems, it seems interesting to identify suitable collections of ultrametric measure spaces where these distances can be computed in polynomial time as done for the Gromov-Hausdorff distance in \cite{memoli2019gromov}.

\subsection*{Acknowledgements}
We are grateful to Prof. Colijn for sharing the data from \cite{colijn2018metric} with us. F.M. and Z.W. acknowledge funding from the NSF  under grants NSF CCF 1740761, NSF DMS 1723003, and NSF RI 1901360. A.M. and C.W. gratefully acknowledge support by the DFG Research Training Group 2088 and Cluster of Excellence MBExC 2067. F.M. and A.M. thank the Mathematisches Forschungsinstitut Oberwolfach. Conversations which eventually led to this project were initiated during the 2019 workshop ``Statistical and Computational Aspects of Learning with Complex Structure".

	\bibliographystyle{plainnat.bst}
	
	\bibliography{ugwbib}

\appendix
\addtocontents{toc}{\protect\setcounter{tocdepth}{0}}
\section{Missing details from Section \ref{sec:preliminaries}}\label{sec:details preliminaries}

\subsection{Proofs from Section \ref{sec:preliminaries}}\label{sec:proof of thm in preliminaries}
In this section we give the proofs of  various results form \Cref{sec:preliminaries}.

\subsubsection{Proof of \Cref{thm:compact ultra-dendro}}\label{proof:thm:compact ultra-dendro}
Recall that for a given $\theta\in\mathcal{D}(X)$, we define $u_\theta:X\times X\rightarrow\Rp$ as follows
	\[u_\theta(x,x')\coloneqq\inf\{t\geq 0|\,x \text{ and }x'\text{ belong to the same block of }\theta(t)\}. \]
	It is straightforward to verify that $u_\theta$ is an ultrametric. For any Cauchy sequence $\{x_n\}_{n\in\mathbb N}$ in $(X,u_\theta)$, let $D_i\coloneqq\sup_{m,n\geq i}u_\theta(x_m,x_n)$ for each $i\in\mathbb N$. Then, each $D_i<\infty$ and $\lim_{i\rightarrow\infty}D_i=0$. By definition of $u_\theta$, we have that for each $i\in\mathbb N $ the set $\{x_n\}_{n=i}^\infty$ is contained in the block $[x_i]_{D_i}\in\theta(D_i)$. Let $X_i\coloneqq [x_i]_{D_i}$ for each $i\in\mathbb N $. Then, obviously we have that $X_j\subseteq X_i$ for any $1\leq i<j$. By condition (7) in \Cref{def:proper dendrogram}, we have that $\bigcap_{i\in\mathbb N}X_i\neq\emptyset$. Choose $x_*\in \bigcap_{i\in\mathbb N}X_i$, then it is easy to verify that $x_*=\lim_{n\rightarrow\infty}x_n$ and thus $(X,u_\theta)$ is a complete space. To prove that $(X,u_\theta)$ is a compact space, we need to verify that for each $t>0$, $X_t$ is a finite space (cf. \Cref{lm:quotient-finite}). Since $\theta(t)$ is finite by condition (6) in \Cref{def:proper dendrogram}, we have that $X_t=\{[x]_t|\,x\in X\}=\theta(t)$ is finite and thus $X$ is compact. Therefore, we have proved that $u_\theta\in\mathcal{U}(X)$. Based on this, the map $\Delta_X:\mathcal{D}(X)\to\mathcal{U}(X)$ by $\theta\mapsto u_\theta$ is well-defined. 
	
	Now given $u\in\mathcal{U}(X)$, we define a map $\theta_u:[0,\infty)\rightarrow\mathbf{Part}(X)$ as follows:
	for each $t\geq 0$, consider the equivalence relation $\sim_t$ with respect to $u$, i.e., $x\sim_t x'$ if and only if $u(x,x')\leq t$. This is actually the same equivalence relation defined in \Cref{sec:ultrametric Gromov-Hausdorff} for introducing quotient ultrametric spaces. We then let $\theta_u(t)$ to be the partition induced by $\sim_t$, i.e., $\theta_u(t)=X_t.$ It is not hard to show that $\theta_u$ satisfies conditions (1)--(5) in \Cref{def:proper dendrogram}. Since $X$ is compact, then $\theta_u(t)=X_t$ is finite for each $t>0$ and thus $\theta_u$ satisfies condition (6) in \Cref{def:proper dendrogram}. Now, let $\{t_n\}_{n\in\mathbb N}$ be a decreasing sequence such that $\lim_{n\rightarrow\infty}t_n=0$ and let $X_n\in \theta_X(t_n)$ such that for any $1\leq n<m$, $X_m\subseteq X_n$. Since each $X_n=[x_n]_{t_n}$ for some $x_n\in X$, $X_n$ is a compact subset of $X$. Since $X$ is also complete, we have that $\bigcap_{n\in\mathbb N}X_n\neq\emptyset$. Therefore, $\theta_u$ satisfies condition (7) in \Cref{def:proper dendrogram} and thus $\theta_u\in\mathcal{D}(X)$. Then, we define the map $\Upsilon_X:\mathcal{U}(X)\to \mathcal{D}(X)$ by $u\mapsto \theta_u$. 
	
	It is easy to check that $\Upsilon_X$ is the inverse of $\Delta_X$ and thus we have established that  $\Delta_X:\mathcal{D}(X)\rightarrow\mathcal{U}(X)$ is bijective.
	%%%%%%%%%%%%%%
	\subsubsection{Proof of Lemma \ref{lm:winfty-finite}}\label{sec:proof of lm winfty-finite}
	First of all, we show that the right hand side of \Cref{eq:ultra Wasserstein infinity} is well defined. More precisely, we employ \Cref{lm:quotient-finite} to prove that the supremum \[\sup_{B\in V(X)\backslash\{X\}\,\text{ and }\,\alpha(B)\neq\beta(B)}\diam{B^*}\] is attained. For arbitrary $B_0\in V(X)\backslash\{X\}$ such that $\alpha(B_0)\neq\beta(B_0)$, we have that $\diam{B^*_0}>0$. By \Cref{lm:quotient-finite} the spaces $X_t$ are finite for $t>0$. Since $V(X)=\{[x]_t|\,x\in X,t> 0\}=\bigcup_{t>0}X_t$, there are only finitely many $B\in V(X)\backslash\{X\}$ such that $\diam B\geq \diam{B^*_0}$ and thus $\diam{B^*}\geq \diam{B^*_0}$. This implies that the supremum is attained and thus 
	\begin{equation}\label{eq:w infinity max}
	    \sup_{B\in V(X)\backslash\{X\}\,\text{ and }\,\alpha(B)\neq\beta(B)}\diam{B^*}=\max_{B\in V(X)\backslash\{X\}\,\text{ and }\,\alpha(B)\neq\beta(B)}\diam{B^*}.
	\end{equation}
	Let $B_1$ denote the maximizer in \Cref{eq:w infinity max} and let $\delta\coloneqq\diam{B_1^*}$. It is easy to see that for any $x\in X$, $\alpha([x]_\delta)=\beta([x]_\delta)$.
	
	By Strassen's theorem (see for example \cite[Theorem 11.6.2]{dudley2018real}),
	\begin{equation}\label{eq:strassen}
	    d_{\mathrm{W},\infty}(\alpha,\beta)=\inf\{r\geq 0|\,\text{for any closed subset }A\subseteq X,\,\alpha(A)\leq \beta(A^r)\},
	\end{equation}
	    where $A^r\coloneqq\{x\in X|\,u_X(x,A)\leq r\}$.
	    
	    Since $\alpha(B_1)\neq\beta(B_1)$, we assume without loss of generality that $\alpha(B_1)>\beta(B_1)$. By definition of $B_1^*$, it is obvious that $(B_1)^\delta=B_1^*$ (recall: $\delta\coloneqq\diam{B_1^*}$) and $(B_1)^r=B_1$ for all $0\leq r<\delta$. Therefore, $\alpha(B_1)\leq \beta((B_1)^r)$ only when $r\geq\delta$. By \Cref{eq:strassen}, this implies that $d_{\mathrm{W},\infty}(\alpha,\beta)\geq \delta$. Conversely, for any closed set $A$, we have that $A^\delta=\bigcup_{x\in A}[x]_\delta$. For two closed balls in ultrametric spaces, either one includes the other or they have no intersection. Therefore, there exists a subset $S\subseteq A$ such that $[x]_\delta\cap[x']_\delta=\emptyset$ for all $x,x'\in S$ and $x\neq x'$, and that $A^\delta=\bigsqcup_{x\in S}[x]_\delta$. Then, $\alpha(A)\leq \alpha(A^\delta)=\sum_{x\in S}\alpha([x]_\delta)=\sum_{x\in S}\beta([x]_\delta)=\beta(A^\delta)$. Hence, $d_{\mathrm{W},\infty}(\alpha,\beta)\leq\delta$ and thus 
		\[d_{\mathrm{W},\infty}(\alpha,\beta)=\max_{B\in V(X)\backslash\{X\}\text{ and }\alpha(B)\neq\beta(B)}\diam{B^*}. \]

\subsection{Technical issues from Section \ref{sec:preliminaries}}

In the following, we address various technical issues from \Cref{sec:preliminaries}.
\subsubsection{Synchronized rooted trees}\label{sec:synchronized rooted tree}
	A \emph{synchronized rooted tree}, is a combinatorial tree $T=(V,E)$ with a root $o\in V$ and a height function $h:V\rightarrow[0,\infty)$ such that $h^{-1}(0)$ coincides with the leaf set and $h(v)< h(v^*)$ for each $v\in V\backslash\{o\}$, where $v^*$ is the parent of $v$. Similar as in \Cref{thm:compact ultra-dendro} that there exists a correspondence between ultrametric spaces and dendrograms, an ultrametric space $X$ uniquely determines a synchronized rooted tree $T_X$ \citep{kloeckner2015geometric}. 
	
	Now given a compact ultrametric space $(X,\uX)$, we construct the corresponding sychronized rooted tree $T_X$ via the dendrogram $\theta_X$ associated with $\uX$. Recall from \Cref{sec:explicit formulat} that $V(X)\coloneqq\bigcup_{t>0}\theta_X(t)$. For each $B\in V(X)\backslash\{X\}$, denote by $B^*$ the smallest element in $V(X)$ such that $B\subsetneqq B^*$, whose existence is guaranteed by the following lemma:
\begin{lemma}\label{lemma:existence of B^*}
		Let $X$ be a compact ultrametric space and let $V(X)=\bigcup_{t>0}\theta_X(t)$, where $\theta_X$ is as defined in \Cref{rem:corresponding dendrogram}. For each $B\in V(X)$ such that $B\neq X$, there exists $B^*\in V(X)$ such that $B^*\neq B$ and $B^*\subseteq B'$ for all $B'\in V(X)$ with $B\subsetneqq B'$.
	\end{lemma}
	\begin{proof}
	Let $\delta\coloneqq\diam{B}$. Let $x\in B$, then $B=[x]_\delta$. By \Cref{lm:quotient-finite}, $X_\delta$ is a finite set. Consider $\delta^*\coloneqq\min\{u_{X_\delta}([x]_\delta,[x']_\delta)|\,[x']_\delta\neq [x]_\delta\}$. Let $B^*\coloneqq [x]_{\delta^*}$, then $B^*$ is the smallest element in $V(X)$ containing $B$ under inclusion. Indeed, $B^*\neq B$ and if $B\subseteq B'$ for some $B'\in V(X)$, then $B'=[x]_r$ for some $r> \delta$. It is easy to see that for all $\delta<r<\delta^*$, $[x]_r=[x]_\delta$. Therefore, if $B'\neq B$, we must have that $r\geq \delta^*$ and thus $B^*=[x]_{\delta^*}\subseteq[x]_r=B'$.
	\end{proof}
	
	Now, we define a combinatorial tree $T_X=(V_X,E_X)$ as follows: we let $V_X\coloneqq V(X)$; for any distinct $B,B'\in V_X$, we let $(B,B')\in E_X$ iff either $B=(B')^*$ or $B'=B^*$. We choose $X\in V_X$ to be the root of $T_X$, then any $B\neq X$ in $V_X$ has a unique parent $B^*$. We define $h_X:V_X\rightarrow[0,\infty)$ such that $h_X(B)\coloneqq\frac{\diam B}{2}$ for any $B\in V_X$. Now, $T_X$ endowed with the root $X$ and the height function $h_X$ is a synchronized rooted tree. It is easy to see that $X$ can be isometrically identified with $h_X^{-1}(0)$ of the so-called \emph{metric completion} of $T_X$ (see \cite[Section 2.3]{kloeckner2015geometric} for details). With this construction \Cref{lemma:Wasserstein on ultrametric spaces} follows directly from \cite[Lemma 3.1]{kloeckner2015geometric}.
	
	%%%%%%%%%%%%%%%%%%%%%%%%%%%%%%
	\subsection{\texorpdfstring{$d^{(\Rp,\Lambda_\infty)}_{\mathrm{W},p}$}{The Wasserstein distance on the ultrametric real line} between compactly supported measures}\label{sec:extension to compactly supported measures}
	Next, we demonstrate that \Cref{thm:closed-form-w-infty-real} extends naturally to the case of compactly supported probability measures in $(\Rp,\Lambda_\infty)$. For this purpose, it is important to note that compact subsets of $(\Rp,\Lambda_\infty)$ have a very particular structure as shown by the subsequent lemma. 
\begin{lemma}\label{lm:compact of R}
Let $X\subseteq\mathbb (\Rp,\Lambda_\infty)$. $X$ is a compact subset if and only if $X$ is either a finite set or a countable set with $0$ being the unique cluster point (w.r.t. the usual Euclidean distance $\Lambda_1$).
\end{lemma}
\begin{proof}
     If $X$ is finite, then obviously $X$ is compact. Assume that $X$ is a countable set with $0$ being the unique cluster point (w.r.t. the usual Euclidean distance $\Lambda_1$). If $\{x_n\}_{n\in\mathbb{N}}\subseteq X$ is a Cauchy sequence with respect to $\Lambda_\infty$, then either $x_n$ is a constant when $n$ is large or $\lim_{n\rightarrow\infty}x_n=0$. In either case, the limit of $\{x_n\}_{n\in\mathbb{N}}$ belongs to $X$ and thus $X$ is complete. Now for any $\eps>0$, by Lemma \ref{lm:quotient-finite}, $X_\eps$ is a finite set. Denote $X_\eps=\{[x_1]_\eps,\ldots,[x_n]_\eps\}$. Then, $\{x_1,\ldots,x_n\}$ is a finite $\eps$-net of $X$. Therefore, $X$ is totally bounded and thus $X$ is compact.
	    
	    Now, assume that $X$ is compact. Then, for any $\eps>0$, $X_\eps$ is a finite set. Suppose $X_\eps=\{[x_1]_\eps,\ldots,[x_n]_\eps\}$ where $0\leq x_1<x_2<\cdots<x_n$. Further, we have that $\Lambda_\infty(x_i,x_j)=x_j$ whenever $1\leq i<j\leq n$. This implies that
	    \begin{enumerate}
	        \item $x_i>\eps$ for all $2\leq i\leq n$;
	        \item $[x_i]_\eps=\{x_i\}$ for all $2\leq i\leq n$.
	    \end{enumerate}
	    Therefore, $X\cap (\eps,\infty)=\{x_2,\ldots,x_n\}$ is a finite set. Since $\eps>0$ is arbitrary, $X$ is an at most countable set and has no cluster point (w.r.t. the usual Euclidean distance $\Lambda_1$) other than $0$. If $X$ is countable, then $0$ must be a cluster point and by compactness of $X$, we have that $0\in X$.	
	    \end{proof}
Based on the special structure of compact subsets of $(\Rp,\Lambda_\infty)$, we derive the following extension of \Cref{thm:closed-form-w-infty-real}. 
\begin{theorem}[$d^{(\Rp,\Lambda_\infty)}_{\mathrm{W},p}$ between compactly supported measures]\label{thm:closed-form-w-infty-real-compact}
		Suppose $\alpha,\beta$ are supported on a countable subset $X\coloneqq \{0\}\cup\{x_i|\,i\in\mathbb N\}$ of $\mathbb{R}_{\geq 0}$ such that $0<\ldots< x_n<x_{n-1}<\ldots<x_1$ and $0$ is the only cluster point with respect to the usual Euclidean distance. Let $\alpha_i\coloneqq \alpha(\{x_i\})$ for $i\in\mathbb N$ and $\alpha_0\coloneqq\alpha(\{0\})$. Similarly, let $\beta_i\coloneqq \beta(\{x_i\})$ and $\beta_0\coloneqq\beta(\{0\})$. Then for $p\in[1,\infty)$,
		\begin{equation}\label{eq:dp} d^{(\Rp,\Lambda_\infty)}_{\mathrm{W},p}(\alpha,\beta)=2^{-\frac{1}{p}}\left(\sum_{i=2}^{\infty}\left|\sum_{j=i}^\infty(\alpha_j-\beta_j)\right|\cdot|x_{i-1}^p-x_{i}^p|+\sum_{i=1}^\infty|\alpha_i-\beta_i|\cdot x_i^p\right)^\frac{1}{p}.
		\end{equation}
	Let $F_\alpha$ and $F_\beta$ denote the cumulative distribution functions of $\alpha$ and $\beta$, respectively. Then, we obtain
	\[d_{\mathrm{W},\infty}^{(\Rp,\Lambda_\infty)}(\alpha,\beta)=\max\left(\max_{2\leq i<\infty, F_\alpha(x_i)\neq F_\beta(x_i)}x_{i-1},\max_{1\leq i<\infty, \alpha_i\neq\beta_i}x_i\right).\]	
		
	\end{theorem}
	
	\begin{proof}
	Note that  $V(X)=\{\{0\}\cup\{x_j|\,j\geq i\}|\,i\in\mathbb N\}\cup\{\{x_i\}|\,i\in\mathbb{N}\}$ (recall that each set corresponds to a closed ball). Thus, we conclude the proof by applying Lemma \ref{lemma:Wasserstein on ultrametric spaces} and \Cref{lm:winfty-finite}.
	\end{proof}
	\subsubsection{Closed-form solution for \texorpdfstring{$d_{\mathrm{W},p}^{(\Rp,\Lambda_q)}$}{the Wasserstein distance on the p-metric real line}}\label{sec:closed form solution}
In the following, we will derive the subsequent theorem. 
	\begin{theorem}\label{thm:closed form solution}
		Given $1\leq p,q <\infty$ and two compactly supported probability measures $\alpha$ and $\beta$ on $\Rp$, we have that 
		\[d_{\mathrm{W},p}^{(\Rp,\Lambda_q)}(\alpha,\beta)\leq\left(\int_0^1\Lambda_q(F_\alpha^{-1}(t),F_\beta^{-1}(t))^pdt\right)^\frac{1}{p}.\]
		When $q\leq p$, the equality holds whereas when $q>p$, the equality does not hold in general.
	\end{theorem}
	
One important ingredient for the proof of \Cref{thm:closed form solution} is Lemma 3.2 of \citet{chowdhury2019gromov} which we restate here for convenience.
	\begin{lemma}[{\citet[Lemma 3.2]{chowdhury2019gromov}}]\label{lm:pushforward_coupling}
		Let $X,Y$ be two Polish metric spaces and let $f:X\rightarrow\mathbb{R}$ and $g:Y\rightarrow\mathbb{R}$ be measurable maps. Denote by $f\times g:X\times Y\rightarrow\mathbb{R}^2$ the map $(x,y)\mapsto(f(x),g(y))$. Then, for any $\muY\in\mathcal{P}(X)$ and $\muY\in\mathcal{P}(Y)$
		\[(f\times g)_\#\mathcal{C}(\muX,\muY)=\mathcal{C}(f_\#\muY,g_\#\muY). \]
	\end{lemma}
	Based on \Cref{lm:pushforward_coupling}, we can show the following auxiliary result.
	\begin{lemma}\label{lm:snowflake wasserstein pq}
	    Let $1\leq q\leq p<\infty$. Assume that $\alpha$ and $\beta$ are compactly supported probability measures on $\Rp$. Then,
	    \[\left(d_{\mathrm{W},p}^{(\Rp,\Lambda_q)}(\alpha,\beta)\right)^p =\left(d_{\mathrm{W},\frac{p}{q}}^{(\Rp,\Lambda_1)}((S_q)_\#\alpha,(S_q)_\#\beta)\right)^\frac{p}{q}, \]
	    where $S_q:\Rp\rightarrow\Rp$ taking $x$ to $x^q$ is the $q$-snowflake transform defined in \Cref{subsec:relation between ugw and usturm}.
	\end{lemma}
	\begin{proof}
		\begin{align*}
		\left(d_{\mathrm{W},p}^{(\Rp,\Lambda_q)}(\alpha,\beta)\right)^p &= \inf_{\mu\in\mathcal{C}(\alpha,\beta)}\int_{\Rp\times \Rp}(\Lambda_q(x,y))^p\mu(dx\times dy)\\
		&=\inf_{\mu\in\mathcal{C}(\alpha,\beta)}\int_{\Rp\times \Rp}|S_q(x)-S_q(y)|^\frac{p}{q}\mu(dx\times dy)\\
		&=\inf_{\mu\in\mathcal{C}(\alpha,\beta)}\int_{\Rp\times \Rp}|s-t|^\frac{p}{q}(S_q\times S_q)_\#\mu(ds\times dt)\\
		&=\left(d_{\mathrm{W},\frac{p}{q}}^{(\Rp,\Lambda_1)}((S_q)_\#\alpha,(S_q)_\#\beta)\right)^\frac{p}{q},
		\end{align*}
		where we use $\frac{p}{q}\geq 1$ and Lemma \ref{lm:pushforward_coupling} in the last equality. 
	\end{proof}
	With \Cref{lm:snowflake wasserstein pq} at our disposal, we can demonstrate \Cref{thm:closed form solution}.
	\begin{proof}[Proof of \Cref{thm:closed form solution}]
		We first note that $d_{\mathrm{W},p}^{(\Rp,\Lambda_q)}(\alpha,\beta)=\inf_{(\xi,\eta)}\big(\mathbb{E}(\Lambda_q(\xi,\eta)^p)\big)^\frac{1}{p}$, where $\xi$ and $\eta$ are two random variables with marginal distributions $\alpha$ and $\beta$, respectively. Moreover, let $\zeta$ be the random variable uniformly distributed on $[0,1]$, then $F_\alpha^{-1}(\zeta)$ has distribution function $F_\alpha$ and $F_\beta^{-1}(\zeta)$ has distribution function $F_\beta$ (see for example \citet{vallender1974calculation}). Let $\xi=F_\alpha^{-1}(\zeta)$ and $\eta=F_\beta^{-1}(\zeta)$, then we have
		\[d_{\mathrm{W},p}^{(\Rp,\Lambda_q)}(\alpha,\beta)\leq\big(\mathbb{E}(\Lambda_q(\xi,\eta)^p)\big)^\frac{1}{p}=\left(\int_0^1\Lambda_q(F_\alpha^{-1}(t),F_\beta^{-1}(t))^pdt\right)^\frac{1}{p}. \]
		
		Next, we assume that $q\leq p$. By \Cref{lm:snowflake wasserstein pq}, we have that
		\begin{align*}
		\left(d_{\mathrm{W},p}^{(\Rp,\Lambda_q)}(\alpha,\beta)\right)^p =\left(d_{\mathrm{W},\frac{p}{q}}^{(\Rp,\Lambda_1)}((S_q)_\#\alpha,(S_q)_\#\beta)\right)^\frac{p}{q}.
		\end{align*}
		Then,
		\[\left(d_{\mathrm{W},\frac{p}{q}}^{(\Rp,\Lambda_1)}((S_q)_\#\alpha,(S_q)_\#\beta)\right)^\frac{p}{q}=\int_0^1|F_{\alpha,q}^{-1}(t)-F_{\beta,q}^{-1}(t)|^\frac{p}{q}dt, \]
		where $F_{\alpha,q}$ and $F_{\beta,q}$ are distribution functions of $(S_q)_\#\alpha$ and $(S_q)_\#\beta$, respectively. It is easy to verify that  $F_{\alpha,q}(t)=(F_\alpha^{-1}(t))^q$ and $F_{\beta,q}(t)=(F_\beta^{-1}(t))^q$. Therefore, 
		\[d_{\mathrm{W},p}^{(\Rp,\Lambda_q)}(\alpha,\beta)= \left(\int_0^1\Lambda_q(F_\alpha^{-1}(t),F_\beta^{-1}(t))^pdt\right)^\frac{1}{p}\]
		
		Finally, we demonstrate that for $q>p$ the equality does not hold in general. We first consider the extreme case $p=1$ and $q=\infty$ (though we require $q<\infty$ in the assumptions of the theorem, we relax this for now). Let $\alpha_0=\frac{1}{2}\delta_1+\frac{1}{2}\delta_2$ and $\beta_0 = \frac{1}{2}\delta_2+\frac{1}{2}\delta_3$ where $\delta_x$ means the Dirac measure at point $x\in\Rp$. Then, we have that 
		\[d_{\mathrm{W},1}^{(\Rp,\Lambda_\infty)}(\alpha_0,\beta_0)=\frac{3}{2}<\frac{5}{2}=\int_0^1\Lambda_\infty(F_\alpha^{-1}(t),F_\beta^{-1}(t))dt. \]
		It is not hard to see that both $d_{\mathrm{W},p}^{(\Rp,\Lambda_q)}(\alpha_0,\beta_0)$ and $\left(\int_0^1\Lambda_q(F_\alpha^{-1}(t),F_\beta^{-1}(t))^pdt\right)^\frac{1}{p}$ are continuous with respect to $p\in[1,\infty)$ and $q\in[1,\infty]$. Then, for $p$ close to 1 and $q<\infty$ large enough, and in particular, $p<q$, we have that
		\[d_{\mathrm{W},p}^{(\Rp,\Lambda_q)}(\alpha_0,\beta_0)< \left(\int_0^1\Lambda_q(F_\alpha^{-1}(t),F_\beta^{-1}(t))^pdt\right)^\frac{1}{p}.\]
		
	\end{proof}
	%%%%%%%%%%%%%%%%%%%%%%%%%
\subsubsection{Miscellaneous}
In the remainder of this section, we collect several technical results that find implicit or explicit usage throughout \Cref{sec:preliminaries}. 
\begin{lemma}\label{lm:quotient-finite}
Let	$X$ be a complete ultrametric space. Then, $X$ is compact ultrametric space if and only if for any $t>0$, $X_t$ is a finite space. 
\end{lemma}
\begin{proof}
    {\citet[Lemma 2.3]{wan2020novel}} proves that whenever $X$ is compact, $X_t$ is finite for any $t>0$.
	    
	    Conversely, we assume that $X_t$ is finite for any $t>0$. We only need to prove that $X$ is totally bounded. For any $\eps>0$, $X_\eps$ is a finite set and thus there exists $x_1,\ldots,x_n\in X$ such that $X_\eps=\left\{[x_1]_\eps,\ldots,[x_n]_\eps\right\}.$ Now, for any $x\in X$, there exists $x_i$ for some $i=1,\ldots,n$ such that $x\in[x_i]_\eps$. This implies that $u_X(x,x_i)\leq \eps$. Therefore, the set $\{x_1,\ldots,x_n\}\subseteq X$ is an $\eps$-net of $X$. Then, $X$ is totally bounded and thus compact. 
\end{proof}

\begin{lemma}\label{lm:vx characterization}
$V(X)$ is the collection of all closed balls in $X$ except for singletons $\{x\}$ such that $x$ is a cluster point in $X$. In particular, $X\in V(X)$ and for any $x\in X$, if $x$ is not a cluster point, then $\{x\}\in V(X)$.
	\end{lemma}
	\begin{proof}
	    Given any $t>0$ and $x\in X$, $[x]_t=B_t(x)=\{x'\in X|\, u_X(x,x')\leq t\}$. Therefore, $V(X)$ is a collection of closed balls in $X$. On the contrary, any closed ball $B_t(x)$ with positive radius $t>0$ coincides with $[x]_t\in\theta_X(t)$ and thus belongs to $V(X)$. Now, for any singleton $\{x\}=B_0(x)$. If $x$ is not a cluster point, then there exists $t>0$ such that $B_t(x)=\{x\}$ which implies that $\{x\}\in V(X)$. If $x$ is a cluster point, then for any $t>0$, $\{x\}\subsetneqq B_t(x)=[x]_t$. In particular, this implies that $\{x\}\neq [x]_t$ for all $t>0$ and thus $\{x\}\notin V(X)$. In conclusion, $V(X)$ is the collection of all closed balls in $X$ except for singletons $\{x\}$ such that $x$ is a cluster point in $X$.
	    
	    If $X$ is a one point space, then obviously $X\in V(X)=\{X\}$. Otherwise, let $\delta\coloneqq\diam X>0$, then for any $x\in X$ we have that $X=[x]_\delta\in V(X)$. As for singletons $\{x\}$ where $x\in X$ is not a cluster point, we have proved above that $\{x\}\in V(X)$.
	\end{proof}

%%%%%%%%%%%%%%%%%%%%%%%%%%%%
\section{Missing details from Section \ref{sec:ultrametric GW distance}}\label{sec:details sec3}

%%%%%%%%%%%%%%%%%%%%%%%%%%%%%%%%%%%%%
\subsection{Proofs from Section \ref{subsec:Sturms ultrametric GW distance}}\label{sec:proofs sec31}
Next, we give the missing proofs of the results stated in \Cref{subsec:Sturms ultrametric GW distance}.
\subsubsection{Proof of \Cref{prop:usturm_basic}}\label{sec:proof of prop usturm_basic}
		\begin{enumerate}
			\item This directly follows from the definitions of $\usturm{p}$ and $\dsturm{p}$ (see \Cref{eq:ultra Sturm} and  \Cref{eq:stdSturm}).
			\item This simply follows from Jensen's inequality.
			
			\item By (2), we know that $\{\usturm{n}(\X,\Y)\}_{n\in\mathbb{N}}$ is an increasing sequence with a finite upper bound $\usturm{\infty}(\X,\Y)$. Therefore, $L\coloneqq\lim_{n\rightarrow\infty}\usturm{n}(\X,\Y)$ exists and $L\leq \usturm{\infty}(\X,\Y)$. 
			
			Next, we come to the opposite inequality. By \Cref{prop:usturm_optimal}, there exist $u_n\in\mathcal{D}^\mathrm{ult}(\uX,\uY)$ and $\mu_n\in\mathcal{C}(\muX,\muY)$ such that 
			\[\left(\int_{X\times Y}(u_n(x,y))^n\mu_n(dx\times dy)\right)^\frac{1}{n}= \usturm{n}(\X,\Y). \]
			By \Cref{lm:measure coupling compact} and \Cref{lm:metric coupling bounded integral}, the sequence $\{u_n\}_{n\in\N}$ uniformly converges to some $u\in\mathcal{D}^\mathrm{ult}(\uX,\uY)$ and $\{\mu_n\}_{n\in\N}$ weakly converges to some $\mu\in\mathcal{C}(\muX,\muY)$ (after taking appropriate subsequences of both sequences). Let $M:=\sup_{(x,y)\in \supp{\mu}}u(x,y)$. Let $\eps>0$ and let $U=\{(x,y)\in X\times Y\,|\,u(x,y)> M-\eps \}.$ Then, $\mu(U)>0$. Since $U$ is open, it follows that there exists a small $\eps_1>0$ such that $\mu_n(U)>\mu(U)-\eps_1>0$ for all $n$ large enough (see e.g. \citet[Thm. 2.1]{billingsleyConvergenceProbabilityMeasures2013}). Moreover, by uniform convergence of the sequence $\{u_n\}_{n\in\N}$, we have $|u(x,y)-u_n(x,y)|\leq\eps$ for any $(x,y)\in X\times Y$ when $n$ is large enough. Therefore, we obtain for $n$ large enough
			\begin{align*}
			\left(\int_{X\times Y}(u_n(x,y))^n\mu_n(dx\times dy)\right)^\frac{1}{n}\geq (\mu_n(U))^\frac{1}{n}(M-2\eps)\geq(\mu(U)-\eps_1)^\frac{1}{n}(M-2\eps).
			\end{align*}
			Letting $n\rightarrow\infty$, we obtain $L\geq M-2\eps$. Since $\eps>0$ is arbitrary, we obtain $L\geq M\geq \usturm{\infty}(\X,\Y)$.
		\end{enumerate}

	\subsubsection{Proof of \Cref{thm:sturms um}}\label{sec:proof of thm sturms um}
In this section, we devote to prove \Cref{thm:sturms um}. To this end, we will first verify the existence of optimal metrics and optimal couplings in \Cref{eq:pseudometric def of usturm}.

\begin{proposition}[Existence of optimal couplings]\label{prop:usturm_optimal}
	Let $\X=\ummspaceX$ and $\Y=\ummspaceY$ be {compact} ultrametric measure spaces. Then, there always exist $u\in\mathcal{D}^\mathrm{ult}(\uX,\uY)$ and $\mu\in\mathcal{C}(\muX,\muY)$ such that for $1\leq p<\infty$
		\[\usturm{p}(\X,\Y)=\left(\int_{X\times Y}(u(x,y))^p\mu(dx\times dy)\right)^\frac{1}{p}\]
		and such that 
		\[\usturm{\infty}(\X,\Y)=\sup_{(x,y)\in\supp{\mu}}u(x,y).\]
	\end{proposition}
	\begin{proof}
	    The following proof is a suitable adaptation from proof of Lemma 3.3 in \cite{sturm2006geometry}. We will only prove the claim for the case $p<\infty$ since the case $p=\infty$ can be shown in a similar manner. Let $u_n\in\mathcal{D}^\mathrm{ult}(\uX,\uY)$ and $\mu_n\in\mathcal{C}(\muX,\muY)$ be such that 
		\[\left(\int_{X\times Y}(u_n(x,y))^p\mu_n(dx\times dy)\right)^\frac{1}{p}\leq \usturm{p}(\X,\Y)+\frac{1}{n}. \]
		By \Cref{lm:measure coupling compact}, $\{\mu_n\}_{n\in\N}$ weakly converges (after taking an appropriate subsequence) to some $\mu\in\mathcal{C}(\muX,\muY)$. By \Cref{lm:metric coupling bounded integral}, $\{u_n\}_{n\in\N}$ uniformly converges (after taking an appropriate subsequence) to some $u\in\mathcal{D}^\mathrm{ult}(\uX,\uY)$. Then, it is easy to verify that
		\[\left(\int_{X\times Y}(u(x,y))^p\mu(dx\times dy)\right)^\frac{1}{p}\leq \usturm{p}(\X,\Y). \]
	\end{proof}
	
	As a direct consequence of the proposition, we get the subsequent result.
	\begin{corollary}\label{coro:usturm_optimal}
	Fix $1\leq p\leq\infty$. Let $\X=\ummspaceX$ and $\Y=\ummspaceY$ be {compact} ultrametric measure spaces. Then, there exist a compact ultrametric space $Z$ and isometric embeddings $\phi:X\hookrightarrow Z$ and $\psi:Y\hookrightarrow Z$ such that
	\[\usturm{p}(\X,\Y)=d_{\mathrm{W},p}^Z(\phi_\#\muX,\psi_\#\muY).\]
	\end{corollary}
	
	Before we come to the proof of \Cref{thm:sturms um}, it remains to establish another auxiliary result. We ensure that the Wasserstein pseudometric of order $p$ on a compact pseudo-ultrametric space $(X,u_X)$ is for $p\in[1,\infty)$ a $p$-pseudometric and for $p=\infty$ a pseudo-ultrametric, i.e., we prove for $1\leq p<\infty$ that for all $\alpha_1,\alpha_2,\alpha_3\in \mathcal{P}(X)$
	\[\pseudoWasser{p}^{(X,u_X)}(\mu_1,\mu_3)\leq\left(\left(\pseudoWasser{p}^{(X,u_X)}(\mu_1,\mu_2)\right)^p+\left(\pseudoWasser{p}^{(X,u_X)}(\mu_2,\mu_3)\right)^p \right)^{1/p}\]
	and for $p=\infty$ that for all $\alpha_1,\alpha_2,\alpha_3\in \mathcal{P}(X)$
    \[\pseudoWasser{p}^{(X,u_X)}(\mu_1,\mu_3)\leq\max\left(\pseudoWasser{p}^{(X,u_X)}(\mu_1,\mu_2),\pseudoWasser{p}^{(X,u_X)}(\mu_2,\mu_3) \right).\]
	\begin{lemma}\label{lemma:wasserstein p-metrics}
		Let $(X,\uX)$ be a compact pseudo-ultrametric space. Then, for $1\leq p\leq\infty $ the $p$-Wasserstein metric $\pseudoWasser{p}^{(X,\uX)}$ is a $p$-pseudometric on $\mathcal{P}(X)$. In particular, when $p=\infty$, it is an pseudo-ultrametric on $\mathcal{P}(X)$.  \end{lemma}
	\begin{proof} We prove the statement by adapting the proof of the triangle inequality for the $p$-Wasserstein distance (see e.g.,  \cite[Theorem 7.3]{villani2003topics}). We only prove the case when $p<\infty$ whereas the case $p=\infty$ follows by analogous arguments.\\
		
	Let $\alpha_1,\alpha_2,\alpha_3\in\mathcal{P}(X)$, denote by $\mu_{12}$ an optimal transport plan between $\alpha_1$ and $\alpha_2$ and by $\mu_{23}$ an optimal transport plan between $\alpha_2$ and $\alpha_3$ (see \cite[Theorem 4.1]{villani2008optimal} for the existence of $\mu_{12}$ and $\mu_{23}$). Furthermore, let $X_i$ be the support of $\alpha_i$, $1\leq i \leq 3$. Then, by the Gluing Lemma \citep[Lemma 7.6]{villani2003topics}
		there exists a measure $\mu\in\mathcal{P}(X_1\times X_2\times X_3)$ with marginals $\mu_{12}$ on $X_1\times X_2$ and $\mu_{23}$ on $X_2\times X_3$. Clearly, we obtain
		\begin{align*}
		\left(\pseudoWasser{p}^{(X,\uX)}(\alpha_1,\alpha_3)\right)^p \leq \int_{X_1\times X_2\times X_3}\! \uX^p\left(x,z\right)\,\mu(dx\times dy\times dz)\\ \leq \int_{X_1\times X_2\times X_3}\!\left( \uX^p\left(x,y\right)+\uX^p\left(y,z\right)\right)\,\mu(dx\times dy\times dz).
		\end{align*}
		Here, we used that $\uX$ is an ultrametric, i.e., in particular a $p$-metric \citep[Proposition 1.16]{memoli2019gromov}. With this we obtain that
		\begin{align*}
		\left(\pseudoWasser{p}^{(X,\uX)}(\alpha_1,\alpha_2)\right)^p\leq \int_{X_1\times X_2}\! \uX^p\left(x,y\right)\,\mu_{12}(dx\times dy)+ \int_{X_2\times X_3}\! \uX^p\left(y,z\right)\,\mu_{23}(dy\times dz)\\
		= \left(\pseudoWasser{p}^{(X,\uX)}(\alpha_1,\alpha_2)\right)^p+\left(\pseudoWasser{p}^{(X,\uX)}(\alpha_2,\alpha_3)\right)^p.
		\end{align*}
	\end{proof}
	With \Cref{prop:usturm_optimal} and \Cref{lemma:wasserstein p-metrics} at our disposal we are now ready to prove \Cref{thm:sturms um} which states that $\usturm{p}$ is indeed a $p$-metric on $\mathcal{U}^w$.

	\begin{proof}[Proof of \Cref{thm:sturms um}]
		It is clear that $\usturm{p}$ is symmetric and that $\usturm{p}(\X,\Y)=0$ if $\X\cong_w\Y$. Furthermore, we remark that $\usturm{p}(\X,\Y)\geq d_{\mathrm{GW},p}^{\mathrm{sturm}}(\X,\Y)$ by \Cref{prop:usturm_basic}. Since $d_{\mathrm{GW},p}^{\mathrm{sturm}}(\X,\Y)=0$ implies that $\X\cong_w\Y$ (\cite{sturm2012space}), we have that $\usturm{p}(\X,\Y)=0$ implies that $\X\cong_w\Y$. It remains to verify the $p$-triangle inequality. To this end, we only prove the case when $p<\infty$ whereas the case $p=\infty$ follows by analogous arguments.
		
		Let $\X,\Y,\Z\in\mathcal{U}^w$. Suppose $u_{XY}\in\mathcal{D}^\mathrm{ult}(\uX,\uY)$ and $u_{YZ}\in\mathcal{D}^\mathrm{ult}(\uY,\uZ)$ are optimal metric couplings such that
		\[\left(\usturm{p}(\X,\Y)\right)^p=\left(\pseudoWasser{p}^{(X\sqcup  Y,u_{XY})}(\muX,\muY)\right)^p \text{  and } \left(\usturm{p}(\Y,\Z)\right)^p=\left(\pseudoWasser{p}^{(Y\sqcup  Z,u_{YZ})}(\muY,\mu_Z)\right)^p.\]
		Further, define $u_{XYZ}$ on $X\sqcup Y\sqcup  Z$ as
		\[u_{XYZ}(x_1,x_2)=\begin{cases}u_{XY}(x_1,x_2) &x_1,x_2\in X\sqcup Y\\
		u_{YZ}(x_1,x_2) &x_1,x_2\in Y\sqcup Z\\
		\inf\{\max (u_{XY}(x_1,y),u_{YZ}(y,x_2))\,|\,y\in Y\} &x_1\in X,x_2\in Z\\
		\inf\{\max (u_{XY}(x_2,y),u_{YZ}(y,x_1))\,|\,y\in Y\} &x_1\in Z,x_2\in X.
		\end{cases}\]
		Then, by Lemma 1.1 of \citet{zarichnyi2005gromov} $u_{XYZ}$ is a pseudo-ultrametric on $X\sqcup Y\sqcup Z$ that coincides with $u_{XY}$ on $X\sqcup Y$ and with $u_{YZ}$ on $Y\sqcup Z$. With this we obtain by \Cref{lemma:wasserstein p-metrics} that 
		\begin{align*}
		&\left(\usturm{p}(\X,\Z)\right)^p\leq \left(\pseudoWasser{p}^{(X\sqcup Y\sqcup Z,u_{XYZ})}(\muX,\mu_Z)\right)^p
		\\ \leq& \left(\pseudoWasser{p}^{(X\sqcup Y\sqcup Z,u_{XYZ})}(\muX,\muY)\right)^p+\left(\pseudoWasser{p}^{(X\sqcup Y\sqcup Z,u_{XYZ})}(\muY,\mu_Z)\right)^p\\
		=& \left(\pseudoWasser{p}^{(X\sqcup Y,u_{XY})}(\muX,\muY)\right)^p+\left(\pseudoWasser{p}^{( Y\sqcup Z,u_{YZ})}(\muY,\mu_Z)\right)^p\\
		=&\left(\usturm{p}(\X,\Y)\right)^p+\left(\usturm{p}(\Y,\Z)\right)^p
		\end{align*}
		This gives the claim for $p<\infty$.
	\end{proof}

\subsubsection{Proof of \Cref{thm:compact usturm A Phi representation}}\label{sec:proof of compact usturm a phi}
In order to proof \Cref{thm:compact usturm A Phi representation}, we will first establish the statement for \emph{finite} ultrametric measure spaces. For this purpose, we need to introduce some notation. Given $\X,\Y\in\mathcal{U}^w$, let $\mathcal{D}^\mathrm{ult}_\mathrm{adm}(\uX,\uY)$ denote the collection of all admissible pseudo-ultrametrics on $X\sqcup Y$, where $u\in \mathcal{D}^\mathrm{ult}(\uX,\uY)$ is called \emph{admissible}, if there exists no $u^*\in \mathcal{D}^\mathrm{ult}(\uX,\uY)$ such that $u^*\neq u$ and $u^*(x,y)\leq u(x,y)$ for all $x,y\in X\sqcup Y$.
	
\begin{lemma}\label{lm:adm non empty}
For any $\X,\Y\in\mathcal{U}^w$, $\mathcal{D}^\mathrm{ult}_\mathrm{adm}(\uX,\uY)\neq\emptyset$. Moreover, 
	\[\usturm{p}(\X,\Y)=\inf_{u \in \mathcal{D}^\mathrm{ult}_\mathrm{adm}(\uX,\uY)} \pseudoWasser{p}^{(X\sqcup Y,u)}(\muX,\muY).\]
	\end{lemma}
	
	\begin{proof}
	     If $\{u_n\}_{n\in\mathbb N}\subseteq \mathcal{D}^\mathrm{ult}(\uX,\uY)$ is a decreasing sequence (with respect to pointwise inequality), it is easy to verify that $u\coloneqq\inf_{n\in\mathbb N}u_n\in\mathcal{D}^\mathrm{ult}(\uX,\uY)$ and thus $u$ is a lower bound of $\{u_n\}_{n\in\mathbb N}$. Then, by Zorn's lemma $\mathcal{D}^\mathrm{ult}_\mathrm{adm}(\uX,\uY)\neq \emptyset$. Therefore, we obtain that
\[\usturm{p}(\X,\Y)=\inf_{u \in \mathcal{D}^\mathrm{ult}_\mathrm{adm}(\uX,\uY)} \pseudoWasser{p}^{(X\sqcup Y,u)}(\muX,\muY).\]
	\end{proof}

	Combined with \Cref{ex:u=0 A}, the following result implies that each $u\in\mathcal{D}^\mathrm{ult}_\mathrm{adm}(\uX,\uY)$ gives rise to an element in $\mathcal{A}$.
	
	\begin{lemma}\label{lm:u-1 neq empty}
	Given \emph{finite} spaces $\X,\Y\in\mathcal{U}^w$, for each $u\in \mathcal{D}^\mathrm{ult}_\mathrm{adm}(\uX,\uY)$, $u^{-1}(0)\neq\emptyset$.
	\end{lemma}
	
	\begin{proof}
	     Assume otherwise that $u^{-1}(0)=\emptyset$. Then, $u$ is a metric (instead of pseudo-metric). Let $(x_0,y_0)\in X\times Y$ such that $u(x_0,y_0)=\min_{x\in X,y\in Y}u(x,y)$. The existence of $(x_0,y_0)$ is guaranteed by the finiteness of $X$ and $Y$.
We define $ u_{(x_0,y_0)}:X\sqcup Y\times X\sqcup Y\rightarrow\mathbb R_{\geq 0}$ as follows:
\begin{enumerate}
    \item $ u_{(x_0,y_0)}|_{X\times X}\coloneqq u_X$ and $ u_{(x_0,y_0)}|_{Y\times Y}\coloneqq u_Y$;
    \item For $(x,y)\in X\times Y$,
    \[ u_{(x_0,y_0)}(x,y)\coloneqq\min\left(u(x,y),\max(u_X(x,x_0),u_Y(y,y_0))\right); \]
    \item For any $(y,x)\in Y\times X$, $ u_{(x_0,y_0)}(y,x)\coloneqq  u_{(x_0,y_0)}(x,y)$.
\end{enumerate}

It is easy to verify that $u_{(x_0,y_0)}\in\mathcal{D}^\mathrm{ult}(\uX,\uY)$. Further, it is obvious that $u_{(x_0,y_0)}(x_0,y_0)=0<u(x_0,y_0)$ and that $u_{(x_0,y_0)}(x,y)\leq u(x,y)$ for all $x,y\in X\sqcup Y$ which contradicts with $u\in \mathcal{D}^\mathrm{ult}_\mathrm{adm}(\uX,\uY)$. Therefore, $u^{-1}(0)\neq\emptyset$.
	\end{proof}

\begin{theorem}\label{thm:usturm A Phi representation}
	Let $\X,\Y\in\mathcal{U}^w$ be \emph{finite} spaces. Then, we have for each $p\in[1,\infty)$ that
\begin{equation}\label{eq:usturm_pair}
\usturm{p}(\X,\Y)=\inf_{(A,\varphi)\in\mathcal{A}}d_{\mathrm{W},p}^{Z_A}\left({\left(\phi^X_{(A,\varphi)}\right)}_\#\muX,{\left(\psi^Y_{(A,\varphi)}\right)}_\#\muY\right).\end{equation}
	\end{theorem}
	
	\begin{proof}
	    By \Cref{lm:adm non empty} it is sufficent to prove that each $u\in \mathcal{D}^\mathrm{ult}_\mathrm{adm}(\uX,\uY)$ induces $(A,\varphi)\in\mathcal{A}$ such that
\[ \pseudoWasser{p}^{(X\sqcup Y,u)}(\muX,\muY)\geq d_{\mathrm{W},p}^{Z_{A}}\left({\left(\phi^X_{(A,\varphi)}\right)}_\#\muX,{\left(\psi^Y_{(A,\varphi)}\right)}_\#\muY\right).\]
 Let $u\in \mathcal{D}^\mathrm{ult}_\mathrm{adm}(\uX,\uY)$.  We define $A_0\coloneqq\{x\in X|\,\exists y\in Y \text{ such that }u(x,y)=0\}$ ($A_0\neq\emptyset$ by \Cref{lm:u-1 neq empty}). By \Cref{ex:u=0 A}, the map $\varphi_0:A_0\rightarrow Y$ defined by taking $x$ to $y$ such that $u(x,y)=0$ is a well-defined isometric embedding. This means in particular that $(A_0,\varphi_0)\in \mathcal{A}$. 

If $u(x,y)\geq u_{Z_{A_0}}\left(\phi^X_{(A_0,\varphi_0)}(x),\psi^Y_{(A_0,\varphi_0)}(y)\right)$ holds for all $(x,y)\in X\times Y$, then we set $A\coloneqq A_0$ and $\varphi\coloneqq\varphi_0$. This gives
\[ \pseudoWasser{p}^{(X\sqcup Y,u)}(\muX,\muY)\geq d_{\mathrm{W},p}^{Z_{A}}\left({\left(\phi^X_{(A,\varphi)}\right)}_\#\muX,{\left(\psi^Y_{(A,\varphi)}\right)}_\#\muY\right).\]

Otherwise, there exists $(x,y)\in X\backslash A_0\times Y\backslash \varphi_0(A_0)$ such that \[u(x,y)< u_{Z_{A_0}}\left(\phi^X_{(A_0,\varphi_0)}(x),\psi^Y_{(A_0,\varphi_0)}(y)\right)\] (if $x\in A_0$ or $y\in \varphi_0(A_0)$, then $u(x,y)\geq u_{Z_{A_0}}\left(\phi^X_{(A_0,\varphi_0)}(x),\psi^Y_{(A_0,\varphi_0)}(y)\right)$ must hold). Let $(x_1,y_1)\in X\backslash A_0\times Y\backslash \varphi_0(A_0)$ be such that 
\begin{align*}
u(x_1,y_1)=\min&\Big\{ u(x,y)|\,(x,y)\in X\backslash A_0\times Y\backslash \varphi_0(A_0)\\&\text{ and }u(x,y)< u_{Z_{A_0}}\left(\phi^X_{(A_0,\varphi_0)}(x),\psi^Y_{(A_0,\varphi_0)}(y)\right)\Big\}>0.\end{align*}
The existence of $(x_1,y_1)$ follows from finiteness of $X$ and $Y$. It is easy to check that $\varphi_0$ extends to an isometry from $A_0\cup\{x_1\}$ to $\varphi_0(A_0)\cup\{y_1\}$ by taking $x_1$ to $y_1$. We denote the new isometry $\varphi_1$ and set $A_1\coloneqq A_0\cup\{x_1\}$. If for any $(x,y)\in X\times Y$, we have that $u(x,y)\geq u_{Z_{A_1}}\left(\phi^X_{(A_1,\varphi_1)}(x),\psi^Y_{(A_1,\varphi_1)}(y)\right)$, then we define $A\coloneqq A_1$ and $\varphi\coloneqq\varphi_1$. Otherwise, we continue the process to obtain $A_2, A_3,\dots$. This process will eventually stop since we are considering finite spaces. Suppose the process stops at $A_n$, then $A\coloneqq A_n$ and $\varphi\coloneqq\varphi_n$ satisfy that $u(x,y)\geq u_{Z_{A}}\left(\phi^X_{(A,\varphi)}(x),\psi^Y_{(A,\varphi)}(y)\right)$ for any $(x,y)\in X\times Y$. Therefore,
\[ \pseudoWasser{p}^{(X\sqcup Y,u)}(\muX,\muY)\geq d_{\mathrm{W},p}^{Z_{A}}\left({\left(\phi^X_{(A,\varphi)}\right)}_\#\muX,{\left(\psi^Y_{(A,\varphi)}\right)}_\#\muY\right).\] 
Since $u\in \mathcal{D}^\mathrm{ult}_\mathrm{adm}(\uX,\uY)$ is arbitrary, this gives the claim.
	\end{proof}

As a direct consequence of \Cref{thm:usturm A Phi representation}, we obtain that it is sufficient, as claimed in \Cref{rem:computation of usturm}, for finite spaces to infimize in \Cref{eq:usturm_pair} over the collection of all maximal pairs $\mathcal{A}^*\subseteq\mathcal{A}$. Recall that a pair $(A,\varphi_1)\in\mathcal{A}$ is denoted as \emph{maximal}, if for all pairs $(B,\varphi_2)\in \mathcal{A}$ with $A\subseteq B$ and $\varphi_2|_A=\varphi_1$ it holds $A=B$. 

\begin{corollary}\label{coro:usturm A Phi representation}
		Let $\X,\Y\in\mathcal{U}^w$ be \emph{finite} spaces. Then, we have for each $p\in[1,\infty]$ that
\begin{equation}\label{eq:usturm_maxiam_pair}
\usturm{p}(\X,\Y)=\inf_{(A,\varphi)\in\mathcal{A}^*}d_{\mathrm{W},p}^{Z_A}\left({\left(\phi^X_{(A,\varphi)}\right)}_\#\muX,{\left(\psi^Y_{(A,\varphi)}\right)}_\#\muY\right).\end{equation}
\end{corollary}

By proving \Cref{thm:usturm A Phi representation}, we have verified \Cref{thm:compact usturm A Phi representation} for finite ultrametric measure spaces. In the following, we will use \Cref{thm:usturm A Phi representation} and weighted quotients to demonstrate \Cref{thm:compact usturm A Phi representation}. However, before we come to this, we need to establish the following two auxiliary results.
\begin{lemma}\label{lm:quotient ultrametric dW}
Let $X\in\mathcal{U}$ be a compact ultrametric space. Let $t>0$ and let $p\in[1,\infty)$. Then, for any $\alpha,\beta\in\mathcal{P}(X)$, we have that
\[\lc d_{\mathrm{W},p}^{X_t}(\alpha_t,\beta_t)\rc^p\geq\lc d_{\mathrm{W},p}^{X}(\alpha,\beta)\rc^p-t^p,\]
where $\alpha_t$ is the push forward of $\alpha$ under the canonical quotient map $Q_t:X\rightarrow X_t$ taking $x\in X$ to $[x]_t\in X_t$.
\end{lemma}
\begin{proof}
For any $\mu_t\in\mathcal{C}(\alpha_t,\beta_t)$, it is easy to see that there exists $\mu\in\mathcal{C}(\alpha,\beta)$ such that $\mu_t=\lc Q_t\times Q_t\rc_\#\mu$ where $Q_t\times Q_t:X\times X\rightarrow X_t\times X_t$ maps $(x,x')\in X\times X$ to $([x]_t,[x']_t)$. For example, suppose $X_t=\{[x_1]_t,\ldots,[x_n]_t\}$, then one can let 
\[\mu\coloneqq\sum_{i,j=1}^n\mu_t(([x_i]_t,[x_j]_t))\frac{\alpha|_{[x_i]_t}}{\alpha([x_i]_t)}\otimes\frac{\beta|_{[x_j]_t}}{\beta([x_j]_t)},\]
where $\alpha|_{[x_i]_t}$ is the restriction of $\alpha$ on $[x_i]_t$.

For any $x,x'\in X$, we have that $\lc u_X(x,x')\rc^p\leq \lc u_{X_t}([x]_t,[x']_t)\rc^p+t^p$. Then,
\begin{align*}
    \lc d_{\mathrm{W},p}^{X}(\alpha,\beta)\rc^p&\leq\int_{X\times X}\lc u_X(x,x')\rc^p\mu(dx\times dx')\\
    &\leq  \int_{X\times X}\lc \lc u_{X_t}([x]_t,[x']_t)\rc^p+t^p\rc\mu(dx\times dx')\\
    &=\int_{X\times X}\lc u_X(Q_t(x),Q_t(x'))\rc^p\mu(dx\times dx')+t^p\\
    &=\int_{X_t\times X_t}\lc u_{X_t}([x]_t,[x']_t)\rc^p\mu_t(d[x]_t\times d[x']_t)+t^p
\end{align*}
Infimizing over all $\mu_t\in\mathcal{C}(\alpha_t,\beta_t)$, we obtain that 
\[\lc d_{\mathrm{W},p}^{X_t}(\alpha_t,\beta_t)\rc^p\geq\lc d_{\mathrm{W},p}^{X}(\alpha,\beta)\rc^p-t^p.\]
\end{proof}

\begin{lemma}\label{lm:limit of quotient usturm}
Let $\X\in\mathcal{U}^w$ and let $p\in[1,\infty]$. Then, for any $t>0$, we have that
\[\usturm p(\X_t,\X)\leq t.\]
In particular, $\lim_{t\rightarrow 0}\usturm p(\X_t,\X)=0$.
\end{lemma}

\begin{proof}
It is obvious that $(\X_t)_t\cong_w\X_t$. Hence, it holds by \Cref{thm:ugw-infty-eq} that $\usturm\infty(\X_t,\X)\leq t$. By \Cref{prop:usturm_basic} we have that for any $p\in[1,\infty]$
\[\usturm p(\X_t,\X)\leq \usturm\infty(\X_t,\X)\leq t.\]
\end{proof}
With \Cref{lm:quotient ultrametric dW} and \Cref{lm:limit of quotient usturm} available, we can come to the proof of \Cref{thm:compact usturm A Phi representation}.

\begin{proof}[Proof of \Cref{thm:compact usturm A Phi representation}]
Clearly, it follows from the definition of $\usturm{p}$ (see \Cref{eq:ultra Sturm}) that
 \[\usturm{p}(\X,\Y)\leq\inf_{(A,\varphi)\in\mathcal{A}}d_{\mathrm{W},p}^{Z_A}\left({\left(\phi^X_{(A,\varphi)}\right)}_\#\muX,{\left(\psi^Y_{(A,\varphi)}\right)}_\#\muY\right)\] Hence, we focus on proving the opposite inequality.

Given any $t>0$, by \Cref{lm:quotient-finite}, both $\X_t$ and $\Y_t$ are finite spaces. By \Cref{thm:usturm A Phi representation} we have that
\[\usturm{p}(\X_t,\Y_t)=\inf_{(A_t,\varphi_t)\in \mathcal{A}_t}d_{\mathrm{W},p}^{Z_{A_t}}\left({\left(\phi^{X_t}_{(A_t,\varphi_t)}\right)}_\#(\muX)_t,{\left(\psi^{Y_t}_{(A_t,\varphi_t)}\right)}_\#(\muY)_t\right),\]
where
\[\mathcal{A}_t\coloneqq\{(A_t,\varphi_t)\,|\,\emptyset\neq A_t\subseteq X_t \text{ is closed and } \varphi_t:A_t\hookrightarrow Y_t \text{ is an isometric embedding } \}.\]

For any $(A_t,\varphi_t)\in \mathcal{A}_t$, assume that $A_t=\{[x_1]_t^X,\ldots,[x_n]_t^X\}$ and that $\varphi_t([x_i]_t)=[y_i]_t\in Y_t$ for all $i=1,\ldots,n$. Let $A\coloneqq\{x_1,\ldots,x_n\}$. Then, the map $\varphi:A\rightarrow Y$ defined by $x_i\mapsto y_i$ for $i=1,\ldots,n$ is an isometric embedding. Therefore, $(A,\varphi)\in \mathcal{A}$.

\begin{claim}
 $\lc(Z_A)_t,u_{(Z_A)_t}\rc\cong \lc Z_{A_t},u_{Z_{A_t}}\rc$.
\end{claim}
\begin{proof}[Proof of the Claim]
We define a map $\Psi:(Z_A)_t\rightarrow Z_{A_t}$ by $[x]_t^{Z_A}\mapsto [x]_t^X$ for $x\in X$ and $[y]_t^{Z_A}\mapsto [y]_t^Y$ for $y\in Y\backslash\varphi(A)$. We first show that $\Psi$ is well-defined. For any $x'\in X$, if $u_{Z_A}(x,x')\leq t$, then obviously we have that $u_X(x,x')=u_{Z_A}(x,x')\leq t$ and thus $[x]_t^X=[x']_t^X$. Now, assume that there exists $y\in Y\backslash\varphi(A)$ such that $u_{Z_A}(x,y)\leq t$, i.e., $[x]_t^{Z_A}=[y]_t^{Z_A}$. Then, by finiteness of $A$ and definition of $Z_A$, there exists $x_i\in A$ such that $u_{Z_A}(x,y)=\max\lc u_X(x,x_i),u_Y(\varphi(x_i),y)\rc\leq t$. This gives that
\[u_{Z_{A_t}}([x]_t^X,[y]_t^Y)\leq \max\lc u_{X_t}\lc[x]_t^X,[x_i]_t^X\rc,u_{Y_t}\lc[\varphi(x_i)]_t^Y,[y]_t^Y\rc\rc\leq t.\]
However, this happens only if $u_{Z_{A_t}}([x]_t^X,[y]_t^Y)=0$, that is, $[x]_t^X$ is identified with $[y]_t^Y$ under the map $\varphi_t$. Therefore, $\Psi$ is well-defined. 

It is easy to see from the definition that $\Psi$ is surjective. Thus, it suffices to show that $\Psi$ is an isometric embedding to finish the proof. For any $x,x'\in X$ such that $u_X(x,x')>t$, we have that
\[u_{(Z_A)_t}\lc[x]_t^{Z_A},[x']_t^{Z_A}\rc = u_{Z_A}(x,x')=u_X(x,x')=u_{X_t}\lc[x]_t^{X},[x']_t^{X}\rc=u_{Z_{A_t}}\lc[x]_t^{X},[x']_t^{X}\rc.\]
Similarly, for any $y,y'\in Y\backslash\varphi(A)$ such that $u_Y(y,y')>t$, we have that
\[u_{(Z_A)_t}\lc[y]_t^{Z_A},[y']_t^{Z_A}\rc =u_{Z_{A_t}}\lc[y]_t^{Y},[y']_t^{Y}\rc.\]
Now, consider $x\in X$ and $y\in Y\backslash\varphi(A)$. Assume that $u_{Z_A}(x,y)>t$ (otherwise $[x]_t^{Z_A}=[y]_t^{Z_A}$). Then, we have that
\begin{align*}
     u_{Z_A}\lc x,y\rc=\min_{i=1,\ldots,n}\max\lc u_{X}\lc x,x_i\rc,u_{Y}\lc \varphi(x_i),y\rc\rc>t.
\end{align*}
This implies that
\begin{align*}
     u_{Z_{A_t}}\lc [x]_t^X,[y]_t^Y\rc&=\min_{i=1,\ldots,n}\max\lc u_{X_t}\lc [x]_t^X,[x_i]_t^X\rc,u_{Y_t}\lc \varphi_t([x_i]_t^X),[y]_t^Y\rc\rc\\
     &=\min_{i=1,\ldots,n}\max\lc u_{X}\lc x,x_i\rc,u_{Y}\lc \varphi(x_i),y\rc\rc\\
     &=u_{Z_A}\lc x,y\rc=u_{(Z_A)_t}\lc [x]_t^{Z_A},[y]_t^{Z_A}\rc.
\end{align*}
Therefore, $\Psi$ is an isometric embedding and thus we conclude the proof.
\end{proof}

By \Cref{lm:quotient ultrametric dW} we have that
\begin{align*}
    &\lc d_{\mathrm{W},p}^{Z_{A_t}}\left({\left(\phi^{X_t}_{(A_t,\varphi_t)}\right)}_\#(\muX)_t,{\left(\psi^{Y_t}_{(A_t,\varphi_t)}\right)}_\#(\muY)_t\right)\rc^p\\
    \geq &\lc d_{\mathrm{W},p}^{Z_{A}}\left({\left(\phi^{X}_{(A,\varphi)}\right)}_\#\muX,{\left(\psi^{Y}_{(A,\varphi)}\right)}_\#\muY\right)\rc^p-t^p
\end{align*}
Therefore,
\begin{align*}
    \usturm{p}(\X_t,\Y_t)&=\inf_{(A_t,\varphi_t)\in \mathcal{A}_t}d_{\mathrm{W},p}^{Z_{A_t}}\left({\left(\phi^{X_t}_{(A_t,\varphi_t)}\right)}_\#(\muX)_t,{\left(\psi^{Y_t}_{(A_t,\varphi_t)}\right)}_\#(\muY)_t\right)\\
    &\geq \inf_{(A,\varphi)\in \mathcal{A}}\lc\lc d_{\mathrm{W},p}^{Z_{A}}\left({\left(\phi^{X}_{(A,\varphi)}\right)}_\#\muX,{\left(\psi^{Y}_{(A,\varphi)}\right)}_\#\muY\right)\rc^p-t^p\rc^\frac{1}{p}.
\end{align*}
Notice that the last inequality already holds when we only consider $(A,\varphi)$ corresponding to $(A_t,\varphi_t)\in\mathcal{A}_t$.

By \Cref{lm:limit of quotient usturm}, we have that
\[\usturm p(\X,\Y)=\lim_{t\rightarrow 0}\usturm p(\X_t,\Y_t)\geq \inf_{(A,\varphi)\in \mathcal{A}} d_{\mathrm{W},p}^{Z_{A}}\left({\left(\phi^{X}_{(A,\varphi)}\right)}_\#\muX,{\left(\psi^{Y}_{(A,\varphi)}\right)}_\#\muY\right),\]
which concludes the proof.
\end{proof}
\subsection{Proofs from Section \ref{subsec:the ultrametric GW distance}}\label{sec:proofs sec32}
In the following, we give the complete proofs of the results stated in  \Cref{subsec:the ultrametric GW distance}.
\subsubsection{Proof of \Cref{prop:ugw-properties}}\label{sec:proof:prop:ugw-properties}
		\begin{enumerate}
		\item This follows directly from the definitions of $\ugw{p}$ and $\dgw{p}$ (see \Cref{eq:def uGW} and \Cref{eq:Gromov Wasserstein}).  
	\item By Jensen's inequality we have that $\disu_p(\mu)\leq\disu_q(\mu)$ for any $\mu\in\mathcal{C}(\muX,\muY)$. Therefore, $\ugw{p}(\X,\Y)\leq\ugw{q}(\X,\Y) $. 
	\item By (2), we know that $\{\ugw{n}(\X,\Y)\}_{n\in\mathbb{N}}$ is an increasing sequence with a finite upper bound $\ugw{\infty}(\X,\Y)$. Therefore, $L\coloneqq\lim_{n\rightarrow\infty}\ugw{n}(\X,\Y)$ exists and  it holds $L\leq \ugw{\infty}(\X,\Y)$. 
	
	To prove the opposite inequality, by \Cref{prop:ugw-ext-opt}, there exists for each $n\in\mathbb{N}$ $\mu_n\in\mathcal{C}(\muX,\muY)$ such that 
		\[\left(\iint_{X\times Y\times X\times Y}\Lambda_\infty(\uX(x,x'),\uY(y,y'))^n\mu_n(dx\times dy)\mu_n(dx'\times dy')\right)^\frac{1}{n}= \ugw{n}(\X,\Y). \]
By \Cref{lm:measure coupling compact}, $\{\mu_n\}_{n\in\N}$  weakly converges (after taking an appropriate subsequence) to some $\mu\in\mathcal{C}(\muX,\muY)$. Let \[M=\sup_{(x,y),(x',y')\in \supp{\mu}}\Lambda_\infty(\uX(x,x'),\uY(y,y'))\] and for any given $\eps>0$ let
	\[U=\{((x,y),(x',y'))\in X\times Y\times X\times Y\,|\,\Lambda_\infty(\uX(x,x'),\uY(y,y'))> M-\eps \}.\]
	Then, we have $\mu\otimes\mu(U)>0$. As $\mu_n$ weakly converges to $\mu$, we have that $\mu_n\otimes\mu_n$ weakly converges to $\mu\otimes\mu$. Since $U$ is open, there exists a small $\eps_1>0$ such that $\mu_n\otimes\mu_n(U)>\mu\otimes \mu(U)-\eps_1>0$ for $n$ large enough (see e.g. \citet[Thm. 2.1]{billingsleyConvergenceProbabilityMeasures2013}). Therefore, 
	\begin{align*}
	&\left(\iint_{X\times Y\times X\times Y}\Lambda_\infty(\uX(x,x'),\uY(y,y'))^n\mu_n(dx\times dy)\mu_n(dx'\times dy')\right)^\frac{1}{n}\\
	\geq& (\mu_n\otimes\mu_n(U))^\frac{1}{n}(M-\eps)\geq(\mu\otimes\mu(U)-\eps_1)^\frac{1}{n}(M-\eps).
	\end{align*}
	Letting $n\rightarrow\infty$, we obtain $L\geq M-\eps$. Since $\eps>0$ is arbitrary, we obtain $L\geq M\geq \ugw{\infty}(\X,\Y)$. 
\end{enumerate}

\subsubsection{Proof of Theorem \ref{thm:ugw-p-metric}}\label{sec:proof of thm ugw-p-metric}
One main step to verify \Cref{thm:ugw-p-metric} is to demonstrate the existence of optimal couplings. 
\begin{proposition}\label{prop:ugw-ext-opt}
Let $\X=\ummspaceX$ and $\Y=\ummspaceY$ be compact ultrametric measure spaces. Then, for any $p\in[1,\infty]$, there always exists an optimal coupling $\mu\in\mathcal{C}(\muX,\muY)$ such that $\ugw{p}(\X,\Y)=\mathrm{dis}_p^\mathrm{ult}(\mu)$.
\end{proposition}
\begin{proof}
    	We will only prove the claim for the case $p<\infty$ since the case $p=\infty$ can be proven in a similar manner. Let $\mu_n\in\mathcal{C}(\muX,\muY)$ be such that 
	\[\left(\iint_{X\times Y\times X\times Y}\Lambda_\infty(\uX(x,x'),\uY(y,y'))^p\,\mu_n(dx\times dy)\mu_n(dx'\times dy')\right)^\frac{1}{p}\leq \ugw{p}(\X,\Y)+\frac{1}{n}. \]
	By \Cref{lm:measure coupling compact}, $\{\mu_n\}_{n\in\N}$ weakly converges to some $\mu\in\mathcal{C}(\muX,\muY)$ (after taking an appropriate subsequence). Then, by the boundedness and continuity  of $\Lambda_\infty(\uX,\uY)$ on $X\times Y\times X\times Y$ (cf. Lemma \ref{lm:continuity-delta-infty}) as well as the weak convergence of $\mu_n\otimes \mu_n$, we have that that
		\[\mathrm{dis}_p^\mathrm{ult}(\mu)=\lim_{n\rightarrow\infty}\mathrm{dis}_p^\mathrm{ult}(\mu_n)\leq \ugw{p}(\X,\Y). \]
		Hence, $\ugw{p}(\X,\Y)=\mathrm{dis}_p^\mathrm{ult}(\mu)$.
\end{proof}

Based on \Cref{prop:ugw-ext-opt}, it is straightforward to prove \Cref{thm:ugw-p-metric}.
\begin{proof}[Proof of \Cref{thm:ugw-p-metric}]
		It is clear that $\ugw{p}$ is symmetric and that $\ugw{p}(\X,\Y)=0$ if $\X\cong_w\Y$. Furthermore, we remark that $\ugw{p}(\X,\Y)\geq d_{\mathrm{GW},p}(\X,\Y)$ by \Cref{prop:ugw-properties}. Since $\dgw{p}(\X,\Y)=0$ implies that $\X\cong_w\Y$ (see \cite{memoli2011gromov}), we have that $\ugw{p}(\X,\Y)=0$ implies that $\X\cong_w\Y$. It remains to verify the $p$-triangle inequality. To this end, we only prove the case when $p<\infty$ whereas the case $p=\infty$ follows by analogous arguments.
		
		Now let $\X,\Y,\mathcal{Z}$ be three ultrametric measure spaces. Let $\mu_{XY}\in\mathcal{C}(\muX,\muY)$ and $\mu_{YZ}\in\mathcal{C}(\muY,\mu_Z)$ be optimal (cf. \Cref{prop:ugw-ext-opt}). By the Gluing Lemma \citep[Lemma 7.6]{villani2003topics},
		there exists a measure $\mu_{XYZ}\in\mathcal{P}(X\times Y\times Z)$ with marginals $\mu_{XY}$ on $X\times Y$ and $\mu_{YZ}$ on $Y\times Z$. Further, we define $\mu_{XZ}=(\pi_{XZ})_\#\mu\in\mathcal{P}(X\times Z)$, where $\pi_{XZ}$ denotes the canonical projection $X\times Y\times Z\to X\times Z$. Then,
		\begin{align*}
		&(\ugw{p}(\X,\Z))^p\leq \iint_{X\times Z \times X\times Z}\big(\Lambda_\infty(\uX(x,x'),u_Z(z,z'))\big)^p\,\mu_{XZ}(dx\times dz)\,\mu_{XZ}(dx'\times dz')  \\
		=&\iint_{X\times Y\times Z \times X\times Y\times Z}\big(\Lambda_\infty(\uX(x,x'),\uZ(z,z'))\big)^p\,\mu_{XYZ}(dx\times dy\times dz)\,\mu_{XYZ}(dx'\times dy'\times dz')\\
		\leq& \iint_{X\times Y\times Z \times X\times Y\times Z}\big(\Lambda_\infty(\uX(x,x'),\uY(y,y'))\big)^p\,\mu_{XYZ}(dx\times dy\times dz)\,\mu_{XYZ}(dx'\times dy'\times dz')\\
		+& \iint_{X\times Y\times Z \times X\times Y\times Z}\big(\Lambda_\infty(\uY(y,y'),\uZ(z,z'))\big)^p\,\mu_{XYZ}(dx\times dy\times dz)\,\mu_{XYZ}(dx'\times dy'\times dz')\\
		=& \iint_{X\times Y \times X\times Y}\big(\Lambda_\infty(\uX(x,x'),\uY(y,y'))\big)^p\,\mu_{XY}(dx\times dy)\,\mu_{XY}(dx'\times dy')\\
		+& \iint_{ Y\times Z \times Y\times Z}\big(\Lambda_\infty(\uY(y,y'),\uZ(z,z'))\big)^p\,\mu_{YZ}(dy\times dz)\,\mu_{YZ}( dy'\times dz')\\
		=&(\ugw{p}(\X,\Y))^p+(\ugw{p}(\Y,\Z))^p,
		\end{align*}
		where the second inequality follows from the fact that $\Lambda_\infty$ in an ultrametric on $\Rp$ (cf. \cite[Remark 1.14]{memoli2019gromov}) and the observation that an ultrametric is automatically a $p$-metric for any $p\in[1,\infty]$ \citep[Proposition 1.16]{memoli2019gromov}. 
	\end{proof}

	\subsubsection{Proof of \Cref{thm:ugw-infty-eq}}\label{sec:proof of thm ugw-infty-eq}
	We first prove that 
	\begin{equation}\label{eq:ugw-infty=inf}
	    u_{\mathrm{GW},\infty}(\X,\Y)=\inf\left\lbrace t\geq 0 \,|\,\X_t \cong_w \Y_t\right\rbrace
	\end{equation}
    and then show that the infimum is attainable.
	
	Since $\X_0\cong_w \X$ and $\Y_0\cong_w \Y$, if $\X_0\cong_w\Y_0$, then $\X\cong_w\Y$ and thus by \Cref{thm:ugw-p-metric}
	\[\ugw \infty(\X,\Y)=0=\inf\left\lbrace t\geq 0 \,|\,\X_t \cong_w \Y_t\right\rbrace\]
	Now, assume that for some $t>0$, $\X_t\cong_w \Y_t$. By Lemma \ref{lm:quotient-finite}, for some $n\in\mathbb N$ we can write ${X}_t=\{[x_1]_t,\dots,[x_n]_t\}$ and ${Y}_t=\{[y_1]_t,\dots,[y_n]_t\}$ such that $u_{X_t}([x_i]_t,[x_j]_t)=u_{Y_t}([y_i]_t,[y_j]_t)$ and $\muX([x_i]_t)=\muY([y_i]_t)$. Let $\muX^i\coloneqq\muX|_{[x_i]_t}$ and $\muY^i\coloneqq\muY|_{[y_i]_t}$ for all $i=1,\dots,n$. Let $\mu\coloneqq\sum_{i=1}^n\muX^i\otimes\muY^i.$ It is easy to check that $\mu\in\mathcal{C}(\muX,\muY)$ and $\mathrm{supp}(\mu)=\bigcup_{i=1}^n[x_i]_t\times[y_i]_t.$ 
	Assume $(x,y)\in[x_i]_t\times[y_i]_t$ and $(x',y')\in[x_j]_t\times [y_j]_t$. If $i\neq j$, then $u_{X_t}([x_i]_t,[x_j]_t)=u_{Y_t}([y_i]_t,[y_j]_t)$ and thus
	\[\Lambda_\infty(\uX(x,x'),\uY(y,y'))=\Lambda_\infty(u_{X_t}([x_i]_t,[x_j]_t),u_{Y_t}([y_i]_t,[y_j]_t))=0.\]
	If $i=j$, then $\uX(x,x'),\uY(y,y')\leq t$ and thus $\Lambda_\infty(\uX(x,x'),\uY(y,y'))\leq t$. In either case, we have that
		\begin{align*}
		u_{\mathrm{GW},\infty}(\X,\Y)\leq \sup_{(x,y),(x',y')\in\mathrm{supp}(\mu)}\Lambda_\infty(\uX(x,x'),\uY(y,y'))\leq t.
		\end{align*}
	Therefore, $\ugw{\infty}(\X,\Y)\leq\inf\left\lbrace t\geq 0 \,|\,\X_t \cong_w \Y_t\right\rbrace.$ 
		
	Conversely, suppose $\mu\in\mathcal{C}(\muX,\muY)$ and let $t\coloneqq\sup_{(x,y),(x',y')\in\mathrm{supp}(\mu)}\Lambda_\infty(\uX(x,x'),\uY(y,y'))$. By \citet[Lemma 2.2]{memoli2011gromov}, we know that $\mathrm{supp}(\mu)$ is a correspondence between $X$ and $Y$. We define a map $f_t:X_t\rightarrow Y_t$ by taking $[x]_t^X\in X_t$ to $[y]_t^Y\in Y_t$ such that $(x,y)\in\mathrm{supp}(\mu)$. It is easy to check that $f_t$ is well-defined and moreover $f_t$ is an isometry (see for example the proof of \citet[Theorem 5.7]{memoli2019gromov}). Next, we prove that $f_t$ is actually an isomorphism between $\X_t$ and $\Y_t$. For any $[x]^X_t\in X_t$, let $y\in Y$ be such that $(x,y)\in \supp{\mu}$ (in this case, $[y]^Y_t=f_t([x]^X_t)$). If there exists $(x',y')\in\mathrm{supp}(\mu)$ such that $x'\in[x]^X_t$ and $y'\not\in[y]^Y_t$, then $\Lambda_\infty(\uX(x,x'),\uY(y,y'))=\uY(y,y')>t$, which is impossible. Consequently, $\mu([x]^X_t\times(Y\setminus [y]^Y_t))=0$ and similarly, $\mu((X\setminus [x]^X_t)\times[y]^Y_t)=0$. This yields that
	\[\muX([x]^X_t)=\mu([x]^Y_t\times Y)=\mu([x]^X_t\times[y]^Y_t)=\mu(X\times[y]^Y_t)=\muY([y]^Y_t).\]
	Therefore, $f_t$ is an isomorphism between $\X_t$ and $\Y_t$. Hence, we have that $\ugw{\infty}(\X,\Y)\geq\inf\left\lbrace t\geq 0 \,|\,\X_t \cong_w \Y_t\right\rbrace$ and hence $\ugw{\infty}(\X,\Y)=\inf\left\lbrace t\geq 0 \,|\,\X_t \cong_w \Y_t\right\rbrace.$ 
		
	Finally, we show that the infimum of $\inf\left\lbrace t\geq 0 \,|\,\X_t \cong_w \Y_t\right\rbrace$ is attainable. Let $\delta \coloneqq\inf\lbrace t\geq 0 \,|\,\X_t \cong_w \Y_t\rbrace$. If $\delta>0$, let $\{t_n\}_{n\in\mathbb{N}}$ be a decreasing sequence converging to $\delta$ such that $\X_{t_n}\cong_w \Y_{t_n}$ for all $t_n$. Since $\X_\delta$ and $\Y_\delta$ are finite spaces, we actually have that $\X_{t_n}=\X_{\delta}$ and $\Y_{t_n}=\Y_{\delta}$ when $n$ is large enough. This immediately implies that $\X_\delta\cong_w \Y_\delta$. Now, if $\delta=0$, then by \Cref{eq:ugw-infty=inf} we have that $\ugw{\infty}(\X,\Y)=\delta=0$. By \Cref{thm:ugw-p-metric}, $\X\cong_w\Y$. This is equivalent to $\X_\delta\cong_w\Y_\delta$. Therefore, the infimum of $\inf\left\lbrace t\geq 0 \,|\,\X_t \cong_w \Y_t\right\rbrace$ is always attainable.

\subsubsection{Proof of \Cref{thm:snow-ugw}}\label{sec:proof of thm snow-ugw}
	An important observation for the proof of \Cref{thm:snow-ugw} is that the snowflake transform relates the $p$-Wasserstein pseudometric on a pseudo-ultrametric space $X$ with the 1-Wasserstein pseudometric on the space $S_p(X)$, $1\leq p<\infty$.
	
    \begin{lemma}\label{lm:snow-w-dis}
		Given a pseudo-ultrametric space $(X,\uX)$ and $p\geq 1$, we have for any $\alpha,\beta\in \mathcal{P}(X)$ that
		\[\pseudoWasser{p}^{(X,\uX)}(\alpha,\beta)=\left(d_{\mathrm{W},1}^{S_p(X)}(\alpha,\beta)\right)^\frac{1}{p}. \]
	\end{lemma}
		
	\begin{remark} Since $S_p\circ\uX$ and $\uX$ induce the same topology and thus the same Borel sets on $X$, we have that $\mathcal{P}(X)=\mathcal{P}(S_p(X))$ and thus the expression $d_{\mathrm{W},1}^{S_p(X)}(\alpha,\beta)$ in the lemma is well defined.
	\end{remark}
	\begin{proof}[Proof of \Cref{lm:snow-w-dis}]
		Suppose $\mu_1,\mu_2\in\mathcal{C}(\alpha,\beta)$ are optimal for $\pseudoWasser{p}^X(\alpha,\beta)$ and $d_{\mathrm{W},1}^{S_p(X)}(\alpha,\beta)$, respectively (see \Cref{sec:Wasserstein pseudometric} for the existence of $\mu_1$ and $\mu_2$). Then,
		\begin{align*}
		\left(\pseudoWasser{p}^{(X,\uX)}(\alpha,\beta)\right)^p=\int_{X\times X}(\uX(x,y))^p\mu_1(dx\times dy)\\=\int_{X\times X}S_p(\uX)(x,y)\mu_1(dx\times dy)\geq d_{\mathrm{W},1}^{S_p(X)}(\alpha,\beta),
		\end{align*}
		and 
		\begin{align*}
		d_{\mathrm{W},1}^{S_p(X)}(\alpha,\beta)=\int_{X\times X}S_p(\uX)(x,y)\mu_2(dx\times dy)\\=\int_{X\times X}(\uX(x,y))^p\mu_2(dx\times dy)\geq \left(\pseudoWasser{p}^{(X,\uX)}(\alpha,\beta)\right)^p.
		\end{align*}
		Therefore, $\pseudoWasser{p}^{(X,\uX)}(\alpha,\beta)=\left(d_{\mathrm{W},1}^{S_p(X)}(\alpha,\beta)\right)^\frac{1}{p}. $
	\end{proof}
	With \Cref{lm:snow-w-dis} at our disposal we can prove \Cref{thm:snow-ugw}.
	\begin{proof}[Proof of \Cref{thm:snow-ugw}]
		Let $\mu\in\mathcal{C}(\muX,\muY)$. Then, 
		\begin{align*}
		&\iint_{X\times Y \times X\times Y}\big(\Lambda_\infty(\uX(x,x'),\uY(y,y'))\big)^p\,\mu(dx\times dy)\,\mu(dx'\times dy')\\
		=&\iint_{X\times Y \times X\times Y}\Lambda_\infty\big(\uX(x,x')^p,\uY(y,y')^p\big)\,\mu(dx\times dy)\,\mu(dx'\times dy').
		\end{align*}
		By infimizing over $\mu\in\mathcal{C}(\muX,\muY)$ on both sides, we obtain that
		\[(\ugw{p}(\X,\Y))^p=\ugw{1}(S_p(\X),S_p(\Y)).\]

		To prove the second part of the claim, let $u\in\mathcal{D}^\mathrm{ult}(\uX,\uY)$. By \Cref{lm:snow-w-dis} we have that

		\[\big(\pseudoWasser{p}^{(X\sqcup Y,u)}(\muX,\muY)\big)^p= d_{\mathrm{W},1}^{(S_p(X)\sqcup S_p(Y),S_p(u))}(\muX,\muY).\]
		Finally, infimizing over $u\in\mathcal{D}^\mathrm{ult}(\uX,\uY)$ yields
		\[\ugw{p}^\mathrm{sturm}(\X,\Y)^p=\ugw{1}^\mathrm{sturm}(S_p(\X),S_p(\Y)). \]
	\end{proof}
	
	As a direct consequence of \Cref{thm:snow-ugw}, we obtain the following relation between $(\mathcal{U}^w,\usturm 1)$ and $\left(\mathcal{U}^w,\usturm p\right)$ for $p\in[1,\infty)$.	
		\begin{corollary}\label{coro:iso-snow}
	For each $p\in[1,\infty)$, the metric space $(\mathcal{U}^w,\usturm 1)$ is isometric to the snowflake transform of $\left(\mathcal{U}^w,\usturm p\right)$, i.e., 
	\[S_p\left(\mathcal{U}^w,\usturm p\right)\cong \left(\mathcal{U}^w,\usturm 1\right) \]
	\end{corollary}
	\begin{proof}
	    Consider the snowflake transform map $S_p:\mathcal{U}^w\rightarrow\mathcal{U}^w$ sending $X\in \mathcal{U}^w$ to $S_p(X)\in\mathcal{U}^w$. It is obvious that $S_p$ is bijective. By \Cref{thm:snow-ugw}, $S_p$ is an isometry from $S_p\left(\mathcal{U}^w,\usturm p\right)$ to $ \left(\mathcal{U}^w,\usturm 1\right)$. Therefore, $S_p\left(\mathcal{U}^w,\usturm p\right)\cong \left(\mathcal{U}^w,\usturm 1\right) $.
	\end{proof}
	\subsection{Proofs from Section \ref{subsec:relation between ugw and usturm}}\label{sec:proofs sec33}
	Throughout the following, we demonstrate the open claims from \Cref{subsec:relation between ugw and usturm}.
	\subsubsection{Proof of \Cref{thm:ugw<ugw-sturm}}\label{sec:proof of thm ugw<ugw-sturm}
		First, we focus on the statement for $p=1$, i.e., on showing
		\begin{equation}\label{eq:ugw1<usturm1}
		    \ugw{1}(\X,\Y)\leq 2\,\ugw{1}^\mathrm{sturm}(\X,\Y).
		\end{equation}
		Let $u\in\mathcal{D}^\mathrm{ult}(\uX,\uY)$ and $\mu\in\mathcal{C}(\muX,\muY)$ be such that $\usturm 1(\X,\Y)=\int u(x,y)\mu(dx\times dy)$. The existence of $u$ and $\mu$ follows from \Cref{prop:usturm_optimal}
		
		\begin{claim}\label{clm:dlt infty ineq}
			For any $(x,y),(x',y')\in X\times Y$, we have
			\[\Lambda_\infty(\uX(x,x'),\uY(y,y'))\leq\max(u(x,y),u(x',y'))\leq u(x,y)+u(x',y'). \]
		\end{claim}
		\begin{proof} 
			We only need to show that \[\Lambda_\infty(\uX(x,x'),\uY(y,y'))\leq\max(u(x,y),u(x',y')).\] If $\uX(x,x')=\uY(y,y')$, then there is nothing to prove. Otherwise, we assume without loss of generality that $\uX(x,x')<\uY(y,y')$. If $\max(u(x,y),u(x',y'))<\uY(y,y')$, then by the strong triangle inequality we must have $u(x,y')=\uY(y,y')=u(x',y)$. However, $u(x',y)\leq \max(\uX(x,x'),u(x,y))<\uY(y,y')$, which leads to a contradiction. Therefore, $\Lambda_\infty(\uX(x,x'),\uY(y,y'))\leq\max(u(x,y),u(x',y'))$.
		\end{proof}	
		By Claim \ref{clm:dlt infty ineq}, we have 
		\begin{align*}
		&\iint_{X\times Y \times X\times Y}\Lambda_\infty(\uX(x,x'),\uY(y,y'))\,\mu(dx\times dy)\,\mu(dx'\times dy')\\
		\leq &\iint_{X\times Y \times X\times Y}u(x,y)\,\mu(dx\times dy)\,\mu(dx'\times dy')\\
		+&\iint_{X\times Y \times X\times Y}u(x',y')\,\mu(dx\times dy)\,\mu(dx'\times dy')\\
		=&\int_{X\times Y}u(x,y)\,\mu(dx\times dy)+\int_{X\times Y}u(x',y')\,\mu(dx'\times dy')\leq 2\usturm 1(\X,\Y).
		\end{align*}
		Therefore, $\ugw{1}(\X,\Y)\leq 2\,\ugw{1}^\mathrm{sturm}(\X,\Y).$
		
		Applying \Cref{thm:snow-ugw} and \Cref{eq:ugw1<usturm1}, yields that for any $p\in[1,\infty)$
\[ \ugw{p}(\X,\Y)=\left(\ugw{1}(S_p(\X),S_p(\Y))\right)^\frac{1}{p} \leq\left(2\,\usturm{1}(S_p(\X),S_p(\Y))\right)^\frac{1}{p}=2^\frac{1}{p}\,\usturm{p}(\X,\Y).\]

\subsubsection{Proof of results in \Cref{ex:non bi lipschitz}}\label{sec: ex non bi lipschitz}
It follows from \cite[Remark 5.17]{memoli2011gromov} that 
\[\dsturm{p}\lc\hat{\Delta}_n(1),\hat{\Delta}_{2n}(1)\rc\geq \frac{1}{4}\,\,\text{and}\,\,\dgw{p}\lc\hat{\Delta}_n(1),\hat{\Delta}_{2n}(1)\rc\leq \frac{1}{2}\lc\frac{3}{2n}\rc^\frac{1}{p}.\]

Then, by \Cref{prop:usturm_basic}, we have that 
\[\usturm{p}\lc\hat{\Delta}_n(1),\hat{\Delta}_{2n}(1)\rc\geq\dsturm{p}\lc\hat{\Delta}_n(1),\hat{\Delta}_{2n}(1)\rc\geq \frac{1}{4}.\]

Let $\mu_n$ denote the uniform probability measure of $\hat{\Delta}_n(1)$. Since $\hat{\Delta}_n(1)$ has the constant interpoint distance 1, it is obvious that for any coupling $\mu\in\mathcal{C}(\mu_n,\mu_{2n})$,
\[\dis_p(\mu)  = \disu_p(\mu)\]
This implies that 
\[\ugw{p}\lc\hat{\Delta}_n(1),\hat{\Delta}_{2n}(1)\rc=2\,\dgw{p}\lc\hat{\Delta}_n(1),\hat{\Delta}_{2n}(1)\rc\leq \lc\frac{3}{2n}\rc^\frac{1}{p}.\]
	
	\subsubsection{Proof of \Cref{thm:ugw infty and sturm ugw infty}}\label{sec:proof of thm ugw infty and sturm ugw infty}
		First, we prove that $\usturm{\infty}(\X,\Y)\geq \ugw{\infty}(\X,\Y).$ Indeed, for any $u\in\mathcal{D}^\mathrm{ult}(\uX,\uY)$ and $\mu\in\mathcal{C}(\muX,\muY)$, we have that 
		\begin{align*}
		\sup_{(x,y)\in \supp{\mu}}u(x,y)&=\sup_{(x,y),(x',y')\in \supp{\mu}}\max(u(x,y),u(x',y'))\\
		&\geq\sup_{(x,y),(x',y')\in \supp{\mu}}\Lambda_\infty(\uX(x,x'),\uY(y,y')) \\
		&\geq \ugw{\infty}(\X,\Y),
		\end{align*}
		where the first inequality follows from Claim \ref{clm:dlt infty ineq} in the proof of \Cref{thm:ugw<ugw-sturm}. Then, by a standard limit argument, we conclude that $\usturm{\infty}(\X,\Y)\geq \ugw{\infty}(\X,\Y).$

		Next, we prove that $\usturm{\infty}(\X,\Y)\leq \min\{t\geq 0|\,\X_t\cong_w \Y_t\}$. Let $t> 0$ be such that $\X_t\cong_w \Y_t$ and let $\varphi:{\X}_t\rightarrow {\Y}_t$ denote such an isomorphism. Then, we define a function $u:X\sqcup Y\times X\sqcup Y\rightarrow\mathbb R_{\geq0}$ as follows:
		\begin{enumerate}
		    \item $u|_{X\times X}\coloneqq u_X$ and $u|_{Y\times Y}\coloneqq u_Y$;
		    \item for any $(x,y)\in X\times Y$, $u(x,y)\coloneqq\begin{cases}u_{Y_t}(\varphi([x]_t^X),[y]_t^Y),&\text{if }\varphi([x]_t^X)\neq[y]_t^Y\\
		t,&\text{if }\varphi([x]_t^X)=[y]_t^Y.\end{cases} $
		\item for any $(y,x)\in Y\times X$, $u(y,x)\coloneqq u(x,y)$.
		\end{enumerate}
		Then, it is easy to verify that $u\in\mathcal{D}^\mathrm{ult}(\uX,\uY)$ and that $u$ is actually an ultrametric. Let $Z\coloneqq (X\sqcup Y,u)$. By Lemma \ref{lm:winfty-finite}, we have 
		\[\usturm\infty(\X,\Y)\leq d_{\mathrm{W},\infty}^Z(\muX,\muY)=\max_{B\in V(Z)\backslash\{Z\}\text{ and }\muX(B)\neq\muY(B)}\diam{B^*}. \]
		We verify that $d_{\mathrm{W},\infty}^Z(\muX,\muY)\leq t$ in the following. It is obvious that $Z_t\cong X_t\cong Y_t$. Write $X_t=\{[x_i]_t^X\}_{i=1}^n$ and $Y_t=\{[y_i]_t^Y\}_{i=1}^n$ such that $[y_i]_t^Y=\varphi([x_i]_t^X)$ for each $i=1,\ldots,n$. Then, $[x_i]_t^{Z}=[y_i]_t^{Z}$ and $Z_t=\{[x_i]_t^{Z}|\,i=1,\ldots,n\}$. Since $\varphi$ is an isomorphism, for any $i=1,\dots,n$ we have that $\muX([x_i]_t^X)=\muY([y_i]_t^Y)$ and thus $\muX([x_i]_t^{Z})=\muY([y_i]_s^{Z})=\muY([x_i]_t^{Z})$ when $\muX$ and $\muY$ are regarded as pushforward measures under the inclusion map $X\hookrightarrow Z$ and $Y\hookrightarrow Z$, respectively. Now for any $B\in V(Z)$ (cf. \Cref{sec:explicit formulat}), if $\diam{B}\geq t$, then $B$ is the union of certain $[x_i]_t^{Z}$'s in $Z_t$ and thus $\muX(B)=\muY(B)$. If $\diam{B}< t$ and $\diam{B^*}> t$, then there exists some $x_i$ such that $B=[x_i]_s^{Z}$ and $[x_i]_s^{Z}=[x_i]_t^{Z}$ where $s\coloneqq\diam{B}$. This implies that $\muX(B)=\muY(B)$. In consequence, we have that
		$d_{\mathrm{W},\infty}^Z(\muX,\muY)\leq t $ and thus $\usturm{\infty}(\X,\Y)\leq d_{\mathrm{W},\infty}^{(X\sqcup Y,u)}(\muX,\muY)\leq t$. Therefore, $\usturm{\infty}(\X,\Y)\leq \inf\{t\geq 0|\,\X_t\cong_w \Y_t\}$.

		Finally, by invoking \Cref{thm:ugw-infty-eq}, we conclude that $\usturm{\infty}(\X,\Y)=\ugw{\infty}(\X,\Y).$

	\subsubsection{Proof of \Cref{thm:optimal A and varphi (usturm)}}\label{sec:proof of thm optimal A and varphi (usturm)}
	We prove the result via an explicit construction. By \Cref{thm:ugw infty and sturm ugw infty}, we have $s=\usturm\infty(\X,\Y)=\ugw{\infty}(\X,\Y)$. By \Cref{thm:ugw-infty-eq}, there exists an isomorpism $\varphi:\X_s\rightarrow \Y_s$. Since $s>0$, by \Cref{lm:quotient-finite}, both $\X_s$ and $\Y_s$ are finite spaces. Let $X_s=\{[x_1]_s^X,\dots,[x_n]_s^X\}$, $Y_s=\{[y_1]_s^Y,\dots,[y_n]_s^Y\}$ and assume $[y_i]_s^Y=\varphi([x_i]_s^X)$ for each $i=1,\ldots,n$. Let $A\coloneqq\{x_1,\dots,x_n\}$ and define $\phi:A\rightarrow Y$ by sending $x_i$ to $y_i$ for each $i=1,\ldots,n$. We prove that $(A,\phi)$ satisfies the conditions in the statement.

		Since $\varphi$ is an isomorphism, for any $1\leq i<j\leq n$,
		\[\uY(y_i,y_j)=u_{Y_s}([y_i]_s^Y,[y_j]_s^Y)= u_{Y_s}(\varphi([x_i]_s^X),\varphi([x_j]_s^X))=u_{X_s}([x_i]_s^X,[x_j]_s^X)=\uX(x_i,x_j).\]
		This implies that $\phi:A\rightarrow Y$ is an isometric embedding and thus $(A,\phi)\in\mathcal{A}$. 
		
		It is obvious that $(Z_A)_s$ is isometric to both $X_s$ and $Y_s$. In fact, $[x_i]_s^{Z_A}=[y_i]_s^{Z_A}$ in $Z_A$ for each $i=1,\ldots,n$ and $(Z_A)_s=\{[x_i]_s^{Z_A}|\,i=1,\ldots,n\}$. Since $\varphi$ is an isomorphism,  for any $i=1,\dots,n$ we have that $\muX([x_i]_s^X)=\muY([y_i]_s^Y)$ and thus $\muX([x_i]_s^{Z_A})=\muY([y_i]_s^{Z_A})=\muY([x_i]_s^{Z_A})$ when $\muX$ and $\muY$ are regarded as pushforward measures under the inclusion maps $X\rightarrow Z_A$ and $Y\rightarrow Z_A$, respectively. Now for any $B\in V(Z_A)$ (cf. \Cref{sec:explicit formulat}), if $\diam{B}\geq s$, then $B$ is the union of certain $[x_i]_s^{Z_A}$'s and thus $\muX(B)=\muY(B)$. If otherwise $\diam{B}< s$ and $\diam{B^*}> s$, then there exists $x_i$ such that $B=[x_i]_t^{Z_A}$ and $[x_i]_t^{Z_A}=[x_i]_s^{Z_A}$ where $t\coloneqq\diam{B}$. This implies that $\muX(B)=\muY(B)$. By Lemma \ref{lm:winfty-finite}, we have $ d_{\mathrm{W},\infty}^{Z_A}(\muX,\muY)\leq s$ and thus $ d_{\mathrm{W},\infty}^{Z_A}(\muX,\muY)=s$ since $d_{\mathrm{W},\infty}^{Z_A}(\muX,\muY)$ is an upper bound for $s=\usturm\infty(\X,\Y)$ due to \Cref{eq:ultra Sturm}.

\subsubsection{Proof of \Cref{thm:equivalence} }\label{app:proof of equivalence}
In this section, we prove \Cref{thm:equivalence} by slightly modifying the proof of Proposition 5.3 in \cite{memoli2011gromov}. 
\begin{lemma}\label{lm:ultra-coupling}
Let $(X,\uX)$ and $(Y,\uY)$ be compact ultrametric spaces and let $S\subseteq X\times Y$ be non-empty. Assume that $\sup_{(x,y),(x',y')\in S}\Lambda_\infty(u_X(x,x'),u_Y(y,y'))\leq\eta.$ Define $u_S:X\sqcup Y\times X\sqcup Y\rightarrow\Rp$ as follows:
\begin{enumerate}
    \item $u_S|_{X\times X}\coloneqq u_X$ and $u_S|_{Y\times Y}\coloneqq u_Y$;
    \item for any $(x,y)\in X\times Y$, $u_S(x,y)\coloneqq\inf_{(x',y')\in S}\max\left(u_X(x,x'),u_Y(y,y'),\eta\right). $
    \item for any $(x,y)\in X\times Y$, $u_S(y,x)\coloneqq u_S(x,y)$.
\end{enumerate}
Then, $u_S\in\mathcal{D}^\mathrm{ult}(\uX,\uY)$ and $u_S(x,y)\leq \eta$ for all $(x,y)\in S$.
	\end{lemma}
	\begin{proof}
	    That $u_S\in\mathcal{D}^\mathrm{ult}(\uX,\uY)$ essentially follows by \citet[Lemma 1.1]{zarichnyi2005gromov}. It remains to prove the second half of the statement. For $(x,y)\in S$, we set $(x',y')\coloneqq(x,y)$. This yields
	    \[u_S(x,y)\leq\max(u_X(x,x'),u_Y(y,y'),\eta)=\max(0,0,\eta)=\eta.\]
	   
	\end{proof}

\begin{proof}[Proof of \Cref{thm:equivalence}]
Let $\mu\in\mathcal{C}(\muX,\muY)$ be a coupling such that $\norm{\Gamma_{X,Y}^\infty}_{L^p(\mu\otimes\mu)}<\delta^5$. Set $\eps\coloneqq4v_\delta(X)\leq 4$. 
    
By \citet[Claim 10.1]{memoli2011gromov}, there exist a positive integer $N\leq [1/\delta]$ and points $x_1,\ldots,x_N$ in $X$ such that $\min_{i\neq j}u_X(x_i,x_j)\geq \frac{\eps}{2}$, $\min_i\muX\lc B_\eps^X(x_i)\rc >\delta$ and $\muX\lc\bigcup_{i=1}^NB_\eps^X(x_i)\rc\geq 1-\eps$.

\begin{claim}\label{claim 1}
For every $i=1,\ldots,N$ there exists $y_i\in Y$ such that
\[\mu\left(B_\eps^X(x_i)\times B_{2(\eps+\delta)}^Y(y_i)\right)\geq (1-\delta^2)\muX\lc B_\eps^X(x_i)\rc. \]
\end{claim}
\begin{proof}
Assume the claim is false for some $i$ and let $Q_i(y)=B_\eps^X(x_i)\times\lc Y\setminus B_{2(\eps+\delta)}^Y(y)\rc$. Then, as $\mu\in\mathcal{C}(\muX,\muY)$ it holds
\begin{align*}
    \muX\lc B_\eps^X(x_i)\rc =&\mu\lc B_\eps^X(x_i)\times Y\rc  \\
    =&\mu\lc B_\eps^X(x_i)\times B_{2(\eps+\delta)}^Y(y)\rc+\mu\lc B_\eps^X(x_i)\times \lc Y\setminus B_{2(\eps+\delta)}^Y(y)\rc\rc.
\end{align*}
Consequently, we have that $\mu(Q_i(y))\geq \delta^2\muX\lc B_\eps^X(x_i)\rc $. Further, let
\[\mathcal{Q}_i\coloneqq \left\{(x,y,x',y')\in X\times Y\times X\times Y\,|\,x,x'\in B_\eps^X(x_i)
\text{ and } \uY(y,y')\geq 2(\eps+\delta)\right\}.\]
Clearly, it holds for $(x,y,x',y')\in \mathcal{Q}_i$ that 
\[\Gamma_{X,Y}^\infty(x,y,x',y')=\Lambda_\infty\lc u_X(x,x'),u_Y(y,y')\rc=u_Y(y,y')\geq 2\delta.\]
Further, we have that $\mu\otimes\mu (\mathcal{Q}_i)\geq \delta^4$. Indeed, it holds
\begin{align*}
    \mu\otimes\mu(\mathcal{Q}_i)=&\int_{B_\eps^X(x_i)\times Y}\!\int_{Q_i(y)}\!1\,\mu(dx'\times dy')\,\mu(dx\times dy)\\
    =&\int_{B_\eps^X(x_i)\times Y}\!\mu(Q_i(y))\mu(dx\times dy)\\
    =&\muX\lc B_\eps^X(x_i)\rc \int_Y\!\mu(Q_i(y))\muY(dy)\\
    \geq& \left(\muX\lc B_\eps^X(x_i)\rc \right)^2\delta^2\\
    \geq &\delta^4.
\end{align*}
However, this yields that 
\[\norm{\Gamma^\infty_{X,Y}}_{L^p(\mu\otimes\mu)}\geq \norm{\Gamma^\infty_{X,Y}}_{L^1(\mu\otimes\mu)}\geq \norm{\Gamma^\infty_{X,Y}\mathds{1}_{\mathcal{Q}_i}}_{L^1(\mu\otimes\mu)}\geq 2\delta \cdot\mu\otimes\mu(\mathcal{Q}_i)\geq 2\delta^5,\]
which contradicts $\norm{\Gamma_{X,Y}^\infty}_{L^p(\mu\otimes\mu)}<\delta^5$.
    \end{proof}
Define for each $i=1,\ldots,N$
\[S_i\coloneqq B_\eps^X(x_i)\times B_{2(\eps+\delta)}^Y(y_i). \]
Then, by Claim \ref{claim 1}, $\mu(S_i)\geq \delta(1-\delta^2)$, for all $i=1,\ldots,N$.
\setcounter{claimcount}{1}

\begin{claim}\label{claim 2}
$\Gamma_{X,Y}^\infty(x_i,y_i,x_j,y_j)\leq 6(\eps+\delta)$ for all $i,j=1,\ldots,N$.
\end{claim}
\begin{proof}
Assume the claim fails for some $(i_0,j_0)$, i.e.,
 \[\Lambda_\infty(u_X(x_{i_0},x_{j_0}),u_Y(y_{i_0},y_{j_0}))>6(\eps+\delta)>0. \]
Then, we have $\Lambda_\infty(u_X(x_{i_0},x_{j_0}),u_Y(y_{i_0},y_{j_0}))=\max(u_X(x_{i_0},x_{j_0}),u_Y(y_{i_0},y_{j_0}))$. We assume without loss of generality that 
\[u_X(x_{i_0},x_{j_0})=\Lambda_\infty(u_X(x_{i_0},x_{j_0}),u_Y(y_{i_0},y_{j_0}))> u_Y(y_{i_0},y_{j_0}).\] 
Consider any $(x,y)\in S_{i_0}$ and $(x',y')\in S_{j_0}$. By the strong triangle inequality and the fact that $u_X(x_{i_0},x_{j_0})>6(\eps+\delta)>\eps$, it is easy to verify that $u_X(x,x')=u_X(x_{i_0},x_{j_0})$. Moreover,
\begin{align*}
u_Y(y,y')&\leq\max(u_Y(y,y_{i_0}),u_Y(y_{i_0},y_{j_0}),u_Y(y_{j_0},y'))\\
&< \max(2(\eps+\delta), u_X(x_{i_0},x_{j_0}),2(\eps+\delta))=u_X(x_{i_0},x_{j_0})=u_X(x,x').
\end{align*}
Therefore, 
\[\Gamma_{X,Y}^\infty(x,y,x',y')=u_X(x,x')=u_X(x_{i_0},x_{j_0})= \Gamma_{X,Y}^\infty(x_{i_0},y_{i_0},x_{j_0},y_{j_0})>6(\eps+\delta)>2\delta.\] Consequently, we have that
\begin{align*}
\norm{\Gamma_{X,Y}^\infty}_{L^p(\mu\otimes\mu)}&\geq \norm{\Gamma_{X,Y}^\infty}_{L^1(\mu\otimes\mu)} \geq \norm{\Gamma_{X,Y}^\infty\mathds{1}_{S_{i_0}}\mathds{1}_{S_{j_0}}}_{L^1(\mu\otimes\mu)}\geq 2\delta\mu(S_{i_0})\mu(S_{j_0})\\
>2\delta\left(\delta(1-\delta^2)\right)^2.
\end{align*}
However, for $\delta\leq 1/2$, $2\delta\left(\delta(1-\delta^2)\right)^2\geq 2\delta^5$. This leads to a contradiction. 
\end{proof}
    
    Consider $S\subseteq X\times Y$ given by $S\coloneqq\{(x_i,y_i)|\,i=1,\ldots,N\}$. Let $u_S$ be the ultrametric on $X\sqcup Y$ given by Lemma \ref{lm:ultra-coupling}. By Claim \ref{claim 2}, $\sup_{(x,y),(x',y')\in S}\Gamma_{X,Y}^\infty(x,y,x',y')\leq6(\eps+\delta)$. Then, for all $i=1,\ldots,N$ we have that $u_S(x_i,y_i)\leq 6(\eps+\delta)$ and for any $(x,y)\in X\times Y$ we have that
    \[u_S(x,y)\leq \max(\diam X,\diam Y,6(\eps+\delta)) \leq \max(\diam X,\diam Y,27)=:M'.\]
    Here in the second inequality we use the assumption that $\delta<\frac{1}{2}$ and the fact that $\eps=4v_\delta(X)\leq 4$.
    
   \setcounter{claimcount}{2}
\begin{claim}\label{claim 3}
Fix $i\in\{1,\dots,N\}$. Then, for all $(x,y)\in S_i$, it holds $u_S(x,y)\leq 6(\eps+\delta)$.
\end{claim}
\begin{proof}
 Let $(x,y)\in S_i$. Then, $\uX(x,x_i)\leq\eps$ and $\uY(y,y_i)\leq 2(\eps+\delta)$. Then, by the strong triangle inequality for $u_S$ we obtain
 \begin{align*}
    u_S(x,y)\leq& \max\{\uX(x,x_i),\uY(y,y_i), u_S(x_i,y_i)\}\\
    \leq & \max\{\eps,2(\eps+\delta),6(\eps+\delta)\}\leq 6(\eps+\delta).
\end{align*}
\end{proof}

    Let $L\coloneqq\bigcup_{i=1}^NS_i$. The next step is to estimate the mass of $\mu$ in the complement of $L$.
    
   \setcounter{claimcount}{3}
   \begin{claim}
    $\mu(X\times Y\backslash L)\leq\eps+\delta$.
   \end{claim}
   \begin{proof}
       For each $i=1,\ldots,N$, let $A_i\coloneqq B_\eps^X(x_i)\times \lc Y\setminus B_{2(\eps+\delta)}^Y(y_i)\rc$. Then,
       \[A_i=\lc B_\eps^X(x_i)\times Y\rc \setminus\lc B_\eps^X(x_i)\times B_{2(\eps+\delta)}^Y(y_i)\rc=\lc B_\eps^X(x_i)\times Y\rc \setminus S_i.\]
       Hence, 
       \[\mu(A_i)=\mu\lc B_\eps^X(x_i)\times Y\rc -\mu(S_i)=\muX\lc B_\eps^X(x_i)\rc-\mu(S_i),\]
       where the last equality follows from the fact that $\mu\in \mathcal{M}(\muX,\muY)$. By Claim \ref{claim 1}, we have that $\mu(S_i)\geq \muX\lc B_\eps^X(x_i)\rc (1-\delta^2)$. Consequently, we obtain
       \[\mu(A_i)\leq \muX\lc B_\eps^X(x_i)\rc \delta^2.\]
       Notice that
       \[X\times Y\setminus L\subseteq \left(X\mathbin{\Big\backslash} \bigcup_{i=1}^NB_\eps^X(x_i)\right)\times Y\cup \left(\bigcup_{i=1}^N A_i\right). \]
       Hence, 
       \begin{align*}
           \mu(X\times Y\setminus L)&\leq \muX\left(X\mathbin{\Big\backslash} \bigcup_{i=1}^NB_\eps^X(x_i)\right)+\sum_{i=1}^N\mu(A_i)\\
           &\leq 1-\muX\left(\bigcup_{i=1}^NB_\eps^X(x_i)\right)+\sum_{i=1}^N\delta^2\muX\lc B_\eps^X(x_i)\rc\\
           &\leq \eps+N\cdot\delta^2\leq\eps+\delta.
       \end{align*}
       Here, the third inequality follows from the construction of $x_i$s in the beginning of this section and from the fact that $N\leq [1/\delta]$.
   \end{proof}

Now, 
    \begin{align*}
        \int_{X\times Y}u_S^p(x,y)\,\mu(dx\times dy)&=\left(\int_L+\int_{X\times Y\backslash L}\right) u_S^p(x,y)\,\mu(dx\times dy)\\
        &\leq (6(\eps+\delta))^p+{M'}^p\cdot(\eps+\delta).
    \end{align*}
   Since we have for any $a,b\geq 0$ and $p\geq 1$ that $a^{1/p}+b^{1/p}\geq (a+b)^{1/p}$, we obtain 
   \begin{align*}
       \usturm p(\X,\Y)&\leq (\eps+\delta)^\frac{1}{p}\left(6(\eps+\delta)^{1-\frac{1}{p}}+M'\right)\leq(\eps+\delta)^\frac{1}{p}\left(27+M'\right)\\
       &\leq \left(4v_\delta(\X)+\delta\right)^\frac{1}{p}\cdot M,\\
   \end{align*}
   where we used $\eps=4v_\delta(\X)$ and $M\coloneqq 2\max(\diam X,\diam Y)+54\geq M'+27$. Since the roles of $\X$ and $\Y$ are symmetric, we have that
    \[\usturm p(\X,\Y)\leq\left(4\min(v_\delta(\X),v_\delta(Y))+\delta\right)^\frac{1}{p}\cdot M. \]
    This concludes the proof.
	\end{proof}
	
	\subsection{Proofs from Section \ref{sec:topology and geodesic properties}}
	The subsequent section contains the full proofs of the statements in \Cref{sec:topology and geodesic properties}.
	\subsubsection{Proof of \Cref{thm: complete and separable}}\label{sec:proof of thm complete and separable}
	\begin{enumerate}
	    \item We first prove that $(\mathcal{U}^w,\ugw p)$ is non-separable for each $p\in[1,\infty]$. Recall notations in \Cref{ex:notation two point space} and consider the family $\{\hat{\Delta}_2(a)\}_{a\in[1,2]}$.
	    
	    \begin{claim}\label{claim:example compute}
	    $\forall a\neq b\in[1,2],$
	     $ \ugw p\left(\hat{\Delta}_2(a),\hat{\Delta}_2(b)\right)=2^{-\frac{1}{p}}\Lambda_\infty(a,b)\geq 2^{-\frac{1}{p}}$, where we let $2^{-\frac{1}{\infty}}=1$.
	    \end{claim}
    \begin{proof}[Proof of Claim 1 ]
First note by \Cref{thm:comparison with original} that \[\ugw p \left(\hat{\Delta}_2(a),\hat{\Delta}_2(b)\right)\geq \uSLB p\left(\hat{\Delta}_2(a),\hat{\Delta}_2(b)\right).\] It is easy to verify that $\uSLB p\left(\hat{\Delta}_2(a),\hat{\Delta}_2(b)\right)=2^{-\frac{1}{p}}\Lambda_\infty(a,b)$. On the other hand, consider the diagonal coupling between $\mu_a$ and $\mu_b$, then for $p\in[1,\infty)$
\[\ugw p \left(\hat{\Delta}_2(a),\hat{\Delta}_2(b)\right)\leq \left(2\cdot\Lambda_\infty(a,b)^p\cdot\frac{1}{2}\cdot\frac{1}{2}\right)^\frac{1}{p}=2^{-\frac{1}{p}}\Lambda_\infty(a,b), \]
and for $p=\infty$
\[\ugw \infty \left(\hat{\Delta}_2(a),\hat{\Delta}_2(b)\right)\leq \Lambda_\infty(a,b). \]
Therefore, 
\[\ugw p \left(\hat{\Delta}_2(a),\hat{\Delta}_2(b)\right)=2^{-\frac{1}{p}}\Lambda_\infty(a,b). \]
%%%%%%%%%%%%%%%%%%%%%%%%%%%
\end{proof}
By Claim \ref{claim:example compute}, we have that $\left\{\hat{\Delta}_2(a)\right\}_{a\in[1,2]}$ is an uncountable subset of $\mathcal{U}^w$ with pairwise distance greater than $2^{-\frac{1}{p}}$, which implies that $(\mathcal{U}^w,\ugw p)$ is non-separable.
	
	Now for $p\in[1,\infty)$, we show that $\ugw p$ is not complete. Consider the family $\{\Delta_{2^n}(1)\}_{n\in\mathbb{N}}$ of $2^n$-point spaces with unitary interpoint distances. Endow each space $\Delta_{2^n}(1)$ with the uniform measure $\mu_n$ and denote the corresponding ultrametric measure space by $\hat{\Delta}_{2^n}(1)$. It is proven in \cite[Example 2.2]{sturm2012space} that $\{\hat{\Delta}_{2^n}(1)\}_{n\in\N}$ is a Cauchy sequence with respect to $\dgw p$ without a compact metric measure space as limit. It is not hard to check that 
	\[\ugw p\left(\hat{\Delta}_{2^m}(1),\hat{\Delta}_{2^n}(1)\right)=2\dgw p\left(\hat{\Delta}_{2^m}(1),\hat{\Delta}_{2^n}(1)\right),\quad\forall n,m\in\mathbb{N}. \]
	Therefore, $\{\hat{\Delta}_{2^n}(1)\}_{n\in\N}$ is a Cauchy sequence with respect to $\ugw p$ without limit in $\mathcal{U}^w$. This implies that $(\mathcal{U}^w,\ugw p)$ is not complete.

\item By \Cref{thm:ugw<ugw-sturm} and (1), we have that $\left(\mathcal{U}^w,\usturm p\right)$ is not separable. As for completeness, consider the subset $X\coloneqq\{1-\frac{1}{n}\}_{n\in\mathbb N}\subseteq (\mathbb R_{\geq 0},\Lambda_\infty)$. By Lemma \ref{lm:compact of R}, $X$ is not a compact ultrametric space. Let $\mu_0\in\mathcal{P}(X)$ be a probability defined as follows:
\[\mu_0\left(\left\{1-\frac{1}{n}\right\}\right)\coloneqq 2^{-n},\quad\forall n\in\mathbb N. \]
For each $N\in\mathbb N$, let $X_N\coloneqq \{1-\frac{1}{n}|\,n=1,\ldots,N\}$. Since each $X_N$ is finite, $(X_N,\Lambda_\infty)$ is a compact ultrametric space. Let $\mu_N\in\mathcal{P}(X_N)$ be a probability defined as follows:
\[\mu_N\left(\left\{1-\frac{1}{n}\right\}\right)\coloneqq\begin{cases} 2^{-n},& 1\leq n<N\\
2^{-N+1}&n=N
\end{cases}.\]
Then, it is easy to verify (e.g. via \Cref{thm:compact usturm A Phi representation}) that $\{(X_N,\Lambda_\infty,\mu_N)\}_{N\in\N}$ is a $\usturm p$ Cauchy sequence with $(X,\Lambda_\infty,\mu_0)$ being the limit. Since the set $X$ is not compact, $(X,\Lambda_\infty,\mu_0)\notin\mathcal{U}^w$ and thus $\left(\mathcal{U}^w,\usturm p\right)$ is not complete.

	\item That $(\mathcal{U}^w,\ugw\infty)$ is non-separable is already proved in (1).
	Given a Cauchy sequence $\{\X_n=(X_n,u_n,\mu_n)\}_{n\in\N}$ with respect to $\ugw\infty$, we have that the underlying ultrametric spaces $\{X_n\}_{n\in\N}$ form a Cauchy sequence with respect to $\ugh$ due to Corollary \ref{coro:ugw>ugh}. Since $(\mathcal{U},\ugh)$ is complete (see \cite[Proposition 2.1]{zarichnyi2005gromov}), there exists a compact ultrametric space $(X,u_X)$ such that 
	\[\lim_{n\rightarrow\infty}\ugh(X_n,X)=0. \]
	For each $n\in\mathbb{N}$, let $\delta_n\coloneqq\ugh(X_n,X)$. By \Cref{thm:ultrametric GH-distance}, we have that $(X_n)_{\delta_n}\cong X_{\delta_n}$. Denote by $\hat{\mu}_n\in\mathcal{P}(X_{\delta_n})$ the pushforward of $(\mu_n)_{\delta_n}$ under the isometry. Furthermore, we have by \Cref{lm:quotient-finite} that $X_{\delta_n}$ is finite and we let $X_{\delta_n}=\{[x_1]_{\delta_n},\ldots,[x_k]_{\delta_n}\}$ for $x_1,\ldots,x_k\in X$. Based on this, we define \[\nu_n\coloneqq\sum_{i=1}^k\hat{\mu}_n([x_i]_{\delta_n})\cdot\delta_{x_i}\in\mathcal{P}(X), \]
	where $\delta_{x_i}$ is the Dirac measure at $x_i$. Since $X$ is compact, $\mathcal{P}(X)$ is weakly compact. Therefore, the sequence $\{\nu_n\}_{n\in\N}$ has a cluster point $\nu\in\mathcal{P}(X)$.

	Now we show that $\X\coloneqq(X,u_X,\nu)$ is a $\ugw\infty$ cluster point of $\{\X_{n}\}_{n\in\mathbb{N}}$ and thus the limit of $\{\X_n\}_{n\in\mathbb{N}}$ since $\{\X_n\}_{n\in\mathbb{N}}$ is a Cauchy sequence. Without loss of generality, we assume that $\{\nu_n\}_{n\in\mathbb{N}}$ weakly converges to $\nu$. Fix any $\eps>0$, we need to show that $\ugw\infty(\X,\X_n)\leq\eps$ when $n$ is large enough. For any fixed $x_*\in X$, $[x_*]_{\eps}$ is both an open and closed ball in $X$. Therefore, $\nu([x_*]_{\eps})=\lim_{n\rightarrow\infty}\nu_n([x_*]_{\eps})$ (see e.g. \citet[Thm. 2.1]{billingsleyConvergenceProbabilityMeasures2013}). Since $\delta_n\rightarrow0$ as $n\rightarrow\infty$, there exists $N_1>0$ such that for any $n>N_1$, $\delta_n<\eps$. We specify an isometry $\varphi_n:(X_n)_{\delta_n}\rightarrow X_{\delta_n}$ that gives rise to the construction of $\nu_n$. Then, we let $\psi_n:(X_n)_\eps\rightarrow X_\eps$ be the isometry such that the following diagram commutes:
	\[\begin{tikzcd}
(X_n)_{\delta_n} \arrow{r}{\varphi_n} \arrow[swap]{d}{\eps\text{-quotient}} & X_{\delta_n} \arrow{d}{\eps\text{-quotient}} \\%
(X_n)_\eps \arrow{r}{\psi_n}& X_\eps
\end{tikzcd} \]
Assume that $[x_*]_\eps^X=\bigcup_{i=1}^l[x_i]_{\delta_n}^X$. Let $x_*^n\in X_n$ be such that $\psi_n([x_*^n]_\eps^{X_n})=[x_*]_\eps^X$ and let $x_1^n,\ldots,x_l^n\in X_n$ be such that $\varphi_n([x_i^n]_{\delta_n}^{X_n})=[x_i]_{\delta_n}^X$ for each $i=1,\ldots,l$. Then, $[x^n_*]_\eps^{X_n}=\bigcup_{i=1}^l[x_i^n]_{\delta_n}^{X_n}$. Therefore,
\begin{align*}
    \nu_n([x_*]_\eps^X)=\sum_{i=1}^l\nu_n([x_i]_{\delta_n}^X)=\sum_{i=1}^l\hat{\mu}_n([x_i]_{\delta_n}^X)=\sum_{i=1}^l{\mu}_n([x_i^n]_{\delta_n}^{X_n})=\mu_n([x^n_*]_\eps^{X_n}).
\end{align*}
Since $\X_n$ is a Cauchy sequence, there exists $N_2>0$ such that $\ugw\infty(\X_n,\X_m)<\eps$ when $n,m>N_2$. Then, by \Cref{thm:ugw-infty-eq}, $(\X_n)_\eps\cong_w(\X_m)_\eps$ for all $n,m>N_2$. By \Cref{lm:quotient-finite}, $(X_n)_\eps$ is finite, then $(X_n)_\eps$ has cardinality independent of $n$ when $n>N_2$. For all $n>N_2$, we define the finite set $A_n\coloneqq\left\{\mu_n([x^n]_\eps^{X_n})|\,x^n\in X_n\right\}$. $A_n$ is independent of $n$ since $(\X_n)_\eps\cong_w(\X_m)_\eps$ for all $n,m>N_2$. This implies that $\mu_n([x^n_*]_\eps^{X_n})$ only takes value in a finite set $A_n$. Combining with the fact that $\lim_{n\rightarrow\infty}\mu_n([x^n_*]_\eps^{X_n})=\lim_{n\rightarrow\infty}\nu_n([x]_\eps^X)=\nu([x_*]_\eps^X)$ exists, there exists $N_3>0$ such that when $n>N_3$, $\mu_n([x^n_*]_\eps)\equiv C$ for some constant $C$. This implies that
\[\nu([x_*]_\eps^X)=\mu_n([x^n_*]_\eps^{X_n}),\quad\text{ when } n>\max(N_1,N_2,N_3). \]
Since $X_\eps$ is finite, there exists a common $N>0$ such that for all $n>N$ and $\forall [x_*]_\eps\in X_\eps$ we have 
\[\nu([x_*]_\eps^X)=\mu_n([x^n_*]_\eps^{X_n}), \] 
where $[x^n_*]^{X_n}_\eps=\psi^{-1}_n([x_*]_\eps^X)\in (X_n)_\eps$. This indicates that $\nu_\eps=(\psi_n)_\#(\mu_n)_\eps$ when $n>N$. Therefore, $\X_\eps\cong_w (\X_n)_\eps$ and thus $\ugw\infty(\X,\X_n)\leq\eps$. 
	\end{enumerate}

	\subsubsection{Proof of \Cref{prop:usturm 1 geodesic}}\label{sec:proof of prop usturm 1 geodesic}
	Next, we will demonstrate \Cref{prop:usturm 1 geodesic}. However, before we come to this we recall some facts about $p$-metric and $p$-geodesic spaces.
	\begin{lemma}[{\citet[Proposition 7.10]{memoli2019gromov}}]\label{lm:p metric not geodesic}
    Given $p\in[1,\infty)$, if $X$ is a $p$-metric space, then $X$ is not $q$-geodesic for all $1\leq q<p$. 
\end{lemma}

\begin{lemma}[{\citet[Theorem 7.7]{memoli2019gromov}}]\label{lm:snow p geod}
    Let $X$ be a geodesic metric space. Then, for any $p\geq 1$, $S_\frac{1}{p}(X)$ is $p$-geodesic, where $S_\alpha$ denotes the snowflake transform for $\alpha>0$ (cf. \Cref{subsec:relation between ugw and usturm}).
\end{lemma}

For $p=1$, the proof is based on the following property of the $1$-Wasserstein space.
	\begin{lemma}[{\citet[Theorem 5.1]{bottou2018geometrical}}]\label{lm:W1 is geodesic}
	    Let $X$ be a compact metric space. Then, the space $W_1(X)\coloneqq(\mathcal{P}(X),d_{\mathrm{W},1}^X)$ is a geodesic space.
	\end{lemma}
Based on the above results and \Cref{coro:usturm_optimal}, the proof of \Cref{prop:usturm 1 geodesic} is straightforward. 
	\begin{proof}[Proof of \Cref{prop:usturm 1 geodesic}]
	 Let $\X$ and $\Y$ be two compact ultrametric measure spaces. First, we consider the case $p=1$. By \Cref{coro:usturm_optimal}, there exist a compact ultrametric space $Z$ and isometric embeddings $\phi:X\hookrightarrow Z$ and $\psi:Y\hookrightarrow Z$ such that
	\[\usturm{p}(\X,\Y)=d_{\mathrm{W},p}^Z(\phi_\#\muX,\psi_\#\muY).\]
	  The space $W_1(Z)$ is geodesic (cf. Lemma \ref{lm:W1 is geodesic}). Therefore, there exists a Wasserstein geodesic $\tilde{\gamma}:[0,1]\rightarrow W_1(Z)$ connecting $\phi_\#\mu_X$ and $\psi_\#\mu_Y$. This induces a curve $\gamma:[0,1]\rightarrow \mathcal{U}^w$ where for each $t\in[0,1]$, $\gamma(t)\coloneqq(\mathrm{supp}(\tilde{\gamma}(t)),u|_{\mathrm{supp}(\tilde{\gamma}(t))\times \mathrm{supp}(\tilde{\gamma}(t))},\tilde{\gamma}(t))$. Note that $\gamma(0)\cong_w\X$ and $\gamma(1)\cong_w\Y$ and hence we simply replace $\gamma(0)$ and $\gamma(1)$ with $\X$ and $\Y$, respectively. Now, for each $s,t\in[0,1]$, we have that
	    \[\dsturm 1(\gamma(s),\gamma(t))\leq d_{\mathrm{W},1}^{Z}(\tilde{\gamma}(s),\tilde{\gamma}(t))=|s-t|d_{\mathrm{W},1}^{Z}(\tilde{\gamma}(0),\tilde{\gamma}(1))=|s-t|\dsturm 1(\X,\Y). \]
	    Therefore, $\gamma$ is a geodesic connecting $\X$ and $\Y$ and thus $(\mathcal{U}^w,\usturm 1)$ is geodesic.

	    Next, we come to the case $p>1$. By \Cref{coro:iso-snow}, $S_p\left(\mathcal{U}^w,\usturm p\right)\cong \left(\mathcal{U}^w,\usturm 1\right)$. This implies that $S_\frac{1}{p}(\mathcal{U}^w,\usturm 1)\cong \left(\mathcal{U}^w,\usturm p\right)$. Hence, by \Cref{lm:snow p geod}, we have that $\left(\mathcal{U}^w,\usturm p\right)$ is $p$-geodesic.
	\end{proof}
\subsection{Technical issues from Section \ref{sec:ultrametric GW distance}}	
In the following, we address various technical issues from \Cref{sec:ultrametric GW distance}.

	\subsubsection{The Wasserstein pseudometric}\label{sec:Wasserstein pseudometric} 
 	Given a set $X$, a pseudometric is a symmetric function $d_X:X\times X\rightarrow\mathbb R_{\geq 0}$ satisfying the triangle inequality and $d_X(x,x)=0$ for all $x\in X$. Note that if moreover $d_X(x,y)=0$ implies $x=y$, then $d_X$ is a metric. There is a canonical identification on pseudometric spaces $(X,d_X)$: $x\sim x'$ if $d_X(x,x')=0$. Then, $\sim$ is in fact an equivalence relation and we define the quotient space $\tilde{X}=X/\sim$. Define a function $\tilde{d}_X:\tilde{X}\times \tilde{X}\rightarrow\mathbb R_{\geq 0}$ as follows:
 	\[\tilde{d}_X([x],[x'])\coloneqq\begin{cases}
 	d_X(x,x')&\text{if }d_X(x,x')\neq 0\\
 	0&\text{otherwise}
	\end{cases}.\]
	$\tilde{d}_X$ turns out to be a metric on $\tilde{X}$.
 	In the following, the metric space $(\tilde{X},\tilde{d}_X)$ is referred to as the \emph{metric space induced by the pseudometric space} $(X,d_X)$. Note that $\tilde{d}_X$ preserves the induced topology (see e.g. \cite{howes2012modern}) and thus the quotient map $\Psi:X\rightarrow \tilde{X}$ is continuous.
 	
 Analogously to the Wasserstein distance, which is defined for probability measures on metric spaces, we define the \emph{Wasserstein pseudometric} for measures on compact pseudometric spaces as done in \cite{thorsley2008model}. Let $\alpha,\beta\in\mathcal{P}(X)$. Then, we define for $p\in[1,\infty)$ the Wasserstein pseudometric of order $p$ as
 \begin{equation}\label{eq:def Wasserstein pseudpmetric p}
     \pseudoWasser{p}^{(X,d_X)}(\alpha,\beta)\coloneqq\left(\inf_{\mu\in\mathcal{C}(\alpha,\beta)}\int_{X\times X} d^p_X(x,y)\,\mu(dx\times dy)\right)^\frac{1}{p}
 \end{equation}
 and for $p=\infty$ as
 \begin{equation}\label{eq:def Wasserstein pseudpmetric infinity}
     \pseudoWasser{\infty}^{(X,d_X)}(\alpha,\beta)\coloneqq\inf_{\mu\in\mathcal{C}(\alpha,\beta)}\sup_{(x,y)\in\mathrm{supp}(\mu)}u(x,y).
 \end{equation}
 It is easy to see that the Wasserstein pseudometric  is closely related to the Wasserstein distance on the induced metric space. More precisely, one can show the following.
\begin{lemma}
Let $(X,d_X)$ denote a compact pseudometric space, let $\alpha,\beta\in\mathcal{P}(X)$. Then, it follows for  $p\in[1,\infty]$ that
\begin{equation}\label{eq:Wasserstein pseudo metric vs induced Wasserstein}
\pseudoWasser{p}^{(X,d_X)}(\alpha,\beta)=d_{\mathrm{W},p}^{(\tilde{X},\tilde{d}_X)}(\Psi_\#{\alpha},\Psi_\#{\beta}) \end{equation}
and in particular that the infimum in \Cref{eq:def Wasserstein pseudpmetric p} (resp. in \Cref{eq:def Wasserstein pseudpmetric infinity} if $p=\infty$) is attained for some $\mu\in\mathcal{C}(\alpha,\beta)$.
\end{lemma}
\begin{proof}
In the course of this proof we focus on the case $p<\infty$ and remark that the case $p=\infty$ follows by similar arguments. The quotient map allows us to define the map $\theta:\mathcal{C}(\alpha,\beta)\to\mathcal{C}(\Psi_\#{\alpha},\Psi_\#{\beta})$ via $\mu\mapsto(\Psi\times\Psi)_\#\mu$. It is easy to see that $\theta$ is well defined and surjective. Furthermore, it holds by construction that
\[\int_{X\times X}{d}^p_X(x,y)\,{\mu}(dx\times dy)=\int_{\tilde{X}\times \tilde{X}} \tilde{d}^p_X(x,y)\,\theta({\mu})(dx\times dy)\]
for all $\mu\in\mathcal{C}(\alpha,\beta)$. Hence, \Cref{eq:Wasserstein pseudo metric vs induced Wasserstein} follows.

We come to the second part of the claim. By \cite[Sec.4]{villani2008optimal} there exists an optimal coupling $\tilde{\mu}^*\in \mathcal{C}(\Psi_\#{\alpha},\Psi_\#{\beta})$ such that
\[d_{\mathrm{W},p}^{(\tilde{X},\tilde{d}_X)}(\Psi_\#{\alpha},\Psi_\#{\beta})=\left(\int_{\tilde{X}\times \tilde{X}} \tilde{d}^p_X(x,y)\,\tilde{\mu}^*(dx\times dy)\right)^\frac{1}{p}.\]
In consequence, we find using our previous results that for any $\mu^*\in \theta^{-1}(\tilde{\mu}^*)$ it holds
\begin{align*}
    d_{\mathrm{W},p}^{(\tilde{X},\tilde{d}_X)}(\Psi_\#{\alpha},\Psi_\#{\beta})=&\left(\int_{\tilde{X}\times \tilde{X}} \tilde{d}^p_X(x,y)\,\tilde{\mu}^*(dx\times dy)\right)^\frac{1}{p}\\=&\left(\int_{{X}\times {X}} {d}^p_X(x,y)\,\mu^*(dx\times dy)\right)^\frac{1}{p}=\pseudoWasser{p}^{(X,d_X)}(\alpha,\beta).\end{align*}
This yields the claim.
\end{proof}

\subsubsection{Regularity of the cost functionals of \texorpdfstring{$\ugw{p}$}{the ultrametric Gromov-Wasserstein distance} and \texorpdfstring{$\usturm{p}$}{Sturm's ultrametric Gromov-Wasserstein distance}} In the remainder of this section, we collect various technical results required to demonstrate the existence of optimizers in the definitions of $\usturm{p}$ (see \Cref{eq:ultra Sturm}) and $\ugw{p}$ (see \Cref{eq:def uGW}). 
\begin{lemma}\label{lm:measure coupling compact}
    Let $\X=\ummspaceX$ and $\Y=\ummspaceY$ be {compact} ultrametric measure spaces. Then, $\mu\in\mathcal{C}(\muX,\muY)\subseteq \mathcal{P}(X\times Y,\max(u_X,u_Y))$ is compact with respect to weak convergence.
\end{lemma}
\begin{proof}
    The proof follows directly from {\citet[Lemma 2.2]{chowdhury2019gromov}}.
\end{proof}
\begin{lemma}\label{lm:pre-compact-metric coupling}
    Let $\X,\Y\in\mathcal{U}^w$. Let $D_1\subseteq\mathcal{D}^\mathrm{ult}(\uX,\uY)$ be a non-empty subset satisfying the following: there exist $(x_0,y_0)\in X\times Y$ and $C>0$ such that $u(x_0,y_0)\leq C$ for all $u\in D_1$. Then, $D_1$ is pre-compact with respect to uniform convergence.
\end{lemma}
\begin{proof}
    Let $\{u_n\}_{n\in\mathbb N}\subseteq D_1$ be a sequence. Note that $X\times Y\subseteq X\sqcup Y\times X\sqcup Y$. Let $v_n\coloneqq u_n|_{X\times Y}$. For any $n\in\mathbb N$ and any $(x,y),(x',y')\in X\times Y$, we have that
    \begin{align*}
        |u_n(x,y)-u_n(x',y')|\leq u_X(x,x')+u_Y(y,y')\leq 2\max\left(u_X,u_Y\right)\left((x,y),(x',y')\right).
    \end{align*}
    This means that $\{v_n\}_{n\in\mathbb N}$ is equicontinuous with respect to the ultrametric $\max\{\uX,\uY\}$ on $X\times Y$. Now, since $u_n(x_0,y_0)\leq C$, we have that for any $(x,y)\in X\times Y$,
    \[ u_n(x,y)\leq 2\max\left(u_X,u_Y\right)\left((x,y),(x_0,y_0)\right)+u_n(x_0,y_0)\leq 2\max(\diam X,\diam Y)+C.\]
    Consequently, $\{v_n\}_{n\in\mathbb N}$ is uniformly bounded. By the Arz\'ela-Ascoli theorem (\cite[Theorem 7 on page 61]{kolmogorov1957elements}), we have that each subsequence of $\{v_n\}_{n\in\mathbb N}$ has a uniformly convergent subsequence. Hence, we can assume without loss of generality that the sequence $\{v_n\}_{n\in\mathbb N}$ converges to $v:X\times Y\rightarrow\mathbb R_{\geq 0}$.
    
    Now, we define $u:X\sqcup Y\times X\sqcup Y\rightarrow\mathbb R_{\geq 0}$ as follows:
    \begin{enumerate}
        \item $u|_{X\times X}\coloneqq u_X$ and $u|_{Y\times Y}\coloneqq u_Y$;
        \item $u|_{X\times Y}\coloneqq v$;
        \item for $(y,x)\in Y\times X$, we let $u(y,x)\coloneqq u(x,y)$.
    \end{enumerate}
    It is easy to verify that $u\in \mathcal{D}^\mathrm{ult}(u_X,u_Y)$ and that $u$ is a cluster point of the sequence $\{u_n\}_{n\in\mathbb N}$. Therefore, $D_1$ is pre-compact.
    \end{proof}
    
    \begin{lemma}\label{lm:metric coupling bounded integral}
        Let $\X=\ummspaceX$ and $\Y=\ummspaceY$ be {compact} ultrametric measure spaces. Let $\{\mu_n\}_{n\in\mathbb N}\subseteq\mathcal{C}(\muX,\muY)$ be a sequence weakly converging to $\mu\in \mathcal{C}(\muX,\muY)$. Let $\{u_n\}_{n\in\mathbb N}\subseteq\mathcal{D}^\mathrm{ult}(\uX,\uY)$. Suppose that there exist a non-decreasing sequence $\{p_n\}_{n\in\mathbb N}\subseteq[1,\infty)$ and $C>0$ such that
        \[ \left(\int_{X\times Y}(u_n(x,y))^{p_n}\mu_n(dx\times dy)\right)^\frac{1}{p_n}\leq C\]
        for all $n\in\mathbb N$. Then, $\{u_n\}_{n\in\mathbb N}$ uniformly converges to some $u\in \mathcal{D}^\mathrm{ult}(\uX,\uY)$ (up to taking a subsequence).
    \end{lemma}
    \begin{proof}
    The following argument adapts the proof of Lemma 3.3 in \cite{sturm2006geometry} to the current setting. 
        For any $(x_0,y_0)\in \supp\mu$, there exist $\eps,\delta>0$ and $N\in\N$ such that for all $n\geq N$
		\begin{align*}
		    C&\geq \left(\int_{X\times Y}(u_n(x,y))^{p_n}\mu_n(dx\times dy)\right)^\frac{1}{p_n}\geq \int_{X\times Y}u_n(x,y)\mu_n(dx\times dy)\\
		    &\geq \int_{B_\eps^X(x_0)\times B_\eps^Y(y_0)}u_n(x,y)\mu_n(dx\times dy)\geq \int_{B_\eps^X(x_0)\times B_\eps^Y(y_0)}(u_n(x_0,y_0)-2\eps)\mu_n(dx\times dy)\\
		    &\geq (u_n(x_0,y_0)-2\eps)\left(\mu\lc B_\eps^X(x_0)\times B_\eps^Y(y_0)\rc-\delta\right).
		\end{align*}        
		Therefore, $\{u_n(x_0,y_0)\}_{n\geq N}$ is uniformly bounded. By \Cref{lm:pre-compact-metric coupling}, we have that $\{u_n\}_{n\in\mathbb N}$ has a uniformly convergent subsequence.
    \end{proof}

	\begin{lemma}\label{lm:continuity-delta-infty}
    Let $X,Y$ be ultrametric spaces, then $\Lambda_\infty(\uX,\uY):X\times Y\times X\times Y\rightarrow\Rp$ is continuous with respect to the product topology (induced by $\max(\uX,\uY, \uX,\uY)$).
\end{lemma}
\begin{proof}
    Fix $(x,y,x',y')\in X\times Y\times X\times Y$ and $\eps>0$. Choose $0<\delta<\eps$ such that $\delta<u_X(x,x')$ if $x\neq x'$ and $\delta<u_Y(y,y')$ if $y\neq y'$. Then, consider any point $(x_1,y_1,x_1',y_1')\in X\times Y\times X\times Y$ such that $\uX(x,x_1),\uY(y,y_1),\uX(x',x_1'),\uY(y',y_1')\leq\delta$. For $\uX(x_1,x_1')$, we have the following two situations:
    \begin{enumerate}
        \item $x=x'$: $\uX(x_1,x_1')\leq\max(\uX(x_1,x),\uX(x,x_1'))\leq\delta<\eps$;
        \item $x\neq x'$: $\uX(x_1,x_1')\leq\max(\uX(x_1,x),\uX(x,x'),\uX(x',x_1'))=\uX(x,x')$. Similarly, $\uX(x,x')\leq\uX(x_1,x_1')$ and thus $\uX(x,x')=\uX(x_1,x_1')$.
    \end{enumerate}
    Similar result holds for $\uY(y_1,y_1')$. This leads to four cases for $\Lambda_\infty(\uX(x_1,x_1'),\uY(y_1,y_1'))$:
    \begin{enumerate}
        \item $x=x',y=y'$: In this case we have $\uX(x_1,x_1'),\uY(y_1,y_1')< \eps$. Then, \begin{align*}
            |\Lambda_\infty(\uX(x_1,x_1'),\uY(y_1,y_1'))-\Lambda_\infty(\uX(x,x'),\uY(y,y'))|=\Lambda_\infty(\uX(x_1,x_1'),\uY(y_1,y_1'))\\\leq \eps; \end{align*}
        \item $x=x',y\neq y'$: Now $\uX(x_1,x_1')<\eps$ and $\uY(y_1,y_1')=\uY(y,y')$. If $\uY(y,y')\geq\eps>\uX(x_1,x_1')$, then
        \[|\Lambda_\infty(\uX(x_1,x_1'),\uY(y_1,y_1'))-\Lambda_\infty(\uX(x,x'),\uY(y,y'))|=|\uY(y,y')-\uY(y,y')|=0.\]
        Otherwise $\uY(y,y')<\eps$, which implies that $\Lambda_\infty(\uX(x_1,x_1'),\uY(y_1,y_1'))\leq \eps$ and $\Lambda_\infty(\uX(x,x'),\uY(y,y'))=\uY(y,y')\leq\eps$. Therefore, 
        \[|\Lambda_\infty(\uX(x_1,x_1'),\uY(y_1,y_1'))-\Lambda_\infty(\uX(x,x'),\uY(y,y'))|\leq \eps;\]
        \item $x\neq x',y=y'$: Similar with (2) we have 
        \[|\Lambda_\infty(\uX(x_1,x_1'),\uY(y_1,y_1'))-\Lambda_\infty(\uX(x,x'),\uY(y,y'))|\leq \eps;\]
        \item $x\neq x',y\neq y'$: Now $\uX(x_1,x_1')=\uX(x,x')$ and $\uY(y_1,y_1')=\uY(y,y')$. Therefore, 
        \[|\Lambda_\infty(\uX(x_1,x_1'),\uY(y_1,y_1'))-\Lambda_\infty(\uX(x,x'),\uY(y,y'))|=0.\]
    \end{enumerate}
    In conclusion, whenever $\uX(x,x_1),\uY(y,y_1),\uX(x',x_1'),\uY(y',y_1')\leq\delta$ we have that 
    \[|\Lambda_\infty(\uX(x_1,x_1'),\uY(y_1,y_1'))-\Lambda_\infty(\uX(x,x'),\uY(y,y'))|\leq \eps.\]
    Therefore, $\Lambda_\infty(\uX,\uY)$ is continuous with respect to the metric $\max(\uX,\uY, \uX,\uY)$.
\end{proof}
\subsubsection{\texorpdfstring{$\ugw{p}$}{the ultrametric Gromov-Wasserstein distance} and the one point space}\label{sec:ugw and one point space}
It is possible to explicitly write down $\ugw{p}$, $1\leq p\leq \infty$, in some simple settings. In the following, we derive an explicit formulation of $\ugw{p}$, $1\leq p\leq \infty$, between an arbitrary ultrametric measure space $\X$ and the one point ultrametric measure space $\ast$. For this purpose, we need to introduce some notation. Let $\X=(\X,\dX,\muX)$ be a ultrametric measure space. Let its $p$-diameter (see e.g.,  \cite{memoli2011gromov}) for $1\leq p<\infty$ be defined as 
\[\mathrm{diam}_p(\X)\coloneqq \left(\iint_{X\times X} \big(\dX(x,x')\big)^p\muX(dx)\,\muX(dx')\right)^{1/p}\]
	and for $p=\infty$ as \[\mathrm{diam}_\infty(\X)\coloneqq\sup_{(x,x')\in\supp{\muX}}\dX(x,x').\]	Then, one can show the subsequent proposition.
	\begin{proposition}\label{prop:ugw and one point space}
		Let  $\ast \in \mathcal{U}^w$ be the one-point space. Then, it holds for any $1\leq p\leq \infty$ that 
		\[\ugw{p}(\X,\ast) = \mathrm{diam}_p(\X).\]
	\end{proposition}
	\begin{proof}
		Denote by $\mu$ the unique coupling $\mu_X\otimes \delta_\ast$ between $\muX$ and $\delta_\ast$. Then, for any $p<\infty$ we have
		\begin{align*}
		\ugw{p}(\X,\ast)&=\left(\iint_{X\times \ast \times X\times \ast}\big(\Lambda_\infty(\uX(x,x'),u_\ast(y,y'))\big)^p\,\mu(dx\times dy)\,\mu(dx'\times dy')\right)^{1/p}\\
		&=\left(\iint_{X\times X} \big(\uX(x,x')\big)^p\muX(dx)\,\muX(dx')\right)^{1/p}=\mathrm{diam}_p(\X).
		\end{align*}
		The case $p=\infty$ follows by analogous arguments.
	\end{proof}
%%%%%%%%%%%%%%%%%%%%%%%%%%%%%

%%%%%%%%%%%%%%%%%%%%%%%%%%%%%
\section{Missing details from Section \ref{sec:lower bounds}}
\subsection{Proofs from Section \ref{sec:lower bounds}}\label{sec:uslb<utlb}
In the following, we state the full proofs of the results from \Cref{sec:lower bounds}.
\subsubsection{Proof of \Cref{thm:comparison with original}}\label{sec:proof of thm comparison with original}
    	We start by proving the first statement. To this end, we observe that for any point $x$ in an ultrametric space $X$, there always exists a point $x'\in X$ such that $\uX(x,x')=\diam{X}$ (see \cite{dordovskyi2011diameter}). By assumption $\mu_X$ is fully supported on $X$. Hence, $s_{X,\infty}\equiv\diam{X}$ is a constant function. Therefore, 
		\[\Lambda_\infty(s_{X,\infty}(x),s_{Y,\infty}(y))\equiv\Lambda_\infty(\diam{X},\diam{Y}),\quad\forall x\in X,y\in Y. \]
		This implies that $\uFLB{\infty}(\X,\Y)=\Lambda_\infty(\diam{X},\diam{Y}).$
		By Corollary 5.8 of \citet{memoli2019gromov} and Corollary \ref{coro:ugw>ugh}, we have that
		\[\ugw{\infty}(\X,\Y)\geq \ugh(X,Y)\geq\Lambda_\infty(\diam X,\diam Y)=\uFLB{\infty}(\X,\Y). \]
		
It remains to prove the second statement. The proof for $d_{\mathrm{GW},p}(\X,\Y)\geq \mathbf{TLB}_p(\X,\Y)$ in \cite[Sec. 6]{memoli2011gromov} can be used essentially without any change for showing $\ugw{p}(\X,\Y)\geq \uTLB{p}(\X,\Y)$. Hence, it only remains to show that $\uTLB{p}(\X,\Y)\geq\uSLB{p}(\X,\Y)$, i.e., the claim follows once we have established \Cref{prop:uTLB>uSLB}. 
\begin{proposition}\label{prop:uTLB>uSLB}
	Let $\X,\Y\in\mathcal{U}^w$ and let $p\in[1,\infty]$. Then,
\[ \uTLB{p}(\X,\Y)\geq\uSLB{p}(\X,\Y).\]
\end{proposition}
In order to prove \Cref{prop:uTLB>uSLB}, we need the following technical lemma.
\begin{lemma}\label{lemma:spectrum of X}
Let $\X=\mmspaceX\in \mathcal{U}^w$. Then, $\spec{X}\coloneqq \{\uX(x,x')\,|\, x,x'\in\X\}$ is a compact subset of $(\Rp,\Lambda_\infty)$.
\end{lemma}
\begin{proof}
By \Cref{lm:quotient-finite}, we have that for each $t>0$, $X_t$ is a finite set. Let $\{t_n\}_{n=1}^\infty$ be a positive sequence decreasing to $0$. Then, it is easy to see that 
\[\spec{X}=\bigcup_{n=1}^\infty\spec{X_{t_n}}.\]
Since each $\spec{X_{t_n}}$ is a finite set, $\spec{X}$ is a countable set.

Now, pick any $0\neq t\in \spec{X}$. Suppose $t$ is a cluster point in $\spec X$. Then, there exists infinitely many $s\in\spec{X}$ greater than $\frac{t}{2}$. However, this will result in $X_\frac{t}{2}$ being an infinite set, which contradicts the fact that $X_\frac{t}{2}$ is finite. Therefore, $0$ is the only possible cluster point of $\spec X$. By \Cref{lm:compact of R}, we have that $\spec X$ is compact.
\end{proof}
With the above auxiliary result available, we can demonstrate \Cref{prop:uTLB>uSLB} and hence finish the proof of \Cref{thm:comparison with original}.

\begin{proof}[Proof of \Cref{prop:uTLB>uSLB}]
We first prove the case when $p<\infty$. Let $dh_\X(x)\coloneqq\uX(x,\cdot)_\#\muX$ and let $dh_\Y(y)\coloneqq\uY(y,\cdot)_\#\muY$. Futher, define $dH_\X\coloneqq(\uX)_\#(\muX\otimes\muX)$ and $dH_\Y\coloneqq(\uY)_\#(\muY\otimes\muY)$. \Cref{lemma:spectrum of X} implies that the set $S\coloneqq\spec{X}\cup\spec{Y}$ is a compact subset of $(\Rp,\Lambda_\infty)$. It is easy to see that $\mathrm{supp}(dh_\X),\mathrm{supp}(dh_\Y),\mathrm{supp}(dH_\X),\mathrm{supp}(dH_\Y)\subseteq S\subseteq \Rp$. Now, recall that by \Cref{prop:flb-slb-w-form}
 \[\uSLB{p}(\X,\Y)=\WasserRpS{p}\left(dH_\X,dH_\Y\right)\]
 and
  \[\uTLB{p}(\X,\Y)=\left(\inf_{\pi\in\mathcal{C}(\muX,\muY)}\int_{X\times Y}\!\left(\WasserRpS{p}(dh_\X(x),dh_\Y(y))\right)^p\,\mu(dx\times dy)\right)^{1/p}.\]

Further, we observe for any $x\in X$ and $y\in Y$ that
\begin{align*}\WasserRpS{p}(dh_\X(x),dh_\Y(y))=\inf_{\pi_{xy}\in \mathcal{C}(dh_\X(x),dh_\Y(y))}\lc\int_{S\times S}\!\Lambda_\infty^p(s,t)\,\pi_{xy}(ds\times dt)\rc^\frac{1}{p}.\end{align*}
For the remainder of this proof, the metric on metric on $S\subseteq \Rp$ is always given by $\Lambda_\infty$. Additionally,  $\mathcal{P}(S)$ denotes the set of probability measures on $S$ and we equip $\mathcal{P}(S)$ with the Borel $\sigma$-field with respect to the topology induced by weak convergence. 

\begin{claim}\label{claim:measurability}
 There is a measurable choice $(x,y)\mapsto \pi^*_{xy}$ such that for each $(x,y)\in X\times Y$, $\pi^*_{x,y}$ is an optimal transport plan between $dh_\X(x)$ and $dh_\Y(y)$.
 \end{claim}
 
 \begin{proof}[Proof of Claim \ref{claim:measurability}]
 
 It is easy to see that both $\Lambda_1$ and $\Lambda_\infty$ induce the same topology and thus Borel sets on $S$. This therefore implies that $d_{\mathrm{W},p}^{(\Rp,\Lambda_1)}$ and $d_{\mathrm{W},p}^{(\Rp,\Lambda_\infty)}$ metrize the same weak topology on $\mathcal{P}(S)$. By \citet[Remark 2.5]{memoli2021distance}, the following two maps are continuous with respect to the weak topology and thus measurable:
 \[\Phi_1:X\to\mathcal{P}(S),~x\mapsto dh_\X(x)\]
 and 
  \[\Phi_2:Y\to\mathcal{P}(S),~y\mapsto dh_\Y(y).\]

 Since $S$ is a compact space, the space $\left(\mathcal{P}(S),d_{\mathrm{W},p}^{(S,\Lambda_\infty)}\right)$ is separable \citep[Theorem 6.18]{villani2008optimal}. This yields that $\mathscr{B}\left( \mathcal{P}(S)\times \mathcal{P}(S)\right)=\mathscr{B}\left( \mathcal{P}(S)\right)\otimes\mathscr{B}\left( \mathcal{P}(S)\right)$ \citep[Proposition 1.5]{folland1999real}. Hence, the product $\Phi$ of $\Phi_1$ and $\Phi_2$, defined by
 \[\Phi:X\times Y\to \mathcal{P}(S)\times \mathcal{P}(S),~(x,y)\mapsto(dh_\X(x),dh_\Y(y))\]
 is measurable \citep[Proposition 2.4]{folland1999real}. Since $\Phi$ is measurable, a direct application of \citet[Corollary 5.22]{villani2008optimal} gives the claim.
 \end{proof}
Now, we have that for every $\mu\in \mathcal{C}(\muX,\muY)$ that
 \begin{align*}
     &\int_{X\times Y}\!\left(\WasserRpS{p}(dh_\X(x),dh_\Y(y))\right)^p\,\mu(dx\times dy)\\
     =&\int_{X\times Y}\!\int_{S\times S}\Lambda_\infty^p(s,t)\,\pi^*_{xy}(ds\times dt)\,\mu(dx\times dy)\\
     =&\int_{S\times S}\!\Lambda^p_\infty(s,t)\,\widebar{\mu}(ds\times dt),
 \end{align*}
 by Fubini's Theorem, where $\widebar{\mu}\in \mathcal{P}(S\times S)$ is defined as
 \begin{equation}\label{eq:definition of mu bar}
     \widebar{\mu}(A)\coloneqq\int_{X\times Y}\!\pi^*_{xy}(A)\,\mu(dx\times dy) \end{equation}
 for measurable $A\subseteq S\times S$. We remark that by Claim 1 the measure $\widebar{\mu}$ in \Cref{eq:definition of mu bar} is well defined. Next, we verify that $\widebar{\mu}\in \mathcal{C}(dH_\X,dH_\Y)$. For any measurable $A\subseteq (S,\Lambda_\infty)$ we have
 \begin{align*}
     \widebar{\mu}(A\times S)=&\int_{X\times Y}\!\pi^*_{x,y}(A\times S)\,\mu(dx\times dy)\\
     = &\int_{X\times Y}\!dh_\X(x)(A)\,\mu(dx\times dy)\\
     = &\int_{X}\!dh_\X(x)(A)\,\muX(dx)\\
    \overset{(i)}{=}&\int_X\!\int_X\!\mathds{1}_{\{\dX(x,x')\in A\}}\,\muX(dx')\,\muX(dx)\\
     = &dH_\X(A),
 \end{align*}
 where we have applied the marginal constraints for $\pi_{xy}$ and $\mu$. Further, $(i)$ follows by the change-of-variables formula. The analogous arguments give that 
 \[\widebar{\mu}(S\times B)=dH_\Y(B),\]
 for any measurable $B\subseteq S$. Thus, we conclude that for every $\mu \in \mathcal{C}(\muX,\muY)$
 \begin{align*}
     \int_{X\times Y}\!\left(\WasserRpS{p}(dh_\X(x),dh_\Y(y))\right)^p\,\mu(dx\times dy)=&\int_{S\times S}\!\Lambda^p_\infty(s,t)\,\widebar{\mu}(ds\times dt)\\\geq&
     \inf_{\pi\in \mathcal{C}(dH_\X,dH_\Y)}\int_{S\times S}\!\Lambda_\infty(s,t)\,\pi(ds\times dt)\\
     =&\left(\WasserRpS{p}(dH_\X,dH_\Y)\right)^p.
 \end{align*}
 This gives the claim for $p<\infty$.
 
 Next, we prove the assertion for the case $p=\infty$. Note that for any $p<\infty$
 \begin{align}
     \uTLB{p}(\X,\Y)&=\inf_{\mu\in\mathcal{C}(\muX,\muY)}\norm{\WasserRpS{p}(dh_\X(\cdot),dh_\Y(\cdot))}_{L^p(\mu)}\\
     &\leq \inf_{\mu\in\mathcal{C}(\muX,\muY)}\norm{\WasserRpS{\infty}(dh_\X(\cdot),dh_\Y(\cdot))}_{L^\infty(\mu)}\\
     &=\uTLB{\infty}(\X,\Y),
 \end{align}
 where the inequality holds since $\WasserRpS{p}\leq \WasserRpS{\infty}$ and $\norm{\cdot}_{L^p(\mu)}\leq \norm{\cdot}_{L^\infty(\mu)}$.

 By \citet[Proposition 3]{givens1984} we have that
 \[\uSLB{\infty}(\X,\Y)=\WasserRpS{\infty}\left(dH_\X,dH_\Y\right)=\lim_{p\rightarrow\infty}\WasserRpS{p}\left(dH_\X,dH_\Y\right)=\lim_{p\rightarrow\infty}\uSLB{p}(\X,\Y).\]
 Therefore,
 \[\uSLB{\infty}(\X,\Y)=\lim_{p\rightarrow\infty}\uSLB{p}(\X,\Y)\leq \limsup_{p\rightarrow \infty}\uTLB{p}(\X,\Y)\leq \uTLB{\infty}(\X,\Y).\]
\end{proof}

	\subsubsection{Proof of \Cref{prop:flb-slb-w-form}}\label{sec:proof of prop flb-slb-w-form}
		We only prove the first statement for $p\in[1,\infty)$. The case $p=\infty$ as well as the second statement can be proven in a similar manner.
		
		By directly using the change-of-variables formula, we have the following:
		\begin{align*}
		\uSLB{p}(\X,\Y)=&\!\!\inf_{\gamma\in\mathcal{C}(\muX\otimes \muX,\muY\otimes\muY)}\int_{X\times X\times Y\times Y}\!\!\!\left(\Lambda_\infty\left(\uX(x,x'),\uY(y,y')\right)\right)^p\,\gamma(d(x,x')\times d(y,y'))\\
		=&\inf_{\gamma\in\mathcal{C}(\muX\otimes \muX,\muY\otimes\muY)}\int_{\Rp\times \Rp}\!\left(\Lambda_\infty\left(s,t\right)\right)^p\,(\uX\times \uY)_\#\gamma(ds\times dt),
		\end{align*}
		where $\uX\times \uY:X\times X\times Y\times Y\rightarrow\Rp\times \Rp$ maps $(x,x',y,y')$ to $(\uX(x,x'),\uY(y,y'))$. By \Cref{lm:pushforward_coupling}, we have that 
		\[ (\uX\times \uY)_\#\mathcal{C}(\muX\otimes \muX,\muY\otimes\muY)=\mathcal{C}\left((\uX)_\#(\muX\otimes \muX),(\uY)_\#(\muY\otimes \muY)\right).\]
		
		Therefore,
		\begin{align*}
		\uSLB{p}(\X,\Y)
		=&\inf_{\gamma\in\mathcal{C}(\muX\otimes \muX,\muY\otimes\muY)}\int_{\Rp\times \Rp}\!\left(\Lambda_\infty\left(s,t\right)\right)^p\,(\uX\times \uY)_\#\gamma(ds\times dt)\\
		=&\inf_{\tilde{\gamma}\in\mathcal{C}\left((\uX)_\#(\muX\otimes \muX),(\uY)_\#(\muY\otimes \muY)\right)}\int_{\Rp\times \Rp}\!\left(\Lambda_\infty\left(s,t\right)\right)^p\,\tilde{\gamma}(ds\times dt)\\
		=& d_{\mathrm{W},p}^{(\Rp,\Lambda_\infty)}((\uX)_\#(\muX\otimes\muX),(\uY)_\#(\muY\otimes\muY)).
		\end{align*}
\subsubsection{The relation between \texorpdfstring{$\uSLB{}$ and $\uTLB{}$}{the Second and the Third ultrametric Lower Bound}}\label{sec:uSLB zero uTLB greater zero}
Next, we will demonstrate that there are ultrametric measure spaces $\X_1$ and $\X_2$ such that $\uSLB{p}(\X_1,\X_2)=0$, while it holds $\uTLB{p}(\X_1,\X_2)>0$. To this end, consider the three point space $\Delta_3(1)=(\{x_1,x_2,x_3\},u)$ where $u(x_i,x_j)=1$ whenever $i\neq j$. Let $\mu_1\coloneqq \frac{2}{3}\delta_{x_1}+\frac{1}{6}\delta_{x_2}+\frac{1}{6}\delta_{x_3}$ and let $\mu_2\coloneqq \frac{1}{3}\delta_{x_1}+\lc\frac{1}{3}-\frac{1}{2\sqrt{3}}\rc\delta_{x_2}+\lc\frac{1}{3}+\frac{1}{2\sqrt{3}}\rc\delta_{x_3}$. Both $\mu_1$ and $\mu_2$ are probability measures on $\Delta_3(1)$. We then let $\X_1\coloneqq(\Delta_3(1),\mu_1)$ and $\X_2\coloneqq(\Delta_3(1),\mu_2)$. It is easy to check that 
\[u_\#(\mu_1\otimes\mu_1)=u_\#(\mu_2\otimes\mu_2)=\frac{1}{2}\delta_0+\frac{1}{2}\delta_1.\]
Then, by \Cref{prop:flb-slb-w-form} we immediately have that $\uSLB{p}(\X_1,\X_2)=0$ for any $p\in[1,\infty]$. Now, note that
\[u(x_1,\cdot)_\#\mu_1=\frac{2}{3}\delta_0+\frac{1}{3}\delta_1,\]
which is obviously different from all $u(x_i,\cdot)_\#\mu_2$ for $i=1,2,3$. This implies (by \Cref{prop:flb-slb-w-form}) that we have $\uTLB{p}(\X_1,\X_2)>0$ for any $p\in[1,\infty]$. 

In fact, this example works as well for showing that $\dTLB{p}(\X_1,\X_2)>\dSLB{p}(\X_1,\X_2)=0$.

\section{Missing details from Section \ref{sec:ultra-dissimilarity spaces} }
\subsection{Proofs from Section \ref{sec:ultra-dissimilarity spaces}}
Next, we give the complete proofs of the results stated in \Cref{sec:ultra-dissimilarity spaces}.
\subsubsection{Proof of \Cref{thm:ugw-p-metric-dis}}\label{sec: proof of thm ugw-p-metric-dis}
 The first step to prove this is to verify the existence of an optimal coupling. To this end, we make the following obvious observation.
\begin{lemma}\label{lm:continuity-delta-infty-dis}
    Let $X,Y$ be finite ultra-dissimilarity spaces, then $\Lambda_\infty(\uX,\uY):X\times Y\times X\times Y\rightarrow\Rp$ is continuous with respect to the discrete topology.
\end{lemma}
This allows us to verify the subsequent analogue to \Cref{prop:ugw-ext-opt}. 
\begin{proposition}\label{prop:ugw-ext-opt-dis}
	Let $\X,\Y\in\ultradiscol$. Then, for any $p\in[1,\infty]$, there always exists an optimal coupling $\mu\in\mathcal{C}(\muX,\muY)$ such that $\ugw{p}(\X,\Y)=\mathrm{dis}_p^\mathrm{ult}(\mu)$.
\end{proposition}
\begin{proof}
The proof is essentially the same as the one for \Cref{prop:ugw-ext-opt}. We only replace \Cref{lm:continuity-delta-infty} with \Cref{lm:continuity-delta-infty-dis}.  The details are left to the reader.
\end{proof}
With \Cref{prop:ugw-ext-opt-dis} available and \Cref{thm:ugw-p-metric} already proven, it is immediately clear how to verify the symmetry and the $p$-triangle inequality for $\ugw{p}$ on $\ultradiscol$. Hence it only remains to demonstrate identity of indiscernibles.
\begin{proof}[Proof of \Cref{thm:ugw-p-metric-dis}]
Due to the similarity between \Cref{thm:ugw-p-metric-dis} and \Cref{thm:ugw-p-metric}, we only verify that
 $\ugw p(\X,\Y)=0$ if and only if $\X\cong_w\Y$. If $\X\cong_w\Y$, then obviously $\ugw p(\X,\Y)=0$. 

Next, we assume that $\ugw p(\X,\Y)=0$. By \Cref{prop:ugw-ext-opt-dis} there exists $\mu\in\mathcal{C}(\muX,\muY)$ such that $\ugw{p}(\X,\Y)=\mathrm{dis}_p^\mathrm{ult}(\mu)=0$. Now, we define a map $\varphi:X\rightarrow Y$ as follows: For any $x\in X$ we have $\muX(\{x\})>0$, since $\muX$ has full support and $X$ is finite. As a result, $\mu(\{(x,y)\})>0$ for some $y\in Y$, then we let $\varphi(x)\mapsto y$. This map is well-defined. Indeed, if there are $x\in X$ and $y,y'\in Y$ such that $\mu(\{(x,y)\}),\mu(\{(x,y')\})>0$, then by $\mathrm{dis}_p^\mathrm{ult}(\mu)=0$ we must have that
$$\Delta_\infty\left(u_X(x,x),u_Y(y,y')\right)= \Delta_\infty\left(u_X(x,x),u_Y(y,y)\right)=\Delta_\infty\left(u_X(x,x),u_Y(y',y')\right)=0.$$
This implies that $u_Y(y,y')=u_Y(y,y)=u_Y(y',y')=u_X(x,x)$. Since $u_Y$ is an ultra-dissimilarity, we have that $y=y'$ (cf. condition (3) in \Cref{def:ultra dissimilarity}). Essentially the same argument gives that $\varphi:X\rightarrow Y$ is an injective map. As $\mu\in\mathcal{C}(\muX,\muY)$ and $\varphi$ is injective, it follows $\muX(\{x\})=\mu(\{(x,\varphi(x))\})\leq \muY(\{\varphi(x)\})$ for any $x\in X$. Since
\[1=\sum_{x\in X}\muX(\{x\})\leq \sum_{x\in X}\muY(\{\varphi(x)\})\leq 1,\]
we have that $\muX(\{x\})=\muY(\{\varphi(x)\})$ for all $x\in X$. Since $\muY$ is fully supported, this implies that $\varphi$ is a bijective measure preserving map. Now, for any $x,x'\in X$, $\mathrm{dis}_p^\mathrm{ult}(\mu)=0$ implies that $\Delta_\infty(u_X(x,x'),u_Y(\varphi(x),\varphi(x')))=0$ and thus $u_X(x,x')=u_Y(\varphi(x),\varphi(x'))$. Therefore, $\varphi$ is also an isometry and thus an isomorphism. In consequence, $\X\cong_w\Y$. 
\end{proof}

\section{Missing details from Section \ref{sec:computational aspects}}
 \subsection{Missing details from Section \ref{sec:ugw toy examples}}\label{app:ugw toy examples}
 Here, we list the precise results for the comparisons of the spaces $\X_i$, $1\leq i\leq 4$, illustrated in \Cref{fig:umm spaces}. They are gathered in \Cref{tab:exemplary comparison} and \Cref{tab:exemplary comparison uslb}.
 
		\begin{table}[H]
		\centering
		\begin{tabular}{|c|c|c|c|c||c|c|c|c|}
			\hline\multicolumn{1}{|c|}{}&\multicolumn{4}{c||}{$\ugw{1}$\rule{0pt}{10pt}}& \multicolumn{4}{c|}{$\ugw{\infty}$\rule{0pt}{10pt}}\\
			\hline 	\rule{0pt}{10pt}  & $\X_1$ &$\X_2$& $\X_3$ & $\X_4$& $\X_1$ &$\X_2$& $\X_3$ & $\X_4$ \\\hline
			
			$\X_1$  &0.0000   & 0.9333      &0.2444 	& 0.7071  & 0.0000 &2.1000  &1.1000    &2.000 \\
			$\X_2$  & 0.9333  &  0.0000    &1.1778 	& 1.5107    & 2.1000  &0.0000  & 2.1000    &2.1000  \\
			$\X_3$  & 0.2444  &   1.1778  & 0.0000 	&0.4493   & 1.1000 &2.1000  &0.0000    &2.0000 	\\
			$\X_4$  & 0.7071  &   1.5107  &0.4493	&  0.0000    & 2.0000 &2.1000& 2.0000    &0.0000  \\\hline 
		\end{tabular}
		\medskip
		\caption{\textbf{Comparison of different ultrametric measure spaces I:} The values of  $\ugw{1}(\X_i,\X_j)$ (approximated by Algorithm 1) and $\ugw{\infty}(\X_i,\X_j)$, $1\leq i\leq j\leq 4$, where $\X_i$, $1\leq i \leq 4$, denote the ultrametric measure spaces displayed in \Cref{fig:umm spaces}.} \label{tab:exemplary comparison}
	\end{table}

	\begin{table}[H]
	\centering
		\begin{tabular}{|c|c|c|c|c|}
			\hline\multicolumn{1}{|c|}{}&\multicolumn{4}{c|}{$\uSLB{1}$\rule{0pt}{10pt}}\\
			\hline 	\rule{0pt}{10pt}  & $\X_1$ &$\X_2$& $\X_3$ & $\X_4$ \\\hline
			$\X_1$  &0.0000   &         0.9333     &0.2444 	& 0.0778 \\
			$\X_2$  &     0.9333 &  0.0000    & 1.1778 	& 1.4522  \\
			$\X_3$  &     0.2444 &   1.1778  & 0.0000 	&0.2764  	\\
			$\X_4$  & 0.0778 &   1.5107  &0.2764	&  0.0000     \\\hline 
		\end{tabular}
		\medskip
		\caption{\textbf{Comparison of different ultrametric measure spaces II:} The values of $\uSLB{1}(\X_i,\X_j)$, $1\leq i\leq j\leq 4$, where $\X_i$, $1\leq i \leq 4$, denote the ultrametric measure spaces displayed in \Cref{fig:umm spaces}.} \label{tab:exemplary comparison uslb}
	\end{table}

 \subsection{Missing details from Section \ref{subsec:relation to the GW-dist}}\label{sec:details from subsec relation to the GW-dist}
 Here, we state more results for the comparison of the ultrametric measure spaces illustrated in \Cref{fig:umm spaces} and give the precise construction of the ultrametric spaces $Z_{k,t}^i$, $2\leq k\leq 5$, $t=0,0.2,0.4,0.4$, $1\leq i\leq 15$.\\
 
  \paragraph{\textbf{The ultrametric measure spaces from \Cref{fig:umm spaces}}}
    First, we give the precise results for comparing the ultametric dissimilarity spaces in \Cref{fig:umm spaces} based on $\dgw{1}$ and $\dSLB{1}$. They are gathered in \Cref{tab:exemplary comparison GW}.
 	\begin{table}[H]
 	\centering
		\begin{tabular}{|c|c|c|c|c||c|c|c|c|}
			\hline\multicolumn{1}{|c|}{}&\multicolumn{4}{c||}{$\dgw{1}$\rule{0pt}{10pt}}& \multicolumn{4}{c|}{$\dSLB{1}$\rule{0pt}{10pt}}\\
			\hline 	\rule{0pt}{10pt}  & $\X_1$ &$\X_2$& $\X_3$ & $\X_4$& $\X_1$ &$\X_2$& $\X_3$ & $\X_4$ \\\hline
			
			$\X_1$  &0.0000   &     0.0444     &0.0222 	& 0.2111  & 0.0000 &    0.0444  &0.0222   &0.0422 \\
			$\X_2$  &     0.0444 &  0.0000    & 0.0667  	& 0.2556   &     0.0444  &0.0000  & 0.0667   &0.0867  \\
			$\X_3$  & 0.0222  &  0.0667   & 0.0000 	&0.2253   & 0.0222 &0.0667  &0.0000    &0.0573 	\\
			$\X_4$  & 0.2111  &   0.2556  &0.2253	&  0.0000    & 0.0422 &0.0867& 0.0573    &0.0000  \\\hline 
		\end{tabular}
		\medskip
		\caption{\textbf{Comparison of different ultrametric measure spaces III:} The values of  $\dgw{1}(\X_i,\X_j)$ (approximated by Algorithm 1) and $\dSLB{1}(\X_i,\X_j)$, $1\leq i\leq j\leq 4$, where $(X_i,d_{X_i},\mu_{X_i})$, $1\leq i \leq 4$, denote the ultrametric measure spaces displayed in \Cref{fig:umm spaces}.} \label{tab:exemplary comparison GW}
	\end{table}
 \paragraph{\textbf{Perturbations at level $t$}}Next, we give the precise construction of the ultrametric measure spaces $Z_{k,t}^i$, $2\leq k\leq 5$, $t=0,0.2,0.4,0.4$, $1\leq i\leq 15$. For each $k=2,3,4,5$ we first draw a sample with $100\times k$ points from the mixture distribution 
	\[\sum_{i=0}^k\frac{1}{k}U[1.5(k-1),1.5(k-1)+1],\]
	where $U[a,b]$ denotes the uniform distribution on $[a,b]$. For each sample, we employ the single linkage algorithm to create a dendrogram, which then induces an ultrametric on the given sample. We further draw a 30-point subspace from each ultrametric space and denote it by $Z_k$. These four spaces have similar diameter values between 0.5 and 0.6. Each space $Z_k$ is equipped with the uniform probability measure and the resulting ultrametric measure spaces are denoted by $\Z_{k}=\left( Z_{k},u_{Z_k},\mu_{Z_k}\right) $, $k=2,3,4,5$. We remark that $k$ can be regarded as the number of blocks in the dendrogram representation of the obtained ultrametric measure spaces (see the top row of \Cref{fig:randomly sampled ultrametric measure spaces} for a visualization of three 3-block spaces).

	Finally, we introduce our method for perturbing ultrametric spaces. Given a perturbation level $t\geq 0$ and an ultrametric space $X$, we consider the quotient space $X_t$. Each equivalence class $[x]_t\subseteq X$ is an ultrametric subspace of $X$. If $|[x]_t|>1$, we let $m\coloneqq\left| \spec{[x]_t}\right|-1$ and write $\spec{[x]_t}=\{0<s_1<\ldots<s_m\}$. Let $\delta\coloneqq\diam{[x]_t}$. We generate $m$ uniformly distributed numbers from $[0, t-\delta]$ and sort them according to ascending order to obtain $a_1\leq\ldots\leq a_m$. We then perturb $u_{X}|_{[x]_t\times [x]_t}$ by replacing $s_i$ with $s_i+a_i$ for each $i=1,\ldots,m$. We do the same for all equivalence classes $[x]_t$ and thus obtain a new ultrametric on $X$.

\end{document}